\documentclass[11pt,letterpaper]{amsart}
\usepackage{dsfont,hyperref,xcolor}
\usepackage[all]{xy}
\usepackage{amssymb}
\SelectTips{cm}{}
\usepackage{fullpage}
\allowdisplaybreaks

\newtheorem{theorem}{Theorem}[section]
\newtheorem{lemma}[theorem]{Lemma}
\newtheorem{proposition}[theorem]{Proposition}
\newtheorem{corollary}[theorem]{Corollary}
\theoremstyle{definition}
\newtheorem{definition}[theorem]{Definition}
\newtheorem{example}[theorem]{Example}

\newtheorem{conjecture}[theorem]{Conjecture}
\newtheorem{remark}[theorem]{Remark}

\newcommand{\Ext}{\text{Ext}}

\newcommand{\Tr}{\text{Tr}}

\newcommand{\Ker}{\text{Ker\,}}

\newcommand{\End}{\text{End}}
\newcommand{\Hom}{{\rm Hom}}

\newcommand{\Rep}{\text{Rep}}

\newcommand{\I}{\mathcal{I}}

\newcommand{\C}{\mathcal{C}}
\newcommand{\A}{\mathcal{A}}
\newcommand{\ot}{\otimes}

\newcommand{\mO}{{\mathcal O}}

\newcommand{\QQ}{\mathbb{Q}}
\newcommand{\ZZ}{{\mathbb{Z}}}

\newcommand{\one}{\mathds{1}}
\newcommand{\lk}{\ar@{-}}

\newcommand{\D}{\mathcal{D}}

\renewcommand{\Vec}{\mathrm{Vec}}
\newcommand{\sVec}{\mathrm{sVec}}
\renewcommand{\k}{\mathbf{k}}
\newcommand{\FPdim}{\mathop{\mathrm{FPdim}}\nolimits}
\newcommand{\ccup}{\text{\small $\bigcup$}}
\newcommand{\coplus}{\text{\small $\bigoplus$}}
\newcommand{\cotimes}{\text{\small $\bigotimes$}}

\newcommand\T{{\mathcal{T}}}
\renewcommand\P{{\mathcal{P}}}
\newcommand\E{{\mathcal{E}}}
\renewcommand\S{{\mathcal{S}}}

\newcommand{\BZ}{\Bbb Z}
\newcommand{\St}{{\rm St}}
\newcommand{\be}{\one}
\newcommand{\Coker}{{\rm Coker}}
\newcommand{\Ver}{{\rm Ver}}
\newcommand{\rvline}{\hspace*{-\arraycolsep}\vline\hspace*{-\arraycolsep}}
\numberwithin{equation}{section}
\begin{document}

\title{New incompressible symmetric tensor categories in positive characteristic}

\author{Dave Benson} 
\address{Institute of Mathematics,
                              University of Aberdeen,
                              Fraser Noble Building,
                              King's College,
                              Aberdeen AB24 3UE,
                              Scotland, UK.}
\email{d.j.benson@abdn.ac.uk}                              
\author{Pavel Etingof}
\address{Department of Mathematics, Massachusetts Institute of Technology,
Cambridge, MA 02139, USA}
\email{etingof@math.mit.edu}
\author{Victor Ostrik}
\address{Department of Mathematics,
University of Oregon, Eugene, OR 97403, USA}
\address{Laboratory of Algebraic Geometry,
National Research University Higher School of Economics, Moscow, Russia}
\email{vostrik@uoregon.edu}

\begin{abstract} We propose a method of constructing abelian envelopes 
of symmetric rigid monoidal Karoubian categories over an algebraically closed field $\k$. 
If ${\rm char}(\k)=p>0$, we use this method to construct generalizations 
$\Ver_{p^n}$, $\Ver_{p^n}^+$ of the incompressible abelian symmetric
tensor categories defined in \cite{BE} for $p=2$ and in \cite{GK,GM} for $n=1$. Namely, $\Ver_{p^n}$ is the abelian envelope of the quotient of the category of tilting modules for $SL_2(\k)$ by the $n$-th Steinberg module, and $\Ver_{p^n}^+$ is its subcategory generated by $PGL_2(\k)$-modules. We show that $\Ver_{p^n}$ are reductions 
to characteristic $p$ of Verlinde braided tensor categories in characteristic zero, which explains the notation. 
We study the structure of these categories in detail, and in particular show that 
they categorify the real cyclotomic rings $\ZZ[2\cos(2\pi/p^n)]$, and that $\Ver_{p^n}$ embeds 
into $\Ver_{p^{n+1}}$. We conjecture that every symmetric tensor category of moderate growth
over $\k$ admits a fiber functor to the union $\Ver_{p^\infty}$ of
the nested sequence $\Ver_{p}\subset \Ver_{p^2}\subset\cdots$.
This would provide an analog of Deligne's theorem in characteristic zero and a generalization of 
the results of \cite{CEO} which shows that this conjecture holds for Frobenius exact (in particular, semisimple) categories, and moreover the fiber functor lands in $\Ver_p$ (in the case of fusion categories, this was shown earlier in \cite{O}). Finally, we classify symmetric tensor categories generated by an object with invertible exterior square; this class contains the categories $\Ver_{p^n}$. 
\end{abstract}

\maketitle

\tableofcontents

\section{Introduction} 
A tensor category $\C$ is said to have {\it moderate growth} if for any object $X\in \C$
the length of $X^{\otimes n}$ grows at most exponentially with $n$. In characteristic zero, a remarkable theorem of Deligne (\cite{D1}) states that every symmetric tensor category of moderate growth is {\it super-Tannakian}, i.e., fibers over the category of supervector spaces. In other words, it is the category of representations of an affine supergroup scheme. However, in any characteristic $p>0$ this theorem fails. Namely, counterexamples in characteristic $p\ge 5$ are given in \cite{GK,GM}, in characteristic $2$ in \cite{Ven,BE}, and in characteristic $3$ for the first time in the present paper. 
In fact, in \cite{BE}, the first two authors constructed a nested sequence of finite symmetric tensor categories in characteristic two,
\[ \Vec=\C_0 \subset \C_1 \subset \C_2 \subset \cdots \]
which are {\it incompressible}, i.e., do not fiber over anything smaller. The main purpose of this
paper is to construct similar incompressible symmetric tensor categories in characteristic $p>2$. Unfortunately, the approach of \cite{BE} does not work in odd characteristic, so we propose a completely different construction. This construction also works in characteristic two, and 
recovers the categories defined in~\cite{BE},
even though there are some differences in the behavior at the prime two.

Our categories are closely connected with the non-symmetric
Verlinde tensor categories in characteristic zero which categorify the Verlinde fusion rules for $SL_2$ (\cite{EGNO}, 8.18), and because
of this connection, we name them $\Ver_{p^n}$ ($n\ge 1$). When $p>2$, the category
$\Ver_{p^n}$ decomposes as a tensor product 
$\Ver^+_{p^n}\boxtimes \sVec$ of the ``even part''  and the category
of supervector spaces. 

The underlying abelian category of $\Ver_{p^n}$ is easy to construct.
We start with the natural two dimensional module $V$ for the algebraic group $SL_2(\k)$,
and take its tensor powers. 
Each $V^{\otimes m}$ has exactly one isomorphism class of indecomposable
direct summand that
does not appear in any previous tensor power,
and this is the tilting module $T_m$. For $m<p$, $T_m$ is
isomorphic to the symmetric power $S^mV$, and for
$m=p^k-1$ it is the Steinberg module $\St_k=S^{p^k-1}V$. We form the direct sum 
$\coplus_{m=p^{n-1}-1}^{p^n-2}T_m$ and form its endomorphism
ring $A$. Then $\Ver_{p^n}$ is the category $A$-mod of finite dimensional $A$-modules.

Constructing the symmetric tensor structure on $\Ver_{p^n}$ is
more subtle. The essential point is to prove a splitting property
of tilting modules.

\begin{proposition} (Proposition \ref{splitmor})
Every morphism between tilting modules $T_m$ with \linebreak $m<p^k-1$
is split by tensoring with $\St_{k-1}$.
\end{proposition}

In order to use this fact for putting an appropriate symmetric tensor structure
on $\Ver_{p^n}$, we set up a general categorical machinery of splitting ideals and abelian envelopes, which is of independent interest and likely to have other applications. 
Namely, given a Karoubian rigid monoidal category 
$\T$, we say that $Q\in \T$ is a {\it splitting object} if tensoring with $Q$ splits every morphism in $\T$; such objects form a thick ideal $\mathcal{S}\subset \T$. We say that $\T$ is {\it separated} if a morphism in $\T$ annihilated by tensoring with every splitting object is zero. Then we prove the following result. 

\begin{theorem}\label{maintt} (Theorems \ref{main0},\ref{main1}) If $\T$ is separated and $\mathcal{S}$ is finitely generated as a left ideal, then there is a multitensor (in particular, abelian) category $\C(\T)$ together with a faithful monoidal functor $F: \T\to \C(\T)$ which identifies $\mathcal{S}$ with the ideal of projectives in $\C(\T)$. Moreover, if this functor is full then $\C(\T)$ is the {\it abelian envelope} of $\T$ (i.e., faithful monoidal functors from $\T$ to any multitensor category $\D$ correspond to tensor functors $\C(\T)\to \D$). 
\end{theorem}

We then apply this construction to the category $\T=\T_{n,p}$ of tilting modules for $SL_2(\k)$ modulo the tensor ideal generated by $\St_n$. This is enabled by Proposition \ref{splitmor}, which implies that the category $\T_{n,p}$ is separated. This yields an abelian equivalence $\C(\T_{n,p})\cong \Ver_{p^n}$, which equips $\Ver_{p^n}$ with a structure of a symmetric tensor category.   

Our main theorem is as follows.

\begin{theorem}\label{maintheorem}
Let $\k$ be an algebraically closed field of prime characteristic $p$. 
Then there are nested sequences of 
symmetric tensor categories $\Ver_{p^n}^+=\Ver_{p^n}^+(\k)\subset \Ver_{p^n}=\Ver_{p^n}(\k)$ over $\k$,
\[ \Ver^+_p \subset \Ver^+_{p^2}\subset 
\Ver^+_{p^3} \subset \cdots \]
and
\[ \Ver_p \subset \Ver_{p^2}\subset 
\Ver_{p^3} \subset \cdots \]
(Theorem \ref{inclusion}) with the following properties.\footnote{We do not define the category 
$\Ver^+_2$.}
\begin{enumerate}
\item The category $\Ver_p$ is the semisimplification of the
module category for $\ZZ/p$. For $n\ge 2$, the categories
$\Ver_{p^n},\Ver_{p^n}^+$ are not semisimple (Theorem \ref{cyclot}(i),(ii)).
\item For $p>2$, the category $\Ver_{p^n}$ decomposes as a tensor product
$\Ver^+_{p^n}\boxtimes\sVec$ (Corollary \ref{Supervec}). 
\item We have $\Ver_{2^n}=\C_{2n-2}$ and $\Ver^+_{2^n}=\C_{2n-3}$,
the symmetric tensor categories contructed in~\cite{BE} (Theorem \ref{cyclot}(iii)).
\item If $p^n>2$, the Grothendieck ring of $\Ver^+_{p^n}$ is isomorphic to the ring of integers in the real part of the cyclotomic field of $p^n$-th roots of unity,
namely $\ZZ[2\cos(2\pi/p^n)]$. In other words, 
$\Ver^+_{p^n}$ gives an abelian categorification of the ring 
$\ZZ[2\cos(2\pi/p^n)]$. The Grothendieck ring of $\Ver_{p^n}$ is 
isomorphic to the group ring of $\ZZ/2$ over $\ZZ[2\cos(2\pi/p^n)]$ if $p>2$ and to 
$\ZZ[2\cos(2\pi/2^{n+1})]$ if $p=2$ (Theorem \ref{cyclot}(iv)). 
\item The Frobenius-Perron dimension of $\Ver_{p^n}$ is $\frac{p^n}{2\sin^2(\pi/p^n)}$ (Proposition \ref{CartanFP}). 
\item The categories $\Ver_{p^n}^+$, $p>2$ and $\Ver_{2^n}$ are incompressible, i.e., any (not necessarily braided) tensor functor out of them is a fully faithful embedding (as a tensor subcategory). In particular, these categories do not admit a tensor functor to a category of smaller Frobenius-Perron dimension. Likewise, the categories $\Ver_{p^n}$  and $\Ver_{2^n}^+$ are incompressible as symmetric tensor categories (Theorem \ref{incompressible}). 
\item The category $\Ver_{p^n}$ admits a semisimple braided lift to characteristic zero, producing the semisimplification $\Ver_{p^n}(K)$ of the category of tilting modules for quantum $SL_2$ at the root of unity of order $p^n$ (Subsection \ref{lifttochar0}). 
\item (The Steinberg tensor product theorem for $\Ver_{p^n}$) There exist simple objects $\Bbb T_j^{[k]}$ of $\Ver_{p^n}$, $j=0,...,p-1$, $k=1,...,n$, and $j\ne p-1$ if $k=1$, such that any simple object of $\Ver_{p^n}$ can be uniquely written in the form 
$L_i=\Bbb T_{i_1}^{[1]}\otimes...\otimes \Bbb T_{i_n}^{[n]}$, where $0\le i=\overline{i_1...i_n}\le p^{n-1}(p-1)-1$ is an integer written in base $p$. The simple objects of $\Ver_{p^n}^+$ are the $L_i$ with even $i$ (Theorem \ref{tpt}). 
\item The category $\Ver_{p^n}$ has $p-1$ blocks of sizes $1,p-1,p(p-1),...,p^{n-2}(p-1)$, so a total of $n(p-1)$ blocks. Moreover, all blocks of the same size are equivalent, even for different $n$. For instance, the blocks of size $1$ are semisimple (i.e., generated by a simple projective object), and the blocks of size $p-1$ are equivalent to the unique non-semisimple block of representations of the symmetric group $S_p$ over $\k$ (Proposition \ref{blockstr}, Proposition \ref{blocksp2}).  
\item The Grothendieck ring of the stable category of $\Ver_{p^n}^+$ is isomorphic to 
$\Bbb F_p[z]/z^{\frac{p^{n-1}-1}{2}}$ for $p>2$ and $\Bbb F_2[z]/z^{2^{n-2}}$ for $p=2$. The Grothendieck ring of the stable category of $\Ver_{2^n}$ is $\Bbb F_p[z]/z^{2^{n-1}-1}$ (Proposition \ref{stabrings}). 
\item The determinant of the Cartan matrix of a block of size $p^r(p-1)$ 
in $\Ver_{p^n}$ is $p^{p^r}$ for $r\ge 0$, and its entries are $0$ or powers of $2$ (Proposition \ref{detcar}, Corollary \ref{powersof2}). 
\item If $p>2$ then $\Ver_{p^{n}}$ is a Serre subcategory in $\Ver_{p^{n+k}}$ (Proposition \ref{serre}).   
\end{enumerate}
\end{theorem}

Theorem \ref{maintheorem} implies that we can define tensor categories $\Ver_{p^\infty}=\ccup_{n\ge 1}\Ver_{p^n}$. Motivated by the main results of \cite{O},\cite{CEO}, we make the following (perhaps imprudently) bold conjecture: 

\begin{conjecture}\label{bold} Any symmetric tensor category of moderate growth over an algebraically closed field $\k$ of characteristic $p>0$ admits a fiber functor to $\Ver_{p^\infty}$. 
\end{conjecture} 

Namely, the main result of \cite{CEO} implies that 
Conjecture \ref{bold} holds for Frobenius exact, in particular, for semisimple categories (in the special case of fusion categories it was established earlier in \cite{O}). However, it is open even for finite tensor categories. In particular, this conjecture would imply that Deligne's theorem holds 
for {\it integral} finite symmetric tensor categories for $p>2$ (for example, for representation categories of triangular Hopf algebras); in other words, that such categories 
are representation categories of finite supergroup schemes. However, even this specialization of the conjecture is currently wide open. 

In any case, the construction of the category $\Ver_{p^\infty}$ opens the door for defining and studying many new interesting symmetric tensor categories, which may be constructed as representation categories of affine group schemes in $\Ver_{p^\infty}$. Examples of such group schemes are the general linear group $GL(X)$ for $X\in \Ver_{p^\infty}$, as well as the orthogonal group $O(X)$ and the symplectic group $Sp(X)$ when $X$ is equipped with a symmetric, respectively skew-symmetric isomorphism $X\to X^*$. In particular, we see that in characteristic $p$ there are infinitely many versions of general linear groups of a given ``rank'' $r$, parametrized by various objects $X$ of length $r$. 

The results of this paper have applications to modular representation theory which are discussed in \cite{BE}, Subsection 4.4. Namely, let $V$ be the tautological 2-dimensional $\k$-representation of the group $SL_2(\Bbb F_{p^n})$, and let $\D_{n,p}$ be the quotient of the Karoubian tensor category generated by $V$ (i.e., with indecomposables being direct summands of tensor powers of $V$) by the ideal of projectives, i.e., the corresponding stable category. Then it is shown in \cite{BE}, Proposition 4.5 that $\D_{n,p}$ is equivalent to $\T_{n,p}$. This implies that $\Ver_{p^n}$ is the abelian envelope of $\D_{n,p}$, which greatly clarifies the structure of $\D_{n,p}$.  

The organization of the paper is as follows. In Section \ref{s2} we develop the technology of splitting ideals and abelian envelopes. In Section \ref{s3} we recall and prove a number of auxiliary results about tilting modules for $SL_2(\k)$. In Section \ref{s4} we apply the technology of Sections \ref{s2}, \ref{s3} to constructing and studying the categories $\Ver_{p^n}$. In particular, we prove the properties of $\Ver_{p^n}$ listed in Theorem \ref{maintheorem}, and also compute the Cartan and decomposition matrices of their blocks, the fusion rules, and $\Ext^1$ between simple objects. At the end of Section \ref{s4} we use our results to classify symmetric tensor categories generated by an object with invertible exterior square. Finally, in Section \ref{s5} we compute explicitly the examples $\Ver_{p^2}$ and $\Ver_{3^3}$. 

\begin{remark} The categories $\Ver_{p^n}$ were also constructed simultaneously and independently by Kevin Coulembier in \cite{C3}. This is an application of his general theory of abelian envelopes alternative to the one we propose in Section \ref{s2}. 
\end{remark} 

{\bf Acknowledgements.} We are grateful to Kevin Coulembier, Thorsten Heidersdorf and Rapha\"el Rouquier for useful discussions and to anonymous referees for careful reading of the paper. This material is based on work supported by
the National Science Foundation under Grant No.~DMS-1440140 while the authors were in residence at the Mathematical Sciences Research
Institute in Berkeley, California (D.~B. in Spring 2018, P. E. in Spring 2018 and Spring 2020, V.~O. in Spring 2020). The work of P. E. was partially supported by the NSF grant DMS-1502244. The work of V. O. was partially supported by the HSE University Basic Research Program, Russian Academic Excellence Project '5-100' and by the NSF grant DMS-1702251.
The work of D. B. was partly supported by the Isaac Newton Institute for Mathematical Sciences during the Spring 2020 semester, under EPSRC
grant EP/R014604/1. Many of the results in Section 5 were obtained by a computer calculation using MAGMA (\cite{BCP}). 

\section{Splitting ideals and abelian envelopes} \label{construction}\label{s2}

In this section we propose a general construction of abelian envelopes $\C(\T)$ of a certain class of Karoubian rigid monoidal categories $\T$ which we call {\it separated and complete}. 
Under a certain finiteness assumption on $\T$, the resulting category $\C(\T)$ is a multitensor category with enough projectives. This construction is designed to build the categories $\Ver_{p^n}$ but is interesting in its own right, and we expect it to have applications beyond the main goals of this paper. Therefore we developed this theory in a greater generality than actually needed here. In particular, a number of statements in this section, specifically Lemma \ref{splile}, Propositions \ref{classif}, \ref{finprop}, Corollary \ref{fini}, Propositions \ref{sepa}, \ref{uniqmin}, Corollary \ref{semis}, Theorem \ref{main2}, are not used in the subsequent sections of this paper. 

\subsection{Conventions} 
Throughout the paper, $\k$ will denote an algebraically closed field. 
We will freely use the theory of tensor categories and refer the reader to \cite{EGNO} for the basics of this theory. In particular, we will use the conventions of this book. 

Namely, by an {\it artinian category} over $\k$ we mean a $\k$-linear abelian category 
in which objects have finite length and morphism spaces are finite dimensional, and 
call such a category {\it finite} if it has finitely many (isomorphism classes of) simple objects and enough projectives. By a  {\it (multi)tensor category} we will mean 
the notion defined in \cite{EGNO}, Definition 4.1.1; in particular, such categories are artinian (hence abelian). We will denote the category of finite dimensional vector spaces by $\Vec_\k$ or shortly by $\Vec$ and for ${\rm char}(\k)\ne 2$ the category of finite dimensional supervector spaces by $\sVec_\k$ or shortly $\sVec$. If $X$ is an object of a tensor category with finitely many simple objects, then we denote by ${\rm FPdim}(X)$ the Frobenius-Perron dimension of $X$  (\cite{EGNO}, Subsections 3.3, 4.5). We will also consider categories with tensor product which are not abelian but only Karoubian; in this case we will use the term {\it Karoubian monoidal category}. Similarly, by a {\it tensor functor} we will mean an exact monoidal functor between tensor categories, see \cite{EGNO}, Definition 4.2.5 (note that such a functor is automatically faithful, see \cite{EGNO}, Remark 4.3.10). In more general situations we will use the term {\it monoidal functor}; all monoidal functors we consider will be additive. Also, a symmetric tensor functor out of a symmetric tensor category into another, usually quite concrete one (such as $\Vec_\k$, $\sVec_\k$ or ${\rm Ver}_{p^n}$ defined below) 
will be called a {\it fiber functor}. 

Finally, we will need the notions of {\it surjective and injective tensor functors}, \cite{EGNO}, Subsection 6.3. Namely, a tensor functor $F: \C\to \D$ is surjective if every object $Y\in \D$ is a subquotient of $F(X)$ for some $X\in \C$, and 
$F$ is injective if it is fully faithful (i.e., embedding of a tensor subcategory).   

\subsection{Splitting objects and splitting ideals} Let $\k$ be an algebraically closed field. Throughout the paper, $\T$ will be a $\k$-linear rigid monoidal Karoubian category with bilinear tensor product and finite dimensional Hom spaces (in particular, it is Krull-Schmidt). 

Recall that a  {\it (thick) ideal} $\P\subset \T$ is a full Karoubian subcategory closed under tensoring on both sides with $\T$. In this case $\P$ is also closed under left and right duals:   indeed, the definition of duals implies that $Q^*$ is a direct summand in $Q^*\otimes Q\otimes Q^*$ and ${}^*Q$ is a direct summand in 
${^*}Q\otimes Q\otimes {^*}Q$.    Thus, $\P$ is determined 
by its indecomposable objects, which form a subset of indecomposable objects of $\T$.
For a thick ideal $\P\subset \T$, the {\it tensor ideal $\langle \P\rangle$ generated by $\P$} is the collection of morphisms which factor through objects of $\P$. We can form the quotient monoidal category $\T/\langle \P\rangle$, which we will sometimes also denote by $\T/\P$. 

Recall also that an ideal $\P$ is said to be {\it finitely generated} (as a left ideal) if there exists an object $P\in \P$ such that every indecomposable object $Q\in \P$ is a direct summand in $X\otimes P$ for some $X\in \T$ (note that this is a stronger condition than finite generation as a two-sided ideal, see Example \ref{exaa}(2)). In this case by the rigidity of $\T$ we may  (and will) choose $P$ so that every indecomposable object $Q\in \P$ is also a direct summand in $P\otimes Y$ for some $Y\in \T$. Such an object $P$ will be called {\it a generator} of $\P$. 
For example, if $\P$ has finitely many isomorphism classes of indecomposables, then the direct sum $P$ 
of representatives of these classes is a generator. 

Recall that a morphism $f: X\to Y$ in $\T$ is called {\it split} if it is a direct sum of a zero morphism and an isomorphism. 
In other words, $f$ is split if it is a composition $\iota\circ \pi$ where $\pi$ is a split epimorphism (projection to a direct summand) and  $\iota$ is a split monomorphism (inclusion of a direct summand). It is easy to see that $f$ is split if and only if there exists $g: Y\to X$ such that 
\begin{equation}\label{splieq}
f\circ g\circ f=f,\ g\circ f\circ g=g.
\end{equation}
In fact, for this it suffices that there exist $h: Y\to X$ such that $f\circ h\circ f=f$; then $g:=h\circ f\circ h$ satisfies \eqref{splieq}. 

The following lemma is well known but we provide a proof for reader's convenience. 

\begin{lemma}\label{dirsum} The direct sum of two morphisms is split if and only if each summand is split. 
\end{lemma}

\begin{proof} The ``if" direction is obvious, so let us prove the ``only if" direction. 
If $f_i: X_i\to Y_i$, $i=1,2$, and $f=f_1\oplus f_2$ is split then there is $h: Y_1\oplus Y_2\to X_1\oplus X_2$ satisfying $f\circ h\circ f=f$. 
Let $h_{ij}$ be the components of $h$. Then we have $f_1\circ h_{11}\circ f_1=f_1$, i.e., $f_1$ is split. 
\end{proof} 

\begin{definition} 
We will say that an ideal $\P\subset \T$ is a {\it splitting ideal} if for any $Q_1,Q_2,R\in \P$ and any morphism $f: Q_1\to Q_2$ the morphism $1_R\otimes f: R\otimes Q_1\to R\otimes Q_2$ is split.
\end{definition} 

Note that in this case it follows by taking duals that the morphism $f\otimes 1_R: Q_1\otimes R\to Q_2\otimes R$ 
is also split. 

  In other words, a splitting ideal is an ideal that 
splits its own morphisms.   

\begin{definition}\label{spliobj} We will say that an object $R\in \T$ is a {\it splitting object} if 
for any $Q_1,Q_2\in \T$ and a morphism $f: Q_1\to Q_2$ the morphisms $1_R\otimes f: R\otimes Q_1\to R\otimes Q_2$ and  $f\otimes 1_R: Q_1\otimes R\to Q_2\otimes R$ are split.
\end{definition}

In other words, a splitting object is an object that splits all morphisms in the category.

It is clear that splitting objects form a splitting ideal, which we denote by $\S=\S(\T)$.   It turns out that the converse also holds, i.e., an object of a splitting ideal splits not just all morphisms in this ideal but, in fact, all morphisms in the category.

\begin{lemma}\label{splile} Every splitting ideal consists of splitting objects. 
 \end{lemma} 
 
 \begin{proof} Let $\P\subset \T$ be a splitting ideal, $R\in \P$ and $f: Q_1\to Q_2$ be a morphism in $\T$. 
 Then $g:=1_R\otimes f: R\otimes Q_1\to R\otimes Q_2$ is a morphism in $\P$ (as $\P$ is an ideal). Thus 
 the morphism 
 $$
 1_{R\otimes R^*}\otimes g: R\otimes R^*\otimes R\otimes Q_1\to  R\otimes R^*\otimes R\otimes Q_2
 $$
 is split. But $R$ is a direct summand in $R\otimes R^*\otimes R$ since $R$ is rigid. This implies that 
 $1_{R\otimes R^*}\otimes g=g\oplus h$, so $g$ itself is split by Lemma \ref{dirsum}. Similarly, $f\otimes 1_R$ is split, as claimed. 
 \end{proof} 
 
 Lemma \ref{splile} implies that $\S$ is the unique maximal splitting ideal in $\T$, containing all other splitting ideals. 
 
 A trivial example of an ideal is $\P=0$ (it is clearly splitting and finitely generated with $P=0$). Here are some more interesting examples. 

\begin{example}\label{multiten}  Suppose that $\T$ is a multitensor category (\cite{EGNO}, Definition 4.1.1) with enough projectives. Then the ideal $\P={\rm Pr}(\T)$ of projectives in $\T$ coincides with $\S(\T)$ and is a finitely generated splitting ideal. A generator $P$ is the projective cover of $\one$. Indeed, it is clear that every projective object 
is splitting. Conversely, if $Q$ splits every morphism between projective objects 
then in particular it splits a morphism $\xi: P_1\to P_0$ whose cokernel is $\one$, 
which implies that $Q$ is projective, since $Q$ is the cokernel of the split morphism $\xi\otimes 1_Q$. 
\end{example}

\begin{example}\label{nonspli} Let $\T$ be the rigid monoidal category 
of rigid finite dimensional bimodules over a finite dimensional non-semisimple Frobenius algebra $B$ (i.e., those projective as a left and as a right module); note that the functor of double dual in this category is given by the Nakayama automorphism. Then the subcategory $\P$ of projective bimodules is a finitely generated ideal (with generator $P=B^e:=B\otimes B^{\rm op}$), but is {\it not} a splitting ideal. Indeed, there are non-split endomorphisms of the bimodule $B^e$ (for example, nonzero elements of the radical of $B^e$), and they are not split by tensoring with any rigid bimodule. In fact, in this case there are no nonzero splitting ideals at all, since for any bimodule $Q$ the bimodule $B^e\otimes_B Q\otimes_B B^e=B\otimes_\k Q\otimes_\k B$ is free. 
\end{example} 

Let us now describe all splitting ideals in $\T$. 
 
\begin{lemma}\label{spliob} (a) Let $R\in \T$ be an indecomposable splitting object. Then the following conditions on an object $Q\in \T$ are equivalent: 

(i) there exists $X\in \T$ such that $\Hom(X\otimes Q,R)\ne 0$; 

(ii) there exists $X\in \T$ such that $\Hom(R,X\otimes Q)\ne 0$; 

(iii) there exists $X\in \T$ such that $R$ is a direct summand of $X\otimes Q$. 

\noindent (b) Let $Q,R\in \T$ be indecomposable splitting objects. 
Then there exists $X\in \T$ such that $R$ is a direct summand in $X\otimes Q$ 
if and only if there exists $X\in \T$ such that $Q$ is a direct summand of $X\otimes R$. 
\end{lemma} 

\begin{proof} (a) It is clear that (iii) implies (i) and (ii).

Suppose $\Hom(X\otimes Q,R)\ne 0$. Then $\Hom(X,R\otimes Q^*)\ne 0$. Thus $R\otimes Q^*\ne 0$. 
Hence the natural map $1\otimes {\rm ev}: R\otimes Q^*\otimes Q\to R$ 
is not zero (as it corresponds to the identity morphism of $R\otimes Q^*$). 
Also this morphism is split since $R$ is a splitting object. Since $R$ is indecomposable, 
$1\otimes {\rm ev}$ is a split epimorphism. Hence $R$ is a direct summand of $Y\otimes Q$, where $Y=R\otimes Q^*$. 
Thus (i) implies (iii). 

Suppose $\Hom(R,X\otimes Q)\ne 0$. Then $\Hom(R\otimes {^*}Q,X)\ne 0$. Thus $R\otimes {^*}Q\ne 0$. 
Hence the natural map $1\otimes {\rm coev}: R\to R\otimes {^*}Q\otimes Q$ 
is not zero (as it corresponds to the identity morphism of $R\otimes {^*}Q$). 
Also this morphism is split since $R$ is a splitting object. Since $R$ is indecomposable, 
$1\otimes {\rm coev}$ is a split monomorphism. Thus $R$ is a direct summand of $Y\otimes Q$, where $Y=R\otimes {^*}Q$. 
Thus (ii) implies (iii). 

(b) Suppose $R$ is a direct summand of $X\otimes Q$. Then $\Hom(R,X\otimes Q)\ne 0$. 
It follows that $\Hom(X^*\otimes R,Q)\ne 0$. So by (a) there is $Y\in \T$ such that $Q$ is a direct summand of $Y\otimes R$. 
\end{proof} 
  
 Now consider the set $S=S(\T)$ of isomorphism classes of indecomposable splitting objects in $\T$. 
 Then all splitting objects are direct sums of objects from $S$. 
 Define a relation on $S$ by the condition 
 that $Q\sim R$ if there exist $X,Y\in \T$ 
 such that $R$ is a direct summand in $X\otimes Q\otimes Y$. 
 Clearly, this relation is reflexive and transitive, and it is also symmetric by Lemma \ref{spliob}(b), so it is an equivalence relation. Note that by Lemma \ref{spliob}(a) $Q\sim R$ iff there exist $X,Y\in \T$ such that 
 $\Hom(Q,X\otimes R\otimes Y)\ne 0$ iff there exist $X,Y\in \T$ such that 
 $\Hom(X\otimes R\otimes Y,Q)\ne 0$. Note also that $Q^*\sim Q$ since $Q$ is a direct summand in $Q\otimes Q^*\otimes Q$.  

Let $\overline{S}:=S/\sim$. Every class $c\in \overline{S}$ generates a minimal splitting ideal $\P(c)$. 
It is easy to see that if $c_1\ne c_2$ then $\P(c_1)\otimes \P(c_2)=0$ (as these ideals are disjoint).

Thus we have the following proposition. 

\begin{proposition}\label{classif} Splitting ideals 
of $\T$ are in 1-1 correspondence with subsets 
$K\subset \overline{S}$ defined by the formula 
$\P_K:=\coplus_{c\in K}\P(c)$. Among them, finitely generated splitting ideals correspond to such sums that are finite 
with $\P(c)$ finitely generated for each $c\in K$. In particular, $\S=\P_{\overline{S}}$. 
\end{proposition} 

\begin{proof} This follows from Lemma \ref{splile}. 
\end{proof} 

We can also define the equivalence relations $\sim_\ell,\sim_r$ as follows: $Q\sim_\ell R$ if there is 
$X\in \T$ such that $Q$ is a direct summand in $X\otimes R$, and $Q\sim_r R$ if there is 
$X\in \T$ such that $Q$ is a direct summand in $R\otimes X$ (the symmetry of these relations follows 
from Lemma \ref{spliob}(b)). Note that by Lemma \ref{spliob}(a) $Q\sim_\ell R$ iff there is $X\in \T$ such that $\Hom(Q,X\otimes R)\ne 0$ and $Q\sim_r R$ iff there is $X\in \T$ such that $\Hom(Q,R\otimes X)\ne 0$. The relation 
$\sim_\ell$ splits an equivalence class $c\in \overline{S}$ into the disjoint union of classes 
$c_{\bullet j}$, $j\in I_c$, and the relation $\sim_\ell$ splits $c$ into the disjoint union of classes $c_{i\bullet}$, $i\in I_c$. We can define $c_{ij}:=c_{i\bullet}\cap c_{\bullet j}$, with $c_{ij}^*=c_{ji}$, then $c$ is a disjoint union of $c_{ij}$. 
Let $\P(c)_{i\bullet},\P(c)_{\bullet j},\P(c)_{ij}$ be the subcategories of $\P$ additively generated by $c_{i\bullet}, c_{\bullet j}$, $c_{ij}$, respectively; clearly, $\P_{\bullet j}$ is a left ideal and $\P_{i\bullet}$ is a right ideal. Thus we have decompositions 
$$
\P(c)=\coplus_{i\in I_c}\P_{i\bullet}=\coplus_{j\in I_c}\P_{\bullet j}=\coplus_{i,j\in I_c}\P(c)_{ij}.
$$ 

\subsection{A finiteness property of splitting objects}

\begin{proposition}\label{finprop} 
Let $Q,R$ be indecomposable splitting objects in $\T$. If $\Hom(R,Q)\ne 0$ then $R$ is a direct summand 
in $Q\otimes Q^*\otimes Q$. 
\end{proposition} 

\begin{proof} 
{\bf Step 1.} Let $Y\in \T$ be indecomposable.   Then $\End(Y)$ is a local algebra 
(otherwise it would have a nontrivial idempotent which would split $Y$). 
Let $\mathfrak{m}_Y$ be the radical of $\End(Y)$. For each indecomposable object $X\in \T$ set $U(X,Y):=X\otimes \Hom(X,Y)$
if $X$ is not isomorphic to $Y$, and $U(Y,Y):=Y\otimes \mathfrak{m}_Y$, 
and extend this definition by additivity to all objects $X\in \T$. (Here we use that any linear category over a field ${\bf k}$ is a module over the category of vector spaces over ${\bf k}$.)
Note that $U(X,Y)$ is equipped with a canonical morphism $\xi: U(X,Y)\to Y$, namely
the one corresponding to the identity operator on $\Hom(X,Y)$ under the natural isomorphism  
$$
\Hom(X\otimes \Hom(X,Y),Y)\cong \Hom_{\bold k}(\Hom(X,Y),\Hom(X,Y)).
$$
Moreover, $\xi$ is not a split epimorphism, and any morphism $g: X\to Y$ which is not a split epimorphism factors through $\xi$.   

{\bf Step 2.} Now let $R,Q$ be indecomposable splitting objects of $\T$. 
Let 
$$
U:=U(Q\otimes Q^*\otimes Q\oplus R\otimes Q^*\otimes Q,R).
$$ 
Then the map $\xi\otimes 1_{Q^*}: U\otimes Q^*\to R\otimes Q^*$ is split. Thus we can write $U\otimes Q^*\cong N\oplus K$, $R\otimes Q^*\cong N\oplus L$,
so that $\xi\otimes 1=1_N\oplus 0_{K\to L}$.

{\bf Step 3.} Now assume that $f: R\to Q$ is a morphism. 
Then the map 
$$
f\otimes 1_{Q^*}: R\otimes Q^*\to Q\otimes Q^*
$$ 
is split. So we can write $Q\otimes Q^*\cong M\oplus C$,
and we have a split epimorphism 
$$
g: R\otimes Q^*=N\oplus L\to M.
$$ 

{\bf Step 4.} Also, the map 
$$
f\circ \xi\otimes 1_{Q^*}: U\otimes Q^*\cong N\oplus K\to Q\otimes Q^*\cong M\oplus C
$$ 
is split. It is a direct sum of a map $\eta: N\to M$ and $0_{K\to C}$, so 
the map $\eta$ is split. So we can write $N\cong W\oplus V$, $M\cong W\oplus Y$, 
so that $\eta=1_W\oplus 0_{V\to Y}$. 
Thus the map $g: N\oplus L\to M$ can be written as $g: W\oplus V\oplus L\to W\oplus Y$. 
Since $g$ is a split epimorphism and $g|_V=0$, we obtain a split epimorphism
$\bar g: W\oplus L\to W\oplus Y$, such that $\bar g|_W=1_W$. Let $h: L\to W$ be the corresponding component of $\bar g$. 
Then composing $\bar g$ with the automorphism $\left(\begin{matrix} 1 & -h\\ 0 & 1\end{matrix}\right)$ of $W\oplus L$, we 
obtain a split epimorphism $\bar g'=1_W\oplus \widetilde g$, where $\widetilde g: L\to Y$ is a split epimorphism. 

{\bf Step 5.} Now assume additionally that $f\ne 0$. 
Then we claim that 
the map $\xi\otimes 1_{Q^*}$ is not a split epimorphism, i.e., $L\ne 0$. Indeed, the morphism 
$$
\circ \xi: \Hom(\bold 1,U\otimes R^*)\cong \Hom(R,U)\to \Hom(\bold 1,R\otimes R^*)\cong \Hom(R,R)
$$ 
is not surjective, as $\xi$ is not split (so $1_R$ can't be in the image). Hence the morphism 
$\xi\otimes 1_{R^*}: U\otimes R^*\to R\otimes R^*$ is not a split epimorphism. But 
$R,Q$ belong to the same $\P(c)_{ij}$, so there is $X\in \T$ such that 
$R^*$ is a direct summand of $Q^*\otimes X$. Thus $\xi\otimes 1_{Q^*\otimes X}=\xi\otimes 1_{Q^*}\otimes 1_X$ 
is not a split epimorphism either (indeed, the direct sum of two morphisms can only be a split epimorphism 
if both summands are). Hence $\xi\otimes 1_{Q^*}$ is not a split epimorphism, as claimed. 

{\bf Step 6.} Moreover, we claim that $Y\ne 0$. 
Assume that $Y=0$. Consider the map $\theta: L\to R\otimes Q^*=W\oplus V\oplus L$ 
given by $\theta=(\theta_1,\theta_2,\theta_3)$, where $\theta_1=-h$, $\theta_2=0$, $\theta_3=1_L$. 
It is easy to see that the map $(f\otimes 1_{Q^*})\circ \theta: L\to Q\otimes Q^*$ is zero. 
Let $\zeta$ be the element of $\Hom(L\otimes Q,R)$ corresponding to $\theta$. 
The composition $f\circ \zeta\in \Hom(L\otimes Q,Q)$ is zero. Thus the map $\zeta$ is not a split epimorphism, as $f\ne 0$. 
Also $L\otimes Q$ is a direct summand in $R\otimes Q^*\otimes Q$, since $L$ is a direct summand 
in $R\otimes Q^*$. Thus by the definition of $U,\xi$, there exists an element $\widetilde\zeta\in \Hom(L\otimes Q,U)$ which maps to $\zeta$. Let $\widetilde\theta\in \Hom(L,U\otimes Q^*)$ correspond to $\widetilde\zeta$. Then 
$\widetilde\theta$ maps to $\theta$. However, the image of any element of $\Hom(L,U\otimes Q^*)$ factors through $N$, so cannot equal $\theta$ (as $L\ne 0$), which contradicts the assumption that $Y=0$. Thus, $Y\ne 0$. 

{\bf Step 7.} Since $Y$ is a direct summand in both $L$ and $Q\otimes Q^*$, this implies that $$\Hom(Q\otimes Q^*,L)\ne 0.$$ 
Thus the map 
$$
\circ(\xi\otimes 1_{Q^*}): \Hom(Q\otimes Q^*,U\otimes Q^*)\to \Hom(Q\otimes Q^*,R\otimes Q^*)
$$
is not surjective. Hence 
the map 
$$
\circ\xi: \Hom(Q\otimes Q^*\otimes Q,U)\to \Hom(Q\otimes Q^*\otimes Q,R)
$$
is not surjective either. So there exists a morphism $\pi: Q\otimes Q^*\otimes Q\to R$ that does not factor through $\xi$. 
By the definition of $U,\xi$, this implies that $\pi$ is a split epimorphism, hence $R$ is a direct summand in $Q\otimes Q^*\otimes Q$, as claimed. 
\end{proof} 

\begin{corollary}\label{fini} For every indecomposable splitting object $Q$ there are finitely many indecomposable splitting objects $R$ such that $\Hom(R,Q)\ne 0$; namely, at most the number of distinct indecomposable direct summands in 
$Q\otimes Q^*\otimes Q$. 
\end{corollary} 

\begin{proof} This follows immediately from Proposition \ref{finprop}. 
\end{proof} 

\begin{remark}\label{directpf} Suppose that in Proposition \ref{finprop}, there is a monoidal functor $F: \T\to \C$ from $\T$ to a multitensor category $\C$ with enough projectives which is fully faithful on projectives, and the objects $Q,R$ are projective in $\C$ (we will show below that this is so under mild assumptions, but the proof uses Proposition \ref{finprop}). Then there is a much more direct proof of Proposition \ref{finprop}. Namely, suppose that $R$ is the projective cover of a simple object $X\in \C$. Then, $\Hom(R,Q)\ne 0$ implies that $X$ is a composition factor in $Q$. Thus $X\otimes Q^*\ne 0$ (as $X\otimes X^*$ is a subquotient of $X\otimes Q^*$), and $X\otimes Q^*\otimes Q$ is a direct summand in $Q\otimes Q^*\otimes Q$. But we have a morphism $1\otimes {\rm ev}: X\otimes Q^*\otimes Q\to X$, which is nonzero since it corresponds to the identity endomorphism of the nonzero object $X\otimes Q^*$. Since $X\otimes Q^*\otimes Q$ is projective, it follows that it contains $R$ as a direct summand. 

In fact, the proof of Proposition \ref{finprop} is just an adaptation of this argument which does not refer to the category $\C$. 
\end{remark} 

\subsection{The abelian category attached to an ideal and its derived category}

Let \linebreak $\P\subset \T$ be an ideal. Let $\bold I$ be the set of isomorphism classes of indecomposables in $\P$ with representatives $\bold P_i$, $i\in \bold I$. Let us say that $\P$ has {\it the finiteness property} if 
for any $Q\in \P$ there are finitely many $i\in \bold I$ such that $\Hom(\bold P_i,Q)\ne 0$. 
For example, by Corollary \ref{fini}, the ideal $\mathcal{S}$ (and, in fact, any splitting ideal) 
automatically has the finiteness property. 

From now on we will assume that $\P$ has the finiteness property. Let $C:=\coplus_{i,j\in \bold I} \Hom(\bold P_i,\bold P_j)^*$. 
This space is naturally a coalgebra whose dual is the algebra $A:=C^*=\prod_{i,j\in \bold I}\Hom(\bold P_i,\bold P_j)$ with operation $ab=b\circ a$ (note that this product is well defined and does not involve infinite summations because of the finiteness property of $\P$). For example, if $\bold I$ is finite then $A=\End(P)^{\rm op}$, where $P:=\coplus_{i\in \bold I} \bold P_i$. 

Let $\C=\C(\T,\P):=C$-comod be the category of finite dimensional $C$-comodules (or, equivalently, $C^*$-modules when $\bold I$ is finite). Equivalently, $\C$ is the category of additive functors $\P^{\rm op}\to \Vec$ which are of finite length, i.e., vanish on almost all $\bold P_i$. 
Then $\C$ is an artinian category with enough projectives and simple objects $\bold L_i,i\in \bold I$. Namely, we have a natural inclusion $\iota: \P\hookrightarrow \C$ as a full subcategory given by $\iota(Q):=\coplus_{i\in \bold I}\Hom(\bold P_i,Q)$, whose image is the ideal ${\rm Pr}(\C)$ of projective objects in $\C$ 
(note that $\dim \iota(Q)<\infty$ since $\Hom(\bold P_i,Q)=0$ for almost all $i$). 
We will identify $\P$ with ${\rm Pr}(\C)$ using $\iota$ and, abusing notation, 
will often denote the image of $Q\in \bold P$ under $\iota$ also by $Q$. Then $\bold P_i$ 
are projective covers of $\bold L_i$. 

Consider now the bounded above derived category $D^-(\C)$. By a standard theorem of homological algebra (\cite{GeM}, III.5.21), we can realize $D^-(\C)$ as the bounded above homotopy category of projectives $K^-({\rm Pr}(\C))\cong K^-(\P)$. This allows us to realize the objects of $D^-(\C)$ explicitly as complexes 
$$
...\to P_{n+2}\to P_{n+1}\to P_n,
$$
where $P_j\in {\rm Pr}(\C)$. Moreover, the subcategory $\C\subset D^-(\C)$ (the heart of the tautological t-structure) is realized as the subcategory of such complexes which are exact in nonzero degrees, i.e., the subcategory of resolutions
$$
...\to P_2\to P_1\to P_0,
$$
where $P_i$ sits in cohomological degree $-i$. 

\subsection{The semigroup category structure on $\C$}

Note that $\P$ is a semigroup category (\cite{EGNO}, p.25; i.e., it has a tensor product with associativity isomorphism satisfying the pentagon axiom, but has no unit in general). This implies 

\begin{proposition}\label{p1a} The category $K^-(\P)\cong D^-(\C)$ has a natural structure of a semigroup category, given by tensoring bounded above complexes of objects from $\P$. Therefore, the category $\C$ is equipped with the induced semigroup category structure (given by taking the zeroth cohomology of the tensor product), which extends the tensor product on $\P$.  
\end{proposition} 

\begin{proof} The proof is standard. Namely, given $X,Y\in \C$, fix projective resolutions $P_\bullet$ of $X$ and $Q_\bullet$ of $Y$. We define the derived tensor product $X\otimes^L Y$ to be the object of $D^-(\C)$ represented by the complex $P_\bullet\otimes Q_\bullet$. In particular, we set $X\otimes Y:=H^0(X\otimes^L Y)\in \C$.  

Let us see why $X,Y\mapsto X\otimes^L Y$ is a well defined bifunctor. Let $P_\bullet',Q_\bullet'$ be other resolutions of $X,Y$. Then there exist morphisms of resolutions $f: P_\bullet\to P_\bullet'$, $g: Q_\bullet\to Q_\bullet'$ which extend the identity morphisms $1_X$, $1_Y$, and they are defined uniquely up to homotopy. Similarly, we have $f': P_\bullet'\to P_\bullet$, $g': Q_\bullet'\to Q_\bullet$, and $f\circ f', f'\circ f, g\circ g', g'\circ g$ are homotopic to the identity morphisms. Now consider the morphism $f\otimes g: P_\bullet\otimes Q_\bullet\to P_\bullet'\otimes Q_\bullet'$. This is well defined up to homotopy, and the inverse is $f'\otimes g'$ up to homotopy. Thus $f\otimes g$ defines an isomorphism $X\otimes^L Y\to (X\otimes^L Y)'$ of the derived tensor products defined using these two pairs of resolutions, and it is easy to show that this isomorphism is independent of the choice of $f,g$. Moreover, if we have three resolutions for $X$ and for $Y$ then it is shown in a standard way that the composition around the corresponding triangle is the identity. Thus $X\otimes^L Y$ and hence $X\otimes Y$ is well defined. 
\end{proof} 

\begin{remark}
In fact, to define $X\otimes Y$, one does not need the full resolutions of $X$ and $Y$ by objects of 
$\P$. All one needs is {\it presentations} of $X,Y$, i.e., their representations as cokernels 
of morphisms $f: P_1\to P_0$ and $g: Q_1\to Q_0$ between objects from $\P$. 
In this case we have a morphism 
$$
f\otimes 1+1\otimes g: P_1\otimes Q_0\oplus P_0\otimes Q_1\to P_0\otimes Q_0,
$$
and $X\otimes Y$ is just the cokernel of this morphism. 
We refer the reader to \cite[Section 3]{MMMT} for more details.
\end{remark} 

\begin{example} 1. If $\P=0$ then $\C=0$. 

2. In the setting of Example \ref{multiten} we have $\C=\T$, and the semigroup structure on $\C$ constructed above recovers the original monoidal 
structure on $\T$. 

3. Similarly, in Example \ref{nonspli}, $\C\supset \T$ is the category of all 
finite dimensional $B$-bimodules with tensor product over $B$, and $X\otimes^L Y$ 
is the usual derived tensor product over $B$. 
\end{example} 

\subsection{Properties of the tensor product on $\C$}

Let us now study the properties of the tensor product on $\C$ defined in Proposition \ref{p1a}. 

\begin{lemma}\label{l2a} (i) Let $Q_\bullet$ be a projective resolution of an object $Y\in \C$. 
Then for each $T\in \T$ the complexes 
$T\otimes Q_\bullet$ and $Q_\bullet\otimes T$ in $K^-(\P)$ 
are acyclic in strictly negative degrees. 

(ii) Let $X\in \P$. Then for any $Y\in \C$ we have $H^{-i}(X\otimes^L Y)=H^{-i}(Y\otimes^L X)=0$ for all $i>0$. 
\end{lemma} 

\begin{proof} (i) Let $P_1\to P_2\to P_3$ be an exact sequence of projective objects of $\C$. 
This means that the complex 
$$
\Hom(Q,P_1)\to \Hom(Q,P_2)\to \Hom(Q,P_3)
$$
is acyclic for any $Q\in \P$. Now consider the complex 
\begin{equation}\label{com1}
\Hom(Q,T\otimes P_1)\to \Hom(Q,T\otimes P_2)\to \Hom(Q,T\otimes P_3).
\end{equation} 
It can be written as 
$$
\Hom(T^*\otimes Q,P_1)\to \Hom(T^*\otimes Q,P_2)\to \Hom(T^*\otimes Q,P_3).
$$
But $T^*\otimes Q\in \P$. Thus the complex \eqref{com1} is acyclic. This implies that the complex 
$$
T\otimes P_1\to T\otimes P_2\to T\otimes P_3
$$
is acyclic, which implies the claim for $T\otimes Q_\bullet$. The claim for $Q_\bullet\otimes T$
is proved similarly.   

(ii) follows from (i). 
\end{proof} 

\begin{corollary}\label{rexa} The bifunctor $(X,Y)\mapsto X\otimes Y$ on $\C$ is right exact in both arguments, and 
$$
X\otimes^{L_i} Y=H^{-i}(X\otimes^L Y), 
$$
where $\otimes^{L_i}$ is the $i$-th left derived functor of the tensor product. 
Moreover, if $X$ is projective then the functors $X\otimes ?, ?\otimes X$ on $\C$ are exact. 
\end{corollary} 

\begin{proof} This follows immediately from Lemma \ref{l2a}(ii). 
\end{proof} 

\begin{corollary}\label{rex1a} The monoidal category $\T$ acts on $\C$ on the left 
by exact functors \linebreak $T_l(X):= T\otimes X$, 
$T\in \T,X\in \C$, commuting with the action of $\C$ on itself by right multiplication. 
There is also a similar action on the right given by $T_r(X):=X\otimes T$ which commutes with the action of $\C$ on itself by left multiplication. Also, these actions map projective objects to projective ones. 
\end{corollary} 

\begin{proof} This follows from Lemma \ref{l2a}(i): the left action is given by 
$T_l(X)=H^0(T\otimes Q_\bullet)$, where $Q_\bullet$ is a projective resolution of $X$. 
It is clear that the functors $T_l$ are well defined and right exact, and 
$L_jT_l(X)=H^{-j}(T\otimes Q_\bullet)=0$ for $j>0$. 
Thus the functors $T_l$ are exact. The same applies to the functors $T_r$. 
\end{proof} 

\subsection{The unit object of $\C$} Now assume that $\P$ is a finitely generated ideal, and $P\in \P$ a generator. Let $\bold 1$ be the unit object of $\T$. Since $P$ is rigid, we have the evaluation morphism ${\rm ev}: P^*\otimes P\to \bold 1$. Consider the morphism $\tau: P^*\otimes P\otimes P^*\otimes P\to P^*\otimes P$ given by the formula 
$$
\tau={\rm ev}\otimes 1\otimes 1-1\otimes 1\otimes {\rm ev}. 
$$
We define the object $\one$ in the category $\C$ to be the cokernel of $\tau$ (note that we use a different font to distinguish it from $\bold 1$). 

\begin{proposition} \label{one} For $Q\in \P\subset \T$ there are isomorphisms 
$$
Q\cong Q\otimes \one,\ Q\cong\one\otimes Q
$$ 
which are functorial in $Q$. 
\end{proposition} 

\begin{proof} By Corollary \ref{rexa}, $Q\otimes \one$ is the cokernel of the morphism 
$$
d_2:=1_Q\otimes \tau: Q\otimes P^*\otimes P\otimes P^*\otimes P\to Q\otimes P^*\otimes P.
$$
Let us extend this morphism to the ``bar complex"
\begin{equation} \label{barcom}
Q\otimes P^*\otimes P\otimes P^*\otimes P\to Q\otimes P^*\otimes P\to Q\to 0,  
\end{equation} 
where the last map is $d_1:=1_Q\otimes {\rm ev}$ (it is easy to see that $d_1d_2=0$). 
To construct a functorial isomorphism $Q\otimes\one\cong Q$, it suffices to show that this complex is exact. 
Since $P$ is a generator, every $Q$ is a direct summand in $X\otimes P$ for some 
$X\in \C$, so it suffices to show that \eqref{barcom} is exact for $Q=X\otimes P$. 
Hence, it suffices to show that 
\eqref{barcom} is split for $Q=P$. 

Since $P$ is rigid, \eqref{barcom} for $Q=P$ is split in the third term, with splitting $h_1: P\to P\otimes P^*\otimes P$ given by $h_1:={\rm coev}\otimes 1_P$, where ${\rm coev}: \bold 1\to P\otimes P^*$ is the coevaluation morphism. So it remains to show that \eqref{barcom} is split in the second term as well. To this end, we define a homotopy 
$$
h_2: P\otimes P^*\otimes P\to P\otimes P^*\otimes P\otimes P^*\otimes P
$$
by 
$$
h_2:={\rm coev}\otimes 1^{\otimes 3}=h_1\otimes 1\otimes 1.
$$
It is easy to check that $h_1d_1+d_2h_2=1$, which implies the statement. 

This proves that there is a canonical isomorphism $Q\to Q\otimes \one$. 
The isomorphism $Q\to \one\otimes Q$ is constructed similarly. 
\end{proof} 

\begin{corollary}\label{monoa} 
The semigroup category $\C$ is actually a monoidal category with unit $\one$. 
Moreover, if $\T$ is braided or symmetric then so is 
$\C$. 
\end{corollary} 

\begin{proof} Since $\C$ is the category of resolutions in $\P$, the functorial isomorphism $\one\otimes Q\cong Q$ for $Q\in \P$ from Proposition \ref{one} gives rise to an isomorphism $\one\otimes{} \cong {\rm Id}_\C$. Similarly, we have an isomorphism ${}\otimes\one\cong {\rm Id}_\C$. Thus, $\one$ is a unit in $\C$, hence $\C$ is monoidal. The second statement is clear. 
\end{proof} 

\begin{example}\label{notfin} The following example shows that the assumption that $\P$ is finitely generated cannot be dropped.
Let ${\rm char}(\k)=2$ and $Q=\k\Bbb Z/2$ be the regular representation of $\Bbb Z/2$ over $\k$. For a set $M$ let $(\Bbb Z/2)^M$ 
be the group of maps $M\to \Bbb Z/2$, and for $m\in M$ let $Q_m$ be the representation of $(\Bbb Z/2)^M$ on the vector space $Q$ given by 
$\rho_{Q_m}(g)=\rho_Q(g(m))$, $g\in (\Bbb Z/2)^M$, where $\rho_Q$ defines the action of $\Bbb Z/2$ on $Q$. Let $\T'=\T'(M)$ be the full subcategory of $\Rep_\k((\Bbb Z/2)^M)$ spanned by the tensor products $Q_N:=\cotimes_{m\in N}Q_m$ for finite subsets 
$N\subset M$. Let $\T=\T(M)$ be the quotient of $\T'$ by the tensor ideal generated by $Q_{ij}=Q_i\otimes Q_j$, where $i,j\in M$ are distinct (i.e., we set all morphisms factoring through a direct sum of $Q_{ij}$ to zero). Then $\T$ has indecomposables $\bold 1$ and $Q_m,m\in M$, with $\Hom(\bold 1,\bold 1)=\Hom(\bold 1,Q_i)=\Hom(Q_i,\bold 1)=\k$, $\Hom(Q_i,Q_i)=\k^2$, but $\Hom(Q_i,Q_j)=0$ for $i\ne j$ (since any morphism of representations $Q_i\to Q_j$ factors through $Q_{ij}$). It is easy to show that for any subset $N\subset M$ the ideal $\P_N=\langle Q_i,i\in N\rangle$ is a splitting ideal in $\T$, which is finitely generated if and only if $N$ is finite. The abelian semigroup category $\C(\T,\P_N)$ 
is $\C_N:=\coplus_{m\in N}\Rep_\k(\Bbb Z/2)$, which is monoidal (i.e., has a unit object) if and only if $N$ is finite. 
\end{example} 

\subsection{The monoidal functor $F$} We continue to assume that $\P\subset \T$ is a finitely generated ideal, and $\C=\C(\T,\P)$. 

\begin{corollary}\label{mono1a}
There is a monoidal functor 
$F: \T\to \C$ such that $F(T)=T_l(\one)=T_r(\one)$ (so $F(\bold 1)=\one$). 
Moreover, if $\T$ is braided or symmetric then so is $F$. 
\end{corollary} 

\begin{proof} This follows from Proposition \ref{one} and Corollaries \ref{rex1a}, \ref{monoa}. Namely, we have 
a monoidal functor $F_l: \T\to \End(\C_\C)\cong \C$ given by $F_l(T):=T_l(\one)$ 
and a monoidal functor $F_r: \T\to \End({}_\C\C)^{\rm op}\cong \C$ given by $F_r(T):=T_r(\one)$. 
Moreover, if $Q\in \P$ then $F_l(Q)$ and $F_r(Q)$ are canonically isomorphic to $Q$. 
However, for any $T\in \T$ we have that $F_l(T)$ is the cokernel of 
the natural morphism 
$$
F_l(1\otimes \tau): F_l(T\otimes P^*\otimes P\otimes P^*\otimes P)\to F_l(T\otimes P^*\otimes P), 
$$
while $F_r(T)$ is the cokernel of the natural morphism 
$$
F_r(1\otimes \tau): F_r(T\otimes P^*\otimes P\otimes P^*\otimes P)\to F_r(T\otimes P^*\otimes P). 
$$
This implies that $F_l\cong F_r$, which gives a construction of the functor $F\cong F_l\cong F_r$ with the claimed properties. 
\end{proof} 

\begin{corollary}\label{rig1} Every projective object in $\C$ is rigid. 
\end{corollary} 

\begin{proof} Indeed, by Corollary \ref{mono1a}, every projective object in $\C$ is of the form $F(Q)$, where $Q\in \P$, hence rigid. But a monoidal functor maps rigid objects to rigid ones. 
\end{proof} 

\begin{corollary}\label{faith} The functors $P\otimes ?$ and $?\otimes P$ on $\C$ are faithful. 
\end{corollary} 

\begin{proof} Let $f: X\to Y$ be a morphism in $\C$. Then the image of $1_P\otimes f$ is $P\otimes {\rm Im}(f)$. 
Now, $\Hom(X,P\otimes {\rm Im}(f))=\Hom(P^*\otimes X,{\rm Im}(f))$. But $P^*$ is a generator of $\P$. 
Thus if $P\otimes {\rm Im}(f)=0$ then $\Hom(Q,{\rm Im}(f))=0$ for all projectives $Q$, hence ${\rm Im}(f)=0$ and $f=0$, as claimed.  
\end{proof} 

\begin{proposition}\label{fullness} The map $F_{X,Y}: \Hom(X,Y)\to \Hom(F(X),F(Y))$ 
is surjective if either $X$ or $Y$ is in $\P$. Thus, the tensor ideal 
$\langle{\rm Pr}\rangle$ in $\C$ of morphisms factoring through projectives is contained 
in the image of $F$. 
\end{proposition} 

\begin{proof} By definition for $Q,R\in \P$ the map 
$F_{Q,R}:\Hom(Q,R)\to \Hom(F(Q),F(R))$ is an isomorphism. 
Now let $X\in \T$ and $f: F(Q)\to F(X)$ be a morphism.
We would like to show that $f=F(g)$ for some morphism $g: Q\to X$. 

Recall that $F(X)$ is a quotient of $F(X\otimes P^*\otimes P)$. 
Since $F(Q)$ is projective, the morphism $f$ admits a lift $f': F(Q)\to F(X\otimes P^*\otimes P)$.
Since $F(X\otimes P^*\otimes P)$ also is projective, there exists $g': Q\to X\otimes P^*\otimes P$ such that $F(g')=f'$. Let $g$ be the composition of $g'$ with the morphism 
$1_X\otimes {\rm ev}: X\otimes P^*\otimes P\to X$. Then $F(g)=f$, as claimed. 

Thus, the map $\Hom(Q,X)\to \Hom(F(Q),F(X))$ is surjective. 
Now, to show that the map $\Hom(X,Q)\to \Hom(F(X),F(Q))$ is also surjective, 
write it as the map \linebreak $\Hom(Q^*\otimes X,\bold 1)\to \Hom(F(Q^*\otimes X),F(\bold 1))$, 
which is surjective since $Q^*\otimes X\in \P$. This proves the proposition. 
\end{proof} 

Let $J_\P$ be the kernel of $F$, i.e., the collection of morphisms $f$ in $\T$ such that $F(f)=0$. 
Clearly, $J_\P$ is a tensor ideal. The following proposition describes $J_\P$ explicitly. 

\begin{proposition}\label{ker} Let $f: X\to Y$ be a morphism in $\T$. The following conditions are equivalent: 

(a) $f\in J_\P$;

(b1) $f\otimes 1_P=0$ in $\T$;

(b2) $1_P\otimes f=0$ in $\T$; 

(c1) For every object $Q\in \P$ we have $f\otimes 1_Q=0$ in $\T$;

(c2) For every object $Q\in \P$ we have $1_Q\otimes f=0$ in $\T$.
\end{proposition} 

\begin{proof} If $f\in J_\P$ then $F(f)=0$, so $F(f)\otimes 1_P=F(f\otimes 1_P)=0$. But 
$X\otimes P,Y\otimes P$ are projective and $F$ is fully faithful on projectives. Thus $f\otimes 1_P=0$. Similarly $1_P\otimes f=0$. Thus, (a) implies (b1) and (b2). 

If $f\otimes 1_P=0$ then $f\otimes 1_{P\otimes X}=f\otimes 1_P\otimes 1_X=0$ for all $X\in \C$, so for every object $Q\in \P$ we have $f\otimes 1_Q=0$. 
Thus, (b1) implies (c1). Similarly, (b2) implies (c2). 

If   for any $Q\in \P$    we have $f\otimes 1_Q=0$ then $F(f)\otimes 1_Z=0$ for any $Z\in \C$, i.e., $F(f)=0$, hence $f\in J_\P$. Thus (c1) implies (a). Similarly, (c2) implies (a).   
\end{proof} 

\begin{definition} If $J_\P=0$, i.e., $F$ is faithful, we will say that the ideal $\P$ is {\it faithful}. 
\end{definition} 

\begin{example} The ideals $\P$ in Examples \ref{multiten} and Example \ref{nonspli} are faithful. On the other hand, if $\T=\T_1\oplus \T_2$ with $\T_1\ne 0$ and $\P\subset \T_2$ 
then $\P$ is not faithful, as $\langle\T_1\rangle\subset J_\P$. Also, in Example \ref{notfin} with $M$ finite, 
$\P_N$ is faithful if and only if $N=M$, i.e., $\P_N=\mathcal{S}$ is the ideal of all splitting objects. 
In general, if $\P=\P_N$ then $J_\P=\langle\P_{M\setminus N}\rangle$. 
\end{example} 

\begin{definition}
Let us say that $\T$ is {\it of finite type} if the ideal $\mathcal{S}=\mathcal{S}(\T)$ is finitely generated. 
If so, denote the ideal $J_\S$ by $J=J(\T)$. Let us say that $\T$ is {\it separated} if $J=0$, i.e., $\mathcal{S}$ is faithful. 
\end{definition} 

In other words, $\T$ is separated if it has ``enough splitting objects", i.e., splitting objects separate morphisms (Proposition \ref{ker}(c1,c2)), which motivates the terminology. 

\begin{remark} It is clear that $\T$ is of finite type if and only if the sets $\overline{S}$ and $I_c,c\in \overline{S}$ are finite. 
\end{remark} 

\begin{example} A multitensor category with enough projectives is separated. On the other hand, the category 
of projective bimodules from Example \ref{nonspli} is not separated, since it has no nonzero splitting objects (so $J$ is the collection of all morphisms in $\T$). Similarly, if $G$ is a semisimple algebraic group in characteristic $p$   (of positive dimension)    then $\Rep(G)$ is not separated, for the same reason. 
\end{example} 

Note that if $X\in \S$ and $1_X\in J$ then $X\otimes X^*=0$, hence $X=0$ (as $X$ is a direct summand in $X\otimes X^*\otimes X$). Thus, faithfulness is not a very restrictive condition: if $\T$ is of finite type then the category $\T/J$ is separated (with the same indecomposable objects as $\T$), and $\C(\T,\S)=\C(\T/J,\S)$.  
Also we obtain the following proposition (which is not used below). 

\begin{proposition}\label{sepa} Let $\T$ be of finite type. Let $\P=\P_K$ for a subset $K\subset \overline{S}$ (as in Proposition \ref{classif}). Then $J_\P$ contains $\langle\P_{S\setminus K}\rangle$. In particular, 
$\P$ is faithful if and only if $\P=\S$ and $\T$ is separated. 
\end{proposition} 

\begin{proof} This follows immediately since $\P_K\otimes \P_{S\setminus K}=0$. 
\end{proof} 

Let us say that $\S$ is {\it indecomposable} if $\S\ne 0$ and $Q\sim_\ell R$ for all indecomposable $Q,R\in \S$. 
In this case, we get 

\begin{proposition}\label{uniqmin} If $\T$ is separated and $\S$ is indecomposable then it is the unique minimal (nonzero) thick ideal in $\T$.  
\end{proposition} 

\begin{proof} Lemma \ref{spliob} implies that $\S$ is minimal. Let $\mathcal{I}\ne 0$ be any thick ideal in $\T$ and $0\ne X\in \mathcal{I}$. 
Since $\T$ is separated, there is $Q\in \S$ such that $Q\otimes X\ne 0$. Then $Q\otimes X$ generates $\S$, so $\S\subset \mathcal{I}$. 
\end{proof} 

Let $\P\subset \T$ be a finitely generated ideal and $\widehat{\T}_\P$ be the full Karoubian completion of $F(\T)=\T/J$ inside $\C=\C(\T,\P)$ (i.e., the Karoubian completion of the full subcategory with the same objects as $F(\T)$). Then $\widehat{\T}_\P$ is a full Karoubian monoidal subcategory of $\C$, and the monoidal abelian category $\C(\widehat{\T}_\P,\P)$ coincides with $\C$. We will call $\widehat{\T}_\P$ the {\it completion} of $\T$ with respect to $\P$. 

\begin{definition} We will say that $\T$ is {\it complete}
with respect to $\P$ if the natural inclusion $\T/J\hookrightarrow \widehat{\T_\P}$ is an equivalence. 
\end{definition} 

It is clear that $\T$ is complete with respect to $\P$ if and only if for any $T\in \T$, the natural map 
$\Hom(\bold 1,T)\to \Hom(\one,F(T))$ is surjective. 

If $\T$ is of finite type and $\P=\S$, we will denote $\widehat{\T}_\P$ simply by $\widehat{\T}$, and say that $\T$ is {\it complete} if it is complete with respect to $\S$. In this case we will also write $\C(\T)$ for $\C(\T,\S)$. 
Thus, $\T$ is separated iff $F: \T\to \C(\T)$ is faithful and complete iff $F$ is full. In particular, 
the category $\widehat{\T}$ is separated and complete, and $\T$ is separated and complete iff the natural functor 
$\T\to \widehat{\T}$ is an equivalence. In this case the functor $F$ is a fully faithful embedding.  

\subsection{Properties of the tensor product on $\C$ in the case of a splitting ideal}

Now assume that $\P\subset \T$ is a finitely generated {\it splitting} ideal, and as before let $\C=\C(\T,\P)$. 

\begin{proposition}\label{spli} (i) For any $Y\in \C$, $Q\in \P$ the tensor products $Q\otimes Y$ and $Y\otimes Q$ are projective. 

(ii) Tensoring with every object $Q\in \P$ on either side splits every morphism in $\C$. 
\end{proposition} 

\begin{proof} (i) Let $f: Q_1\to Q_0$ be a presentation of $Y$. By Corollary \ref{rexa}, 
the cokernel of the morphism $1_Q\otimes f: Q\otimes Q_1\to Q\otimes Q_0$ is $Q\otimes Y$. 
But since $\P$ is a splitting ideal, this morphism is split. Hence $Q\otimes Y$ is a direct summand 
in $Q\otimes Q_0$, hence projective. 

(ii) Let $f: X\to Y$ be a morphism in $\C$, and $K,K'$ be the kernel and cokernel of $f$, respectively. Then by Corollary \ref{rexa}, $Q\otimes K$, $Q\otimes K'$ are the kernel and cokernel of $1_Q\otimes f$, respectively. But by (i) 
these objects are projective. This implies that the morphism $1_Q\otimes f$ is split. Similarly, the morphism $f\otimes 1_Q$ is split, as claimed. 
\end{proof} 

Recall that a {\it multiring category} over a field $\k$ is a $\k$-linear artinian monoidal category with biexact tensor product, \cite{EGNO}, Definition 4.2.3. 

\begin{proposition}\label{mainpro1a} The bifunctor $(X,Y)\mapsto X\otimes Y$ on $\C$ is exact in both arguments. Thus $\C$ is a multiring category.  
\end{proposition} 

\begin{proof} Let $Q_\bullet$ be a projective resolution of $Y\in \C$. Then by Corollary \ref{rexa} $H^{-i}(X\otimes^L Y)$ is the $-i$-th cohomology of the complex $X\otimes Q_\bullet$. So it suffices to show that if $P_1,P_2,P_3$ are projective and the complex 
$P_1\to P_2\to P_3$ is acyclic then the complex 
$$
X\otimes P_1\to X\otimes P_2\to X\otimes P_3
$$
is also acyclic. That is, we need to show that the complex 
$$
\Hom(R,X\otimes P_1)\to \Hom(R,X\otimes P_2)\to \Hom(R,X\otimes P_3)
$$
is acyclic for every projective $R$. 

Let $Q$ be projective. Then by Proposition \ref{spli}, $Q\otimes X$ is projective, so by Corollary \ref{rexa} the complex 
$$
\Hom(R,Q\otimes X\otimes P_1)\to \Hom(R,Q\otimes X\otimes P_2)\to \Hom(R,Q\otimes X\otimes P_3)
$$
is acyclic. Since $X\otimes P_i$ are projective (again by Proposition \ref{spli}), this can be written as 
$$
\Hom(Q^*\otimes R,X\otimes P_1)\to \Hom(Q^*\otimes R,X\otimes P_2)\to \Hom(Q^*\otimes R,X\otimes P_3).
$$
Taking $Q=R\otimes {^*}R$, we have $Q^*\otimes R=R\otimes R^*\otimes R$. Thus, using that $R$ 
is a direct summand in $R\otimes R^*\otimes R$ (as $R$ is rigid), we get that the complex 
$$
\Hom(R,X\otimes P_1)\to \Hom(R,X\otimes P_2)\to \Hom(R,X\otimes P_3)
$$
is acyclic, implying that the functor $X\otimes ?$ is exact. Similarly, the functor $?\otimes X$ is exact, as claimed. 
\end{proof} 

\begin{proposition}\label{rig2} Let $\D$ be a multiring category, and $X\in \D$ be the cokernel of a morphism 
$f: Q_1\to Q_0$ such that the objects $Q_0$, $Q_1$ are rigid. Then $X$ is also rigid. 
Namely, $X^*$ is the kernel  of $f^*: Q_0^*\to Q_1^*$ 
and ${^*}X$ is  the kernel of ${^*}f: {^*}Q_0\to {^*}Q_1$.
\end{proposition} 

\begin{proof} 
Define $X^*$ to be the kernel of the map $f^*: Q_0^*\to Q_1^*$. Let $i: X^*\to Q_0^*$, $j: f(Q_1)\to Q_0$ be the natural inclusions, and $\bar f: Q_1\to f(Q_1)$ the natural projection. We claim that the map ${\rm ev}_{Q_0}\circ (i\otimes j): X^*\otimes f(Q_1)\to \one$ is zero. To show this, it suffices to show that the map ${\rm ev}_{Q_0}\circ (i\otimes f): X^*\otimes Q_1\to \one$ is zero (indeed, tensoring with 
$X^*$ is right exact, so the map $1\otimes \bar f: X^*\otimes Q_1\to X^*\otimes f(Q_1)$ is surjective, and $f=j\circ \bar f$). But we have ${\rm ev}_{Q_0}\circ (i\otimes f)={\rm ev}_{Q_1}\circ (f^*\circ i\otimes 1)$, and 
 by definition $f^*\circ i=0$, which implies the claim. 

This means that the map ${\rm ev}_{Q_0}\circ (i\otimes 1): X^*\otimes Q_0\to \one$ gives rise to a map ${\rm ev}_X: X^*\otimes X\to \one$. 

Similarly, let $p: Q_0\to X$ be the natural projection and consider the coevaluation map ${\rm coev}_{Q_0}: \one\to Q_0\otimes Q_0^*$. It gives rise to a map $(p\otimes 1)\circ {\rm coev}_{Q_0}: \one\to X\otimes  Q_0^*$. Note that since the tensor product by $X$ is left exact, $X\otimes X^*\subset X\otimes Q_0^*$. 

We claim that $(p\otimes 1)\circ {\rm coev}_{Q_0}$ lands in $X\otimes X^*$, i.e., that the composition  
$$
(p\otimes f^*)\circ {\rm coev}_{Q_0}: \one\to X\otimes Q_1^*
$$ 
is zero. We have $=(p\otimes f^*)\circ {\rm coev}_{Q_0}=(p\circ f\otimes 1)\circ {\rm coev}_{Q_1}$. But by definition $p\circ f=0$, as claimed.  

Thus, the map $(p\otimes 1)\circ {\rm coev}_{Q_0}$ gives rise to a map ${\rm coev}_X: \one \to X\otimes X^*$. 

Now the axioms of the left dual for $X^*$ (\cite{EGNO}, Subsection 2.10) are easily deduced from those for $Q_0^*$. 
For example, let us check that the composition $X\to X\otimes X^*\otimes X\to X$ is the identity (the other axiom is proved similarly). Indeed, let $x\in X$ and let $y\in X$ be the image of $x$ under this composition. 
Fix a lift $x_1\in Q_0$ of $x$ and let $x_2$ be the image of $x_1$ in $Q_0\otimes Q_0^*\otimes Q_0$, and $x_3$ the projection of $x_2$ to $X\otimes Q_0^*\otimes Q_0$. Then $x_3$ has a unique preimage $x_4$ 
in $X\otimes X^*\otimes Q_0$. Let $x_5$ be the image of $x_4$ in $X\otimes X^*\otimes X$. By the definition of ${\rm coev}_X$ we have $(1\otimes {\rm ev}_X)(x_5)=y$. By the definition of ${\rm ev}_X$, it follows 
that $(1\otimes {\rm ev}_{Q_0})(x_3)=y$. Thus it follows from the rigidity of $Q_0$ that $y$ is the image of 
$x_1$ in $X$, i.e., $y=x$, as desired. 

Similarly, $X$ admits a right dual ${^*}X$, which is the kernel of the map ${^*}f: {^*}Q_0\to {^*}Q_1$. 
Thus, $X$ is rigid and the proposition is proved. 
\end{proof} 

\begin{proposition} \label{rigimono} 
(i) The monoidal category $\C=\C(\T,\P)$ is rigid, i.e., it is a multitensor category. 

(ii) Let $\P=\S$, so that $\C=\C(\T)$. Then we have $\C=\coplus_{c\in \overline{S}}\C(c)$, 
where $\C(c)$ are indecomposable multitensor subcategories. Namely, $\C(c)$ is the 
subcategory of objects generated by $c$. In particular, $\one=\coplus_{c\in \overline{S}}\one_c$. 

(iii) We have $\C(c)=\coplus_{i,j\in I_c}\C(c)_{ij}$, where $\C(c)_{ij}$ are component categories of $\C(c)$ (\cite{EGNO}, Definition 4.3.5). In particular, $\one_c=\coplus_{i\in I_c}\one_{c,i}$ where $\one_{c,i}$ 
are simple, and $\C(c)_{ij}=\one_{c,i}\C(c)\one_{c,j}$. 

(iv) $\S$ is indecomposable if and only if $\one\in \C$ is indecomposable, i.e., $\C$ is a tensor category. 
\end{proposition} 

\begin{proof} (i) Any $X\in \C$ can be represented as a cokernel of a morphism 
$f: Q_1\to Q_0$ between projective objects. By Corollary \ref{rig1}, the objects $Q_0$ and $Q_1$ are rigid. Thus
by Proposition \ref{rig2} and Proposition \ref{mainpro1a}, $X$ is rigid, as claimed. 

The proof of (ii), (iii) and (iv) is straightforward. 
\end{proof} 

\begin{remark} Example \ref{nonspli} shows that the splitting assumption for $\P$ cannot be dropped, as the category of finite dimensional bimodules over a finite dimensional non-semisimple Frobenius algebra is not rigid. 
\end{remark} 

As a by-product we obtain the following corollary (which is not used below). 

\begin{corollary}\label{semis} If $\T$ is separated then the algebra ${\rm End}(\bold 1)$ is semisimple. 
\end{corollary} 

\begin{proof} This follows since $\End(\one)$ is semisimple and commutative and $F_{\bold 1,\bold 1}: \End(\bold 1)\to \End(\one)$ is injective (for $\P=\S$). 
\end{proof} 

\begin{example} Let ${\rm char}(\k)=2$ and $\T_0:=\Rep_\k(\Bbb Z/2)$, with the usual monoidal structure. Then 
$\T_0$ has two indecomposable objects, $\bold 1$ and $P$ (the regular representation). 
Let $\T$ be the following modification of $\T_0$: we set $\End(\bold 1):=\k[\varepsilon]/\varepsilon^2$ 
instead of $\k$. The monoidal structure is modified in an obvious way (namely, 
the tensor product of $\varepsilon$ with any morphism $X\to Y$ where $X=P$ or $Y=P$ is zero).  
Then $\S$ is generated by $P$, so $J=\k \varepsilon$, and 
$\T/J=\C(\T/J)=\Rep_\k(\Bbb Z/2)$ (i.e., $\T$ is not separated but complete). This example 
shows that $J$ does not have to be a thick ideal, and that the separatedness assumption 
in Corollary \ref{semis} cannot be dropped. 
\end{example} 

\subsection{The main theorems} 

\begin{theorem}\label{main0} $\T$ admits a faithful monoidal functor $E$ into a multitensor category $\D$ with enough projectives which is full on projectives (i.e., any morphism between projectives is in the image) if and only if 
$\T$ is separated. Moreover, if such a functor exists, it is unique up to an isomorphism, with $\D\cong \C(\T)$ and $E\cong F$. 
The same holds for braided and symmetric categories. 
\end{theorem} 

\begin{proof} Suppose $\T$ is separated. Then by Proposition \ref{rigimono}, 
there is a monoidal functor $F: \T\to \C(\T)$ which satisfies the conditions of the theorem. 

Conversely, let $E: \T\to \D$ be a faithful monoidal functor into a multitensor category with enough projectives which is full on projectives. Then, since $E$ is full on projectives, for every projective $U\in \D$ there exists $Q\in \T$ such that $E(Q)=U$. Moreover, since $U$ splits every morphism in $\D$ and $E$ is faithful, we have $Q\in \S$. 

Now let $f: X\to Y$ be a morphism in $J$. Then $1_Q\otimes f=0$, hence $1_{E(Q)}\otimes E(f)=0$, i.e., $1_U\otimes E(f)=0$. 
Thus $U\otimes {\rm Im}(E(f))={\rm Im}(1_U\otimes E(f))=0$ in $\D$. This holds for every projective $U$, which implies ${\rm Im}(E(f))=0$, hence $E(f)=0$. Thus $f=0$ by the faithfulness of $E$. Thus, $\T$ is separated (the ideal of projectives is generated as a left ideal by the projective cover of $\one$, hence is finitely generated).  

Also, in this case $\S$ coincides with the ideal of projectives ${\rm Pr}\subset \D$, 
so we get that $\D\cong \C(\T)$ and $E\cong F$. 
\end{proof} 

Recall (\cite{EHS}, Section 9) that a multitensor category $\C\supset \T$ is called 
the {\it abelian envelope} of $\T$ if it has the following universal property: 
for any multitensor category $\D$ there is an equivalence between the category of tensor functors $E: \C\to \D$ and the category of faithful monoidal functors $\overline{E}: \T\to \D$, defined by the formula $\overline{E}=E|_\T$. The same definition is made for braided and symmetric categories (in which case the functors have to be braided). 

\begin{theorem}\label{main1} (i) $\T$ admits a fully faithful monoidal functor $E: \T\to \D$ into a multitensor category $\D$ with enough projectives if and only if it is separated and complete. Moreover, in this case there exists 
a tensor embedding $\widetilde{E}: \C(\T)\hookrightarrow \D$ such that $E\cong \widetilde{E}\circ F$. 
The same applies to braided and symmetric categories. 

(ii) In this case $\C(\T)$ is the abelian envelope of $\T$. 
\end{theorem}

\begin{proof} (i) If $\T$ is separated and complete then by Proposition \ref{rigimono}, 
we have a fully faithful monoidal functor $F: \T\to \C(\T)$ satisfying the conditions in the theorem.  

Conversely, let $E: \T\to \D$ be a fully faithful monoidal functor into a multitensor category $\D$ with enough projectives. 
Let $\E\subset \D$ be the tensor subcategory generated by $E(\T)$, \cite{EGNO}, Definition 4.11.1 (full and closed under subquotients). 
Then $\E$ is a multitensor category with enough projectives, and we have a fully faithful monoidal functor 
$E: \T\to \E$. Thus we may assume without loss of generality that $\E=\D$ (i.e., $E$ is surjective). 
In this case, let $U$ be a projective object of $\D$. Since $E$ is surjective, 
there exists $T\in \T$ such that $U$ is a subquotient of $E(T)$. Since $U$ is both projective and injective, 
it is in fact a direct summand of $E(T)$. Since $F$ is full, 
this means that $U=E(Q)$ for some direct summand $Q$ 
of $T$. Thus, $F$ is full on projectives. By Theorem \ref{main0} 
this implies that $\T$ is separated and $\D\cong \C(\T)$, $E\cong F$. 
Finally, since $F$ is full, $\T$ is complete. 

(ii) This follows from \cite{EHS}, Theorem 9.2.1. More precisely, this theorem applies to symmetric categories, but the proof extends without significant changes to the case of braided categories and categories without a fixed braided structure. 
\end{proof} 

Let $\D$ be a multitensor category with enough projectives. Recall that the {\it stable category} ${\rm Stab}(\D)$ is the quotient $\D/{\rm Pr}$, where ${\rm Pr}$ is the ideal of projectives. The indecomposables of ${\rm Stab}(\D)$ are the non-projective indecomposables of $\D$, so an object of ${\rm Stab}(\D)$ may be realized as an object in $\D$ with no nonzero projective direct summands. There is a natural monoidal functor ${\rm Stab}: \D\to {\rm Stab}(\D)$. 

Suppose $\E$ is a rigid monoidal Karoubian category and $E: \E\to {\rm Stab}(\D)$ 
a faithful monoidal functor. Consider the fiber product $\A:=\E\times_{{\rm Stab}(\D)}\D$, whose objects are 
triples $(X,Y,f)$, where $X\in \E$, $Y\in \D$, and $f: E(X)\to {\rm Stab}(Y)$ is an isomorphism, and morphisms 
$(X_1,Y_1,f_1)\to (X_2,Y_2,f_2)$ are pairs $(g,h)$, $g: X_1\to X_2$, $h: Y_1\to Y_2$ 
such that $f\circ E(g)={\rm Stab}(h)\circ f$. This is naturally a rigid monoidal Karoubian category. 
If $(X,Y,f)$ and $(X,Y,f')$ are objects of $\A$ then there exists an automorphism
$g: Y\to Y$ such that ${\rm Stab}(g)=f'\circ f^{-1}$ and it defines an isomorphism $(X,Y,f)\cong (X,Y,f')$. Thus 
isomorphism classes of objects in $\A$ are isomorphism classes of pairs $(X,Y)$ 
such that $E(X)\cong {\rm Stab}(Y)$, i.e., isomorphism classes of pairs $(X,Q)$ where 
$Q$ is a projective object of $\D$; namely, $(X,Q)=X\oplus Q$. Thus the indecomposables 
of $\A$ are the indecomposables of $\E$ and the indecomposable projectives 
of $\D$. Moreover, for $X,X'\in \E$ and projective $Q,Q'\in \D$ we have 
$$
\Hom_\A(X,X')=\Hom_\E(X,X'),\ \Hom_\A(Q,Q')=\Hom_\D(Q,Q'),
$$
$$
\Hom_\A(Q,X)=\Hom_\D(Q,E(X)), \Hom_\A(X,Q)=\Hom_\D(E(X),Q). 
$$
The category $\A$ is equipped with a natural faithful monoidal functor $\widetilde E: \A\to \D$. 

Note that if $E$ is pseudomonic, i.e., full on isomorphisms (in particular, if it is fully faithful) then 
$\A$ is naturally equivalent to a rigid monoidal isomorphism-closed subcategory of $\D$ containing 
the tensor ideal generated by projective objects. 

\begin{theorem}\label{main2} There is a 1-1 correspondence between (equivalence classes of): 

(1) Separated categories $\T$ such that $\C(\T)\cong \D$; 

(2) Rigid monoidal Karoubian categories $\overline{\T}$ together with a faithful monoidal functor 
\linebreak $\overline{F}: \overline{\T}\to {\rm Stab}(\D)$; 

\noindent which sends complete categories $\T$ to full subcategories $\overline{\T}\subset {\rm Stab}(\D)$, 
and vice versa. Namely, this correspondence is defined by the formulas 
$$
\overline{\T}=\T/{\rm Pr},\ \T=\overline \T\times_{{\rm Stab}(\D)}\D.
$$
\end{theorem} 

\begin{proof} Given $\T\subset \D$, define $\overline{\T}:=\T/{\rm Pr}$. Then $\overline{\T}$ is a rigid monoidal Karoubian category with a faithful monoidal functor $\overline{F}:\overline{\T}\to {\rm Stab}(\D)$. Conversely, suppose $\E$ is a rigid monoidal Karoubian category with a faithful monoidal functor $E: \E\to {\rm Stab}(\D)$ and set \linebreak $\widetilde{\E}:=\E\times_{{\rm Stab}(\D)} \D$. Then by Theorem \ref{main0} $\widetilde{\E}$ is a separated category and $\C(\widetilde{\E})\cong \D$.

Now, it is clear that $\widetilde{\overline{\E}}=\E$, and the equality $\overline{\widetilde{\T}}=\T$ follows from Proposition \ref{fullness}. 
\end{proof} 

\begin{example}\label{exaa} 1. Suppose $\D$ is semisimple. Then ${\rm Stab}(\D)=0$, so Theorem \ref{main2} says that 
a separable category $\T$ such that $\C(\T)=\D$ must be complete and coincide with $\D$ (which is easy). 

2. Let $\C$ be a finite multitensor category and $\T={\rm Stab}^{-1}(\Vec)$ the preimage of $\Vec$ under the projection
${\rm Stab}: \C\to {\rm Stab}(\C)$. The indecomposable objects of $\T$ are $\one$ and the indecomposable projectives 
of $\C$ (note that $\one$ may be decomposable in $\C$ but not in $\T$). 
Then the ideal $\P\subset \T$ of projectives  coincides with $\S$, $\T$ is separated and $\C(\T)=\C$. 
But if $\one$ is decomposable and has length $m\ge 2$ then the map $\k=\Hom_\T(\bold 1,\bold 1)\to \Hom_\C(\one,\one)=\k^m$ is not surjective, hence $F$ is not full and $\T$ is not complete. 

A simple example of this is the case when $\C={\rm Mat}_N(\Rep_\k(\Bbb Z/2))$, ${\rm char}(\k)=2$, where $N$ is a finite set (i.e., the objects of $\C$ are matrices of objects of $\Rep_\k(\Bbb Z/2)$ with rows and columns labeled by $N$ and tensor product given by matrix multiplication using the usual tensor product of $\Rep_\k(\Bbb Z/2)$). In this case the indecomposable objects of  $\T$ are $\bold 1$ and also $R_{ij}$, $i,j\in N$, where $R_{ij}$ is the regular representation of $\Bbb Z/2$ standing in the $(i,j)$-th position. Note that this example also makes sense when $N$ is an infinite set, in which case the ideal $\mathcal{S}$ (generated by all the $R_{ij}$) is not finitely generated as a left ideal, but is finitely generated as a two-sided ideal (in fact, by any one of the $R_{ij}$). 

3. Let $\C=\Rep_\k(\Bbb Z/p)$ where ${\rm char}(\k)=p>2$. Recall that the {\it Heller shift} $\Omega(\bold 1)$ is the kernel of the surjection $P_{\bold 1}\to \bold 1$, where $P_{\bold 1}$ is the projectiuve cover of $\bold 1$. We have a faithful symmetric monoidal functor $E:\sVec\hookrightarrow {\rm Stab}(\C)$ sending the non-trivial simple object to the indecomposable module $L=\Omega(\bold 1)$ of size $p-1$.
 Let $\T={\rm Stab}^{-1}({\rm Im}(E))$. Then $\T$ has three indecomposables: $\bold 1,L,P$, with $L\otimes P=P^{p-1}$, $P\otimes P=P^p$ and $L\otimes L=P^{p-2}\oplus \bold 1$, and $\Hom(\bold 1,P)=\Hom(P,\bold 1)=\k$, $\Hom(P,P)=\k^p$, 
$\Hom(P,L)=\Hom(L,P)=\k^{p-1}$, $\Hom(L,L)=\k[z]/z^{p-1}$, 
$\Hom(\bold 1,L)=\Hom(L,\bold 1)=0$. 
In this case $\S:=\langle P\rangle$ is a finitely generated ideal and $\T$ is separated, but 
the functor $F$ is not full, since the morphism \linebreak $0=\Hom_\T(\bold 1,L)\to \Hom(\one,F(L))=\k$ is not surjective. Thus $\T$ is not complete.  

On the other hand, in this example we may replace $\sVec$ by the {\it full} monoidal subcategory of ${\rm Stab}(\C)$ generated by $\bold 1$ and $L$ (i.e., $\Hom(\bold 1,L)=\Hom(L,\bold 1)=\k$ rather than zero), and the functor $E$ with the tautological embedding. In this case 
$\mathcal \T={\rm Stab}^{-1}({\rm Im}(E))$ is both separated and complete; it is the full monoidal subcategory of $\C$ spanned by $\bold 1,L,P$. We have $\T=\C$ for $p=3$, but for $p>3$ it is smaller than $\C$ (and not abelian). But for any $p>2$ the category $\C$ is the abelian envelope of $\T$. 

4. Let ${\rm char}(\k)=2$ and $\C=\Rep_\k(\Bbb Z/2)$. Then ${\rm Stab}(\C)=\Vec$ is the category of finite dimensional vector spaces. Let $H$ be a finite dimensional Hopf algebra over $\k$ and 
$\E:=\Rep(H)$. Let $E: \E\to \Vec={\rm Stab}(\C)$ be the forgetful functor. Let $\T=\E\times_{{\rm Stab}(\C)}\C$. 
Then $\T$ is separated and $\C(\T)=\C$. Note that $E$ is not full on isomorphisms in general, so $\T$ is not equivalent to an isomorphism-closed subcategory of $\C$, but it admits a faithful monoidal functor 
$F: \T\to \C$. 
\end{example} 

\begin{example}\label{mainex} Here is a prototypical example of a separated and complete category $\T$. 
Let $\C$ be a multitensor category with enough projectives which is tensor generated by an object $X\in \C$. 
This means that any object of $\C$ is a subquotient of a direct sum $Y$ of tensor products of objects $X^{(n)}$, where 
$X^{(n)}$ is the $n$-th dual of $X$. Let $\T$ be the full rigid monoidal {\it Karoubian} subcategory of $\C$ generated by $X$, 
i.e., its objects are {\it direct summands} (rather than subquotients) of direct sums $Y$ as above. Then by Theorem \ref{main1}, $\S\subset \T$ is the ideal of projective objects in $\C$, and $\C=\C(\T)$.  
 
This example becomes interesting when the category $\T$ has an independent (often diagrammatic) description not involving $\C$, and the corresponding universal property. This happens in a number of concrete examples, notably the following. 

1. Let $\C=\Rep^{\rm ab}(S_t)$ is the abelian envelope of the Deligne category $\Rep(S_t)$ for $t\in \Bbb Z_+$, ${\rm char}(\k)=0$ (\cite{D2,CO}). Let $X$ be the ``tautological" object (the interpolation of the permutation representation of $S_n$). Then $\T$ is the Karoubian Deligne category $\Rep(S_t)$ (\cite{D2}). The indecomposable objects of $\T$ are $X_\lambda$ 
labeled by partitions $\lambda$, with $X_0=\bold 1$. The blocks in this category are described in \cite{CO}; there are semisimple blocks corresponding to a single isolated partition and also semi-infinite chains of partitions, whose structure depends on $t$. The ideal $\S$ is then spanned by $X_\lambda$ for the isolated $\lambda$ and also all other $\lambda$ except the first element of each chain; these exceptions are exactly the partitions for which $t\ge \lambda_1+|\lambda|$ (they are in bijection with partitions of $t$ defined by $\lambda\mapsto (t-\lambda_1,\lambda_1,...,\lambda_n)$). Indeed, it is precisely these $X_\lambda$ that become projective in $\C$ (\cite{EA}, 4.4). 

2. Let $\C=\Rep (GL(m|n))$, ${\rm char}(\k)=0$, and assume for simplicity that $m\ge n$. Let $X=\k^{m|n}$ be the defining representation. Then the category $\T$ is the category of modules that occur as direct summands in $X^{\otimes i}\otimes X^{*\otimes j}$. Let $t=m-n$ and $\Rep(GL_t)$ be the Karoubian Deligne category for parameter $t$. It is shown in \cite{C1,Co} that the only nontrivial tensor ideals in $\Rep(GL_t)$ are $J_0\supset J_1\supset J_2...$, such that $J_i$ is the kernel of the natural functor $\Rep(GL_t)\to \Rep(GL(t+i|i))$. Then $\T=\Rep(GL_t)/J_n$ (see \cite{Co,CW,EHS}). 

3. Let $G$ be a simply connected simple algebraic group and $q$ a root of unity of sufficiently large order. 
Let $\C=\Rep(G_q)$ be the category of finite dimensional representations of the corresponding quantum  group (${\rm char}(\k)=0$), and $\T={\rm Tilt}(G_q)\subset \Rep(G_q)$ be the category of tilting modules; it is obtained by the above construction from the direct sum of the fundamental representations $X=\coplus_i L_{\omega_i}$ (note that all projectives in $\C$ are tiltings, so $\T$ is separated and complete and $\C=\C(\T)$). If $G=SL_n$, the category $\T$ is the category of modules occurring 
as direct summands in tensor powers of the tautological representation $V$, and it is expected to have a nice diagrammatic description as a braided category using webs (a root of unity analog of \cite{CKM}). In particular, we see that 
$\C$ is the abelian envelope of $\T$, which is a quantum analog of \cite{CEH}, Theorem 3.3.1.

4. Let $G$ be as above, $p$ a prime, $\k$ a field of characteristic $p$, $\C=\Rep(G)$ be the category 
of finite dimensional representations of $G$ over $\k$, and $\T\subset \C$ be the category of tilting modules. Then $\C$ is the abelian envelope of $\T$ (\cite{CEH}, Theorem 3.3.1). Since the category $\C$ has no nonzero projective objects, this is not quite an instance of the above theory. However, it is a limit of such instances. 
More precisely, let $G_n$ be the $n$-th Frobenius kernel of $G$. Then we have the restriction functor 
$\Rep(G)\to \Rep(G_n)$. Denote by $\T(n)$ the image of $\T$ in $\Rep(G_n)$. Then 
$\Rep(G_n)=\C(\T(n))$. This follows from the fact that the $n$-th Steinberg representation ${\rm St}_n$ of $G$ is projective when restricted to $G_n$.

5. Let $\T_p$ be the category of tilting modules for the group $SL_2(\k)$ (${\rm char}(\k)=p$), with indecomposable tilting objects $T_i$, $i\ge 0$, with highest weight $i$. Then $\T_p$ has a thick ideal $\mathcal{I}_n$ generated by the $n$-th Steinberg representation ${\rm St}_n:=T_{p^n-1}$, which is spanned by the objects $T_i$ with $i\ge p^n-1$. 
Let $\T_{n,p}:=\T_p/\mathcal{I}_n$ (\cite{BE}). It is shown in \cite{BE}, Theorem 4.14 that for $n\ge 0$ the category $\T_{n+1,2}$ has an abelian envelope $\C_{2n}$, a finite symmetric tensor category in which it sits as a full monoidal generating subcategory. Thus, $\T_{n+1,2}$ is separated and complete, and $\C_{2n}=\C(\T_{n+1,2})$.  
\end{example} 

However, while these are nice examples, the true power of the above theory is seen when the category $\C(\T)$ is not known a priori, and this theory provides its first construction. This is what happens for $\T_{n,p}$  in the case $p>2$, and the main goal of this paper is to define and study the corresponding categories $\C(\T_{n,p})$.  

\subsection{Splitting objects in monoidal categories over discrete valuation rings}\label{DVR}

For applications, we will need to generalize the above theory of splitting ideals and associated abelian envelopes to the deformation-theoretic setting, when the ground ring is not a field but rather a complete discrete valuation ring (DVR), which in applications will be of mixed characteristic. 

Let $R$ be a complete DVR with maximal ideal $\mathfrak{p}$ and uniformizer $\varepsilon$.
Let $\k:=R/\mathfrak{p}$ be its residue field, and $K$ be the fraction field of $R$. 

\begin{definition}  By a {\it flat monoidal Karoubian category} over $R$ we will mean an $R$-linear monoidal Karoubian category $\T$ for which $\Hom(X,Y)$ is a free finite rank $R$-module for any $X,Y\in \T$. 
\end{definition} 

If $\T$ is a flat monoidal Karoubian category over $R$ then we have the monoidal Karoubian category $\T_\k:=\T/\mathfrak{p}$ over the residue field $\k$. We also have the category $\T_K$, the Karoubian closure of $K\otimes_R\T$ (which may not be Karoubian by itself), a monoidal Karoubian category over $K$.\footnote{Note that the field $K$ is not algebraically closed, but this is not important for our considerations unless specified otherwise. For background on multitensor categories over arbitrary fields see \cite{EGNO}, Subsection 4.16.} Moreover, it is easy to show that if $\T_\k$ is rigid then so are $\T$, $\T_K$.   
 
We have a natural bijection between objects of $\T$ and $\T_\k$ given by $X\in \T\mapsto X_\k\in \T_\k$,  $Y\in \T_\k\mapsto \widetilde{Y}\in \T$, such that $\Hom(X_\k,Z_\k)=\k\otimes_R\Hom(X,Z)$.  Also, any object $X\in \T$ defines an object $X_K\in \T_K$, with $\Hom(X_K,Z_K)=K\otimes_R\Hom(X,Z)$. Thus any object  $Y\in \T_\k$ uniquely lifts to an object of $\widetilde{Y}_K\in \T_K$ (although the lift of an indecomposable in $\T$ may be decomposable in $\T_K$). 
 
 Note that the assignments $X\mapsto X_\k$ and $X\mapsto X_K$ 
 are symmetric monoidal functors $\T\to \T_\k$ and $\T\to \T_K$. Moreover, 
 if $f: X\to Y$ is a split surjection or injection in $\T$ then so are $f_\k: X_\k\to Y_\k$ and 
 $f_K: X_K\to Y_K$.  

A flat monoidal category $\T$ over $R$ cannot have nonzero splitting objects in the sense of Definition \ref{spliobj}, since any element 
$x\in \mathfrak{p}$, $x\ne 0$ defines a morphism $x_{\bf 1}: {\bf 1}\to {\bf 1}$ which is not split by tensoring with any nonzero object $Q\in \T$. However we can make a less restrictive definition. 

\begin{definition} 1. A morphism $f: Q_1\to Q_0$ in $\T$ is {\it $R$-split} if for every 
$X\in \T$ the induced morphism of free $R$-modules $f_X: \Hom(X,Q_1)\to \Hom(X,Q_0)$ is split (i.e., its cokernel is a free $R$-module). 

2. An object $Q\in \T$ is {\it a splitting object} if tensoring with $Q$ on either side splits every $R$-split morphism. 
\end{definition} 

\begin{example} Any split morphism in $\T$ is automatically $R$-split. Indeed, 
let $f: X\to Y$ be a split morphism in $\T$. Then $f=0\oplus g$, where $g$ 
is an isomorphism. Since $0$ and $g$ are $R$-split, so is $f$.   
\end{example} 

\begin{lemma}\label{splisa} Let $f: X\to Y$ be an $R$-split morphism in $\T$. 

(i)  If $f$ is nonzero then
the corresponding morphism $f_\k: X_\k\to Y_\k$ is nonzero. 

(ii) If $f_\k$ is split then $f$ is split. 

(iii) If $Q\in \T$ then the morphisms $1_Q\otimes f$ and $f\otimes 1_Q$ are $R$-split. 

(iv) If $Q_\k$ is a splitting object of $\T_\k$ then $Q$ is a splitting object of $\T$. 
\end{lemma} 

\begin{proof} (i) Suppose $f_\k=0$. For an object $Z\in \T$ consider the map 
$f\circ{}: \Hom(Z,X)\to \Hom(Z,Y)$. The map $f_\k\circ{}: \Hom(Z_\k,X_\k)\to \Hom(Z_\k,Y_\k)$ 
is zero. Thus $f\circ{}=\varepsilon g$ for some map $g: \Hom(Z,X)\to \Hom(Z,Y)$. 
But $f\circ$ is $R$-split, which implies that $g=0$, i.e., $f\circ=0$. Thus $f=0$. 

(ii) Since $f_\k$ is split, up to composing with isomorphisms on both sides it can be written as a direct sum of the zero morphism $X'\to Y'$ and the identity morphism $Z\to Z$, i.e., as a block matrix
$$
f_\k\sim \left(\begin{matrix} 0& 0\\ 0& 1\end{matrix}\right).
$$
This means that $f$ can be similarly written as a block matrix
$$
f\sim \left(\begin{matrix} \varepsilon a& \varepsilon b\\ \varepsilon c& 1+\varepsilon d\end{matrix}\right).
$$
Thus up to composing with isomorphisms
$$
f\sim \left(\begin{matrix} \varepsilon (a-\varepsilon b(1+\varepsilon d)^{-1}c) & 0\\ 0& 1+\varepsilon d\end{matrix}\right).
$$
Since $f$ is $R$-split, $\varepsilon (a-\varepsilon b(1+\varepsilon d)^{-1}c)$ must be $R$-split, hence it is zero by (i). Thus 
$$
f\sim \left(\begin{matrix} 0& 0\\ 0& 1+\varepsilon d\end{matrix}\right),
$$
i.e., $f$ is a direct sum of zero and an isomorphism. Thus, $f$ is split. 

(iii) The morphism $1_Q\otimes f$ is $R$-split since the corresponding map
$$
(1_Q\otimes f)\circ: \Hom(Z,Q\otimes X)\to \Hom(Z,Q\otimes Y)
$$ 
coincides with the map $f\circ: \Hom(Q^*\otimes Z,X)\to \Hom(Q^*\otimes Z,Y)$, which is 
split by definition. Similarly, $f\otimes 1_Q$ is $R$-split. 

(iv) Assume that $Q_\k$ is splitting. Then the morphism 
$1_{Q_\k}\otimes f_\k=(1_Q\otimes f)_\k$ is split. Also the morphism 
$1_Q\otimes f$ is $R$-split by (iii). Thus by (ii) the morphism $1_Q\otimes f$ is split. Similarly, the morphism $f\otimes 1_Q$ is split. Thus, $Q$ is a splitting object. 
\end{proof} 

It follows from Lemma \ref{splisa}(iii) that splitting objects form a thick ideal $\S\subset \T$. 
Thus, if $Q$ is a splitting object then so is $Q^*$. Hence, if $f$ is an $R$-split morphism 
and $Q$ is a splitting object then $Q$ splits the morphism $f^*$ (as $Q^*$ splits $f$). 

From now on assume that the category $\T_\k$ is separated and its ideal of splitting objects $\S(\T_\k)$ is finitely generated as a left ideal. Thus we can define the multitensor category $\C(\T_\k)$ with enough projectives and a faithful monoidal functor $F: \T_\k\to \C(\T_\k)$.

\begin{lemma}\label{splisa1} (i) Let $P$ be a generator of $\S(\T_\k)$. 
Then the morphism ${\rm coev}_{\widetilde{P}}: \bold 1\to \widetilde{P}\otimes \widetilde{P}^*$ 
is $R$-split. 

(ii) If $Q$ is a splitting object of $\T$ then $Q_\k$ is a splitting object of $\T_\k$. In other words, $\S(\T_\k)=\S_\k$. 
\end{lemma} 

\begin{proof} (i) By applying the faithful monoidal functor $F: \T_\k\to \C(\T_\k)$ and using Proposition \ref{fullness} we get that for any $X\in \T_\k$ the map 
${\rm coev}_P\circ:\Hom(X,\bold 1_\k)\to \Hom(X,P\otimes P^*)$ is injective. 
This implies that the map 
$$
({\rm coev}_{\widetilde{P}}\circ)_K: K\otimes_R\Hom(\widetilde{X},\bold 1)\to K\otimes_R \Hom(\widetilde{X},\widetilde{P}\otimes \widetilde{P}^*)
$$
is also injective, and thus has the same rank as ${\rm coev}_P$ (the dimension of the source). 
Thus ${\rm coev}_{\widetilde{P}}$ is split as a morphism of $R$-modules, as claimed. 

(ii) Let $Q$ be an indecomposable splitting object of $\T$, and $P$ a generator of $\S(\T_\k)$. Then  $Q_\k\otimes P\ne 0$ as $\T_\k$ is separated. Consider the morphism $f=1_Q\otimes 
{\rm coev}_{\widetilde P}:  Q\to Q\otimes \widetilde{P}\otimes \widetilde{P}^*$. Since $Q_\k\otimes P\ne 0$, we have $f_\k\ne 0$, hence $f\ne 0$. On the other hand, the morphism 
${\rm coev}_{\widetilde P}$ is $R$-split by (i), so the morphism $f$ is split. Since $Q$ is indecomposable, this implies that $f$ is a split injection. Thus $f_\k: Q_\k\to Q_\k\otimes P\otimes P^*$ is also a split injection, i.e., $Q_\k$ is a direct summand in $Q_\k\otimes P\otimes P^*$, hence split, as claimed. 
 \end{proof} 

For simplicity assume that $\T$ has finitely many indecomposables. Then we can define the category $\C(\T)$ in the same way as we did over a field. Namely, let $P$ be the direct sum of all indecomposable splitting objects of $\T_\k$, 
$\widetilde P$ the corresponding object of $\T$, $A=\End(\widetilde P)^{\rm op}$, and $\C(\T)=\Rep(A)$ be the category of finitely generated $A$-modules. Then $\C(\T)$ can be realized as a category 
of {\it resolutions} 
$$
...\to Q_2\to Q_1\to Q_0,
$$ 
where $Q_i\in \T$ are splitting objects (i.e., complexes $Q^\bullet$ of splitting objects such that 
the complex $\Hom(\widetilde P,Q^\bullet)$ is exact). As before, this is an $R$-linear abelian monoidal category with right exact tensor product. Moreover, we have a faithful monoidal functor 
$F: \T\to \C(\T)$ given by $F(X)=\Hom(\widetilde P,X)$. 

Of course, we cannot hope that the tensor product in $\C(\T)$ is exact, since for example $\C(\T)$ contains the category of finitely generated $R$-modules with tensor product over $R$ (quotients of a multiple of $\one_i$ for some $i$). 
However, it turns out that it is exact on a certain class of objects we call {\it flat}. 

\begin{definition} An object $X\in \C(\T)$ is {\it flat} if it is presented by an $R$-split morphism $f: Q_1\to Q_0$, $Q_1,Q_0\in \S$; i.e., $X={\rm Coker}(F(f))$. 
\end{definition} 

For instance, every splitting object is automatically flat.

\begin{lemma}\label{flatobj} The following conditions on $X\in \C(\T)$ are equivalent: 

(i) $X$ is flat; 

(ii) $X$ is a free $R$-module;   

(iii) Any morphism $f: Q_0\to Q_1$ with $Q_0,Q_1\in \S$ presenting $X$ is $R$-split. 
\end{lemma} 

\begin{proof} Clearly (iii) implies (i). 

Suppose $X$ is flat, and presented by an $R$-split morphism 
$f: Q_1\to Q_0$. Then $X$ is the cokernel of the map 
$f\circ: \Hom(\widetilde P,Q_1)\to \Hom(\widetilde P,Q_0)$. By definition, this map is $R$-split, 
hence $X$ is a free $R$-module. Thus (i) implies (ii). 

Suppose that $X$ is a free $R$-module, and presented by a morphism 
$f: Q_1\to Q_0$. Then by Proposition \ref{fullness} for $Z\in \T$ the natural map 
$\Hom(Z,Q_i)\to \Hom(F(Z),F(Q_i))$ is an isomorphism. Thus we have an exact sequence 
$$
\Hom(Z,Q_1)\to \Hom(Z,Q_0)\to \Hom(F(Z),X). 
$$  
Thus the cokernel of the map $f\circ: \Hom(Z,Q_1)\to \Hom(Z,Q_0)$ is an $R$-submodule 
of $\Hom(F(Z),X)$ which is, in turn, contained in $\Hom_R(F(Z),X)$. But $\Hom_R(F(Z),X)$ is a free $R$-module since $X$ is free as an $R$-module. Thus, ${\rm Coker}(f\circ)$ is a free $R$-module and $f$ is $R$-split.  
\end{proof} 

\begin{corollary}\label{satmor} Let $Q_0,Q_1\in \S$. A morphism 
$f: Q_1\to Q_0$ is $R$-split if and only if $F(f): F(Q_1)\to F(Q_0)$ 
is split as a morphism of $R$-modules. 
\end{corollary} 

\begin{proof} Suppose $f$ is $R$-split. Then $F(f)=f\circ: \Hom(\widetilde P,Q_1)\to \Hom(\widetilde P,Q_0)$ is split as a morphism of $R$-modules. Conversely, if $F(f)$ is split then $f$ presents the object ${\rm Coker}(F(f))$, 
which is free as an $R$-module, so by Lemma \ref{flatobj} $f$ is $R$-split. 
\end{proof} 

\begin{corollary}\label{tiltflat} For every $Z\in \T$ the object $F(Z)$ is flat. 
\end{corollary} 

\begin{proof} This follows from Lemma \ref{flatobj} since $\Hom(\widetilde P,Z)=F(Z)$ is a free $R$-module. 
\end{proof} 

\begin{proposition}\label{flatrig} $X\in \C(\T)$ is flat if and only if it is rigid, and in this case the functors $X\otimes ?$ and $?\otimes X$ are exact. 
\end{proposition} 

\begin{proof} We start with the following lemma. 
\begin{lemma}\label{spli11} For any flat $Y\in \C(\T)$ and projective $Q\in \C(\T)$ the tensor products $Q\otimes Y$ and $Y\otimes Q$ are projective. 
\end{lemma} 

\begin{proof} 
Let $f: Q_1\to Q_0$ be a presentation of $Y$ by an $R$-split morphism. By a generalization of Corollary \ref{rexa} to $R$-linear categories, the cokernel of the morphism 
$$
1_Q\otimes F(f): Q\otimes F(Q_1)\to Q\otimes F(Q_0)
$$ 
is $Q\otimes Y$. Since $f$ is $R$-split and $Q=F(Q')$ for some splitting object $Q'$, the morphism $1_Q\otimes F(f)=F(1_{Q'}\otimes f)$ is split (as so is $1_{Q'}\otimes f$). Hence $Q\otimes Y$ is a direct summand in $Q\otimes F(Q_0)=F(Q'\otimes Q_0)$, hence projective. 
\end{proof} 

\begin{lemma} \label{mainpro1a11} The following conditions on $X\in \C(\T)$ are equivalent: 

(i) $X$ is flat;

(ii) the functor $X\otimes$ is exact; 

(iii) the functor $\otimes X$ is exact. 
\end{lemma} 

\begin{proof} The proof that (i) implies (ii), (iii) is the same as that of Proposition \ref{mainpro1a}, using Lemma \ref{spli11} instead of Proposition \ref{spli}. 

To prove that (ii) implies (i), we may assume that $X\in \C(\T)_{ij}$ for some $i,j$. 
If the functor $X\otimes$ is exact then in particular it is exact on the subcategory 
of finitely generated $R$-modules, $R$-${\rm mod}\subset \C(\T)_{jj}$ (quotients of a multiple of $\one_j$). If $M\in R$-${\rm mod}$ then $X\otimes M$ is just the usual tensor product $X\otimes_R M$. Since this functor (of $M$) is exact, $X$ is flat as an $R$-module, i.e., free (as it is finitely generated). Hence it is flat as an object of $\C(\T)$ by Lemma \ref{flatobj}. 
\end{proof} 

\begin{lemma}\label{rig211} Let $X\in \C(\T)$ be a flat object which is 
the cokernel of an $R$-split morphism $f: Q_1\to Q_0$. Then $X$ is rigid. 
Namely, $X^*$ is the kernel  of $f^*: Q_0^*\to Q_1^*$ 
and ${^*}X$ is  the kernel of ${^*}f: {^*}Q_0\to {^*}Q_1$.
\end{lemma} 

\begin{proof} The proof is the same as that of Proposition \ref{rig2}, using Lemma \ref{mainpro1a11} instead of Proposition \ref{mainpro1a}. 
\end{proof} 

Lemma \ref{rig211} shows that any flat object is rigid. Also, for every rigid $X$ the functor $X\otimes $ is exact, which implies that $X$ is flat by Lemma \ref{mainpro1a11}. 
This completes the proof of Proposition \ref{flatrig}. 
\end{proof} 

Let us now describe the tensor product of $\C(\T)$ in terms of its realization as the category of finitely generated left $A$-modules. We have an $(A,A\otimes_R A)$-bimodule $T$, and 
$$
X\otimes Y=T\otimes_{A\otimes_R A}(X\otimes_R Y).
$$
The exactness of this functor on flat objects proved in Proposition \ref{flatrig} is equivalent to the condition that $T$ is projective as a right module (in particular, it is free over $R$). 

We see that the algebra $A$ has a structure of a {\it pseudo-Hopf algebra} over $R$ (\cite{EO}, 4.2). Thus $A_\k$ is a pseudo-Hopf algebra over $\k$ and $A_K:=K\otimes_R A$ is one over $K$. 
We have $\C(\T_\k)=\Rep(A_\k)$. Also it follows that the category $\T_K$ is separated (namely, 
for any splitting object $Q\in \T$ the object $Q_K$ is splitting in $\T_K$) and $\C(\T_K)=\Rep(A_K)$. 

\begin{remark} If $X\in \C(\T)$ is a flat object then $X^*$, ${}^*X$ are both the dual $R$-module of $X$ 
with the action of $a\in A$ given by $a_{X^*}=S(a)_X^*$, $a_{{}^*X}=S^{-1}(a)_X^*$ where $S: A\to A$ is the antipode.  
\end{remark} 

If there exists a pseudo-Hopf algebra $A$ over $R$ which is free of finite rank as an $R$-module and $\C_\k=\Rep(A_\k)$, $\C_K=\Rep(A_K)$, then we will say that $\C_K$ is {\it a lift} of $\C_\k$ over $K$. Note that a specific object of $\C_\k$ may admit more than one lift over $K$, or no such lifts at all. 
  
For example, let $H$ be a (quasi-)Hopf algebra over $R$ which is free of finite rank as an $R$-module (for example, the group algebra $RG$ of a finite group $G$). Then the tensor category $\Rep(H_K)$ is a lift of 
$\Rep(H_\k)$ over $K$. 

We see that in the above setting $\C(\T_K)$ is a lift of $\C(\T_\k)$ over $K$.  

From now on assume that the algebra $A_K$ is split over $K$, or, equivalently, the category $\C_K$ is split (note that this can always be achieved by a finite extension of $R$). For brevity denote $\C(\T),\C(\T)_\k$, $\C(\T)_K$ by $\C,\C_\k,\C_K$. Given a simple object $Y\in \C_K$, it is well known that one can choose (non-uniquely) a flat object $\overline{Y}\in \C$ such that $\overline{Y}_K=Y$. For a simple object $X\in \C_\k$ let $[Y:X]$ be the multiplicity of $X$ in the Jordan-H\"older series of the reduction $\overline{Y}_\k\in \C_\k$. This notation is justified because 
this multiplicity is, in fact, independent on the choice of $\overline{Y}$ (see \cite{Ek}). 
The {\it decomposition matrix} $D$ of $\C$ is the matrix $D$ with 
entries $D_{XY}=[Y:X]$.   

\begin{proposition}\label{cde} (i) (CDE triangle) Let $C_\k,C_K$ be the Cartan matrices 
of $\C_\k,\C_K$, and $D$ be the decomposition matrix of $\C$. Then the decomposition matrix for lifts of projectives in $\C_\k$ into projectives in $\C_K$ 
is $D^T$, and we have $C_\k=DC_KD^T$. 

(ii) ${\rm FPdim}(\C_\k)={\rm FPdim}(\C_K)$;

(iii) If $\T_K$ is complete and for all $i,j$ 
$$
\dim\Hom_{\C_\k}(F_\k(T_i),F_\k(T_j))=\dim\Hom_{\C_K}(F_K(T_i),F_K(T_j))
$$
then $\T_\k$ is complete.\footnote{Note that we automatically have $\dim\Hom_{\C_\k}(F_\k(T_i),F_\k(T_j))\ge \dim\Hom_{\C_K}(F_K(T_i),F_K(T_j))$ since 
$F_K(T_i)$ are lifts of $F_\k(T_i)$ over $K$
(as $A$-modules).} 
\end{proposition} 

\begin{proof} (i) This holds for any flat deformation of finite dimensional algebras over a DVR, \cite{Ek}.  

(ii) Let $\bold d_\k,\bold d_K$ be the vectors of Frobenius-Perron dimensions in $\C_\k,\C_K$ (\cite{EGNO}, Subsections 3.3, 4.5). By (i) we have 
$$
{\rm FPdim}(\C_\k)=\bold d_\k^T C_\k \bold d_\k=\bold d_\k^T DC_K D^T\bold d_\k=
\bold d_K^T C_K \bold d_K={\rm FPdim}(\C_K). 
$$

(iii) We need to show that the functor $F_\k: \T_\k\to \C_\k$ is full. Since it is faithful, it suffices to show that 
$\dim\Hom_{\C_\k}(F_\k(T_i),F_\k(T_j))\le \dim\Hom_{\T_\k}(T_i,T_j)$. 
But 
$$
\dim\Hom_{\T_\k}(T_i,T_j)=\dim\Hom_{\T_K}(T_i,T_j)=\dim\Hom_{\C_K}(F_K(T_i),F_K(T_j)),
$$
which implies the statement.
\end{proof} 

\begin{corollary}\label{semisi} If $\T_K$ is semisimple and 
and for all $i,j$ 
$$
\dim\Hom_{\C_\k}(F_\k(T_i),F_\k(T_j))=\dim\Hom_{\C_K}(F_K(T_i),F_K(T_j))
$$
then $\T_\k$ is complete and the Cartan matrix of $\C_\k$ 
is given by $C=DD^T$. 
\end{corollary} 

\begin{proof} This follows immediately from Proposition \ref{cde}, since the Cartan matrix of $\C_K=\T_K$ is the identity matrix. 
\end{proof} 

\section{Properties of tilting modules for $SL_2(\k)$ and splitting objects of $\T_{n,p}$}\label{s3}

\subsection{Splitting objects in $\T_{n,p}$} 
Let $\T=\T_p$ be the category of tilting $SL_2(\k)-$modules, ${\rm char}(\k)=p>0$. Its indecomposable object of highest weight $m\in \BZ_{\ge 0}$ will be denoted $T_m$. For $r\in \BZ_{\ge 0}$ let $\St_r =T_{p^r-1}$ be the $r$-th Steinberg module; we will also denote 
$\St_1$ as $\St$. For $k\in \BZ_{\ge 1}$ let $\T_{<p^k-1}$ be the additive subcategory of $\T$
generated by $T_m$ with $m<p^k-1$.

Recall the definitions of Example \ref{mainex}(5). Namely, for $n\ge 0$ let $\mathcal{I}_n=\I_{n,p}$ be the thick ideal in $\T=\T_p$ generated by ${\St_n}$. We have $\mathcal{I}_n=\T_{\ge p^n-1}$, the additive subcategory of $\T_p$ spanned by the $T_m$ with $m\ge p^n-1$. Thus the tensor ideal $\langle\mathcal{I}_n\rangle$ consists of morphisms which factor through a direct sum of $T_m$, $m\ge p^n-1$. For $n\ge 1$ define the category $\T_{n,p}:=\T_p/\mathcal{I}_n$. This is a rigid monoidal Karoubian category with indecomposables $T_i$, $0\le i\le p^n-2$.

Let $L_m$ be the simple $SL_2(\k)-$module of highest weight $m\in \BZ_{\ge 0}$. Let $M\mapsto M^{(1)}$ be the Frobenius twist.

We first recall the following standard facts. 

\begin{proposition}\label{standfacts}
(i) (\cite[E.1]{Ja}) $L_a=T_a$ for $0\le a\le p-1$.

(ii) (Steinberg tensor product theorem, \cite[3.17]{Ja}) $L_{a+bp}=L_a\otimes L_b^{(1)}$, $0\le a\le p-1, b\in \BZ_{\ge 0}$.

(iii) (\cite[E.9]{Ja}) One has
\begin{equation} \label{Donkin alg}
T_{a+pb}=T_a\otimes T_b^{(1)}\; \mbox{for}\; p-1\le a\le 2p-2,\; b\in \BZ_{\ge 0}.
\end{equation}
In particular
$\St_r=\St \ot \St_{r-1}^{(1)}$ for $r\in \BZ_{\ge 1}$.
\end{proposition} 

An important role below will be played by the following proposition. 

\begin{proposition}\label{splitmor} Let $k\in \BZ_{\ge 1}$ and $X,Y \in \T_{<p^k-1}$. Then for any morphism $f: X\to Y$ the morphism
$f\ot 1 :X\ot \St_{k-1} \to Y\ot \St_{k-1}$ is split.
\end{proposition}

\begin{proof} We will use the following lemma. 

\begin{lemma}\label{splitmor1} Assume that $k\in \BZ_{\ge 1}$ and $m\le p^k-1$. Then $L_m\ot \St_{k-1}$ is tilting.
\end{lemma}

\begin{proof} 
We use induction in $k$. In the base case $k=1$, $\St_{k-1}=\St_0=\bold 1$ and 
the result follows from Proposition \ref{standfacts}(i). 

Now assume that Lemma \ref{splitmor1} holds for some $k$. Let $m\le p^{k+1}-1$. We write $m=a+bp$ where
$0\le a\le p-1$ and $0\le b\le p^k-1$, so $L_m=L_{a+bp}=L_a\otimes L_b^{(1)}$ by Proposition \ref{standfacts}(ii). We have
$$L_m\ot \St_k=L_a\otimes L_b^{(1)}\ot \St \ot \St_{k-1}^{(1)}=
T_a\ot \St \ot (L_b\ot \St_{k-1})^{(1)}.$$
By the induction assumption $L_b\ot \St_{k-1}$ is tilting, so $\St \ot (L_b\ot \St_{k-1})^{(1)}$ is also
tilting by Proposition \ref{standfacts}(iii). Thus $T_a\ot \St \ot (L_b\ot \St_{k-1})^{(1)}$ is also tilting as a product of two tilting
modules. We have thus proved the induction step and the Lemma follows.
\end{proof}

Now consider the exact sequences
$$0\to \Ker(f)\to X\to {\rm Im}(f)\to 0,$$ 
$$0\to {\rm Im}(f)\to Y\to \Coker(f)\to 0.$$
Let $M$ be one of $\Ker(f), {\rm Im}(f), \Coker(f)$.
The Jordan-Holder factors of $M$ are of the form $L_m, 0\le m<p^k-1$. 
Thus $M\ot \St_{k-1}$ has a filtration with subquotients of the form $L_m\ot \St_{k-1}, 0\le m<p^k-1$ 
which are tilting by Lemma \ref{splitmor1}. Since all $\Ext$s between tilting modules vanish (see \cite[Corollary E.2 (b)]{Ja}) this filtration splits  and $M\ot \St_{k-1}$ is tilting. Thus the sequences
$$0\to \Ker(f)\ot \St_{k-1}\to X\ot \St_{k-1} \to {\rm Im}(f)\ot \St_{k-1}\to 0,$$ 
$$0\to {\rm Im}(f)\ot \St_{k-1} \to Y\ot \St_{k-1} \to \Coker(f)\ot \St_{k-1}\to 0$$
also split and the result follows.
\end{proof}

\begin{remark} The inequality for $m$ in Lemma \ref{splitmor1} is the best possible: the module \linebreak $L_{p^k}\ot \St_{k-1}$
is not tilting. Namely this module has highest weight $p^k+p^{k-1}-1$ and dimension $2p^{k-1}$ (since
$L_{p^k}=L_1^{(k)}$); this dimension is less than the dimension of the Weyl module $W_{p^k+p^{k-1}-1}$, which equals $p^k+p^{k-1}$.
\end{remark}

\subsection{The functor $\T_{<p^n-1}\to \T_{n,p}$ is fully faithful} 

\begin{proposition}\label{fullfaith} The composite functor $\T_{<p^n-1}\hookrightarrow \T_p \to \T_{n,p}$ is an equivalence of additive categories. 
\end{proposition}

\begin{proof} By definition this functor is essentially surjective and full, so it remains to show that it is faithful. Thus our job is to show that 
no nonzero morphism in $\T_{<p^n-1}$ factors through $\I_n=\T_{\ge p^n-1}$. Assume that this is not
the case and let $\widetilde f\in \Hom(T_m, T_k)$ be a morphism in $\langle\mathcal{I}_n\rangle$ with $0\le m,k<p^n-1$. Then 
the morphism $f\in \Hom(\bold 1, T_m^*\ot T_k)$ which is the image of $\widetilde f$ under the canonical
isomorphism $\Hom(T_m, T_k)\simeq \Hom(\bold 1, T_m^*\ot T_k)$ is also in $\langle\mathcal{I}_n\rangle$. Thus there exists
an object $T_s$ with $s\ge p^n-1$ and morphisms $\bold 1 \to T_s$ and $T_s\to T_m^*\ot T_k$ such that
their composition is not zero. 

\begin{lemma}\label{lem2} (see for example \cite{C1}, Lemma 5.3.3 and references therein) 
Assume $\Hom(\bold 1 ,T_s)\ne 0$ in $\T_p$. Then $s=2p^l-2$ for some $l\in \BZ_{\ge 0}$ 
and the socle of $T_s$ is one-dimensional.
\end{lemma}

\begin{proof} For $0\le s\le 2p-2$ the lemma follows from the description of tilting modules
in \cite[E.1]{Ja}. Now assume that $s>2p-2$ and write $s=a+bp$ where $p-1\le a\le 2p-2$ and
$b\in \BZ_{\ge 0}$. Then by \cite[Lemma E.9]{Ja} we have $T_s=T_a\otimes T_b^{(1)}$. 
Thus $\Hom(\bold 1 ,T_s)=\Hom(T_a^* ,T_b^{(1)})$. Thus $T_a^*$ must have a quotient on which the first
Frobenius kernel of $SL_2(\k)$ acts trivially. Again by the description in \cite[E.1]{Ja} we see that $a=2p-2$ and the quotient
is the object $\bold 1=L_0$. Thus $\Hom(\bold 1, T_b^{(1)})\ne 0$, so $\Hom(\bold 1,T_b)\ne 0$. By induction
we get that $b=2p^{l-1}-2$, so $s=2p^{l}-2$ as required. 

The module $T_{2p^l-2}$ is a lift
to $SL_2(\k)$ of the injective hull of the trivial module over the $l$-th Frobenius kernel, see \cite[E.9]{Ja}; this
implies the statement about the socle.
\end{proof}

Lemma \ref{lem2} implies that the map $T_s\to T_m^*\ot T_k$ is injective and $s=2p^l-2$ with $l\ge n$. This
is impossible as all weights of $T_m^*\ot T_k$ have absolute value $\le m+k<2(p^n-1)$.
\end{proof} 

\subsection{Tensor ideals in $\T_p$}

Lemma \ref{lem2} can be used to establish the classification of tensor ideals
in $\T$.

\begin{proposition}\label{tensid} (\cite{C1}, Theorem 5.3.1) The only nonzero tensor ideals in $\T_p$ are the ideals $\langle\mathcal{I}_k\rangle$, $k\ge 0$.
\end{proposition} 
 
\begin{proof} 
Let $I$ be a nonzero tensor ideal in $\T_p$. Let $f\in I$ be a nonzero morphism, $f: X\to Y$. 
Then $f$ defines a nonzero element $\widetilde{f}\in \Hom(\bold 1,X\otimes Y^*)$ which belongs to $I$. 
Thus there exists a nonzero morphism $g: \bold 1\to T_m$ for some $m\ge 0$ which belongs to $I$. 
Let $m$ be the smallest such number. By Lemma \ref{lem2}, $m=2p^k-2$ for some $k\ge 0$, and  
$I$ contains $\Hom(\bold 1 ,T_{2p^k-2})$.
Since the map $\Hom(\bold 1 ,T_{2p^l-2})\to \Hom(\bold 1 ,\St_l\ot \St_l^*)$ induced by the inclusion of the
direct summand $T_{2p^l-2}\hookrightarrow \St_l\ot \St_l^*$ is an isomorphism  (as both spaces are 1-dimensional and the map is injective),    we see that $1_{\St_k}$ is contained
in $I$, i.e., $\langle\mathcal{I}_k\rangle\subset I$. This implies that $1_{\St_l}$ is contained in $I$ for all $l\ge k$. 
Thus, $\Hom(\bold 1,T_{2p^l-2})$ is contained in $I$ if and only if $l\ge k$, i.e., if and only if it is contained in $\langle\mathcal{I}_k\rangle$. 
Hence $I=\langle\mathcal{I}_k\rangle$, since a tensor ideal is uniquely determined by its components in $\Hom (\bold 1, T)$ where $T$ runs over indecomposable objects.
\end{proof} 

\subsection{Separatedness of $\T_{n,p}$ and its ideal of splitting objects} 

\begin{corollary}\label{separa} The category $\T_{n,p}$ is separated. 
\end{corollary} 

\begin{proof} By Proposition \ref{splitmor}, $\St_{n-1}=T_{p^{n-1}-1}$ is a splitting object 
in $\T_{n,p}$.

Also we claim that if $\St_{n-1}\otimes X=0$ for an object $X\in \T_{n,p}$ then $X=0$. Indeed, in this case the composite map $\St_{n-1}^*\otimes \St_{n-1}\to {\bold 1}\to X\otimes X^*$ is zero, and the evaluation map $\St_{n-1}^*\otimes \St_{n-1}\to \bold 1$ is an epimorphism by Proposition \ref{fullfaith}. Hence the coevaluation map 
$\bold 1\to X\otimes X^*$ is zero.

This implies the statement by Proposition \ref{tensid}.  
\end{proof} 

Let $\S_{n,p}$ be the ideal of splitting objects of $\T_{n,p}$. 

\begin{proposition}\label{idspli} We have $\S_{n,p}=\mathcal{I}_{n-1}/\mathcal{I}_n$. 
\end{proposition} 

\begin{proof} By Proposition \ref{tensid} and Corollary \ref{separa}, $\S_{n,p}=\mathcal{I}_m/\mathcal{I}_n$ for some $m<n$. 
Thus $\S_{n,p}$ is indecomposable. Hence the claim follows from Corollary \ref{separa} and Proposition \ref{uniqmin}.
\end{proof} 

\subsection{Donkin's recursive formula and auxiliary properties of tilting modules} \label{Donkin1}

We recall Donkin's recursive formula for the tilting $SL_2(\k)-$modules, see 
\cite[Example 2, p. 47]{D}.  

Let $W_a$ denote the standard (i.e., Weyl) $SL_2(\k)$--module with highest weight $a$. 
Recall that the modules $T_i=W_i=L_i$ are simple for $0\le i\le p-1$, and $T_0=\bold 1$. 
Let $V:=L_1$. Recall the following result. 

\begin{lemma}\label{smalltilt} (\cite[E.1]{Ja}) (i)  One has
\begin{equation} \label{eq2}
V\ot T_i=T_{i-1}\oplus T_{i+1}\; \mbox{if}\; 1\le i\le p-2.
\end{equation}

(ii) For $p\le a\le 2p-2$ the module $T_a$ has a submodule $W_a$ such that $T_a/W_a=W_{2p-2-a}$; thus its composition series is $T_a=[L_{2p-2-a},V^{(1)}\otimes L_{a-p},L_{2p-2-a}]$. 
\end{lemma} 

Using Proposition \ref{standfacts}(iii), 
this implies 
\begin{proposition}\label{tenst1}  
(i) for $p>2$ we have
\begin{equation} \label{eq1}
V\ot T_{a+pb}=\left\{ 
\begin{array}{ccc} 
T_p\ot T_b^{(1)}=T_{p+pb}&\mbox{if}&a=p-1,\\
2T_{p-1+pb}\oplus T_{p+1+pb}&\mbox{if}&a=p,\\
T_{a-1+pb}\oplus T_{a+1+pb}&\mbox{if}&p<a<2p-2,\\
T_{2p-3+pb}\oplus T_{p-1}\ot (V\ot T_b)^{(1)}&\mbox{if}&a=2p-2.
\end{array}
\right.
\end{equation}

(ii) For $p=2$ we have 
\begin{equation} \label{eq1new}
V\ot T_{a+2b}=\left\{ 
\begin{array}{ccc} 
T_2\ot T_b^{(1)}=T_{2+2b},\ a=1\\
2T_{1+2b}\oplus V\ot (V\ot T_b)^{(1)}, a=2.
\end{array}
\right.
\end{equation}
\end{proposition} 

We also have the following result describing tensor products of tilting modules with small highest weights. 

\begin{proposition}\label{tenprotilt}  
The tensor product $T_m\otimes T_r$ for $1\le m,r\le p-1$ is given by the formula 
$$
T_m\otimes T_r=\coplus_{|m-r|\le k\le 2(p-2)-m-r}T_k\oplus \coplus_{p-1\le k\le m+r}T_k,
$$
if $m+r\ge p$, and 
$$
T_m\otimes T_r=\coplus_{|m-r|\le k\le m+r}T_k
$$
if $m+r<p$, where the summation is over integers $k$ such that $m+r-k$ is even. 
\end{proposition} 

\begin{proof} This follows immediately from Lemma \ref{smalltilt} and Equations \eqref{eq1},\eqref{eq1new}
\end{proof} 

\begin{proposition} \label{Tm appears}
Assume that $0\le i\le p-1$ and $m=pr+p-2$ for some $r\in \BZ_{\ge 0}$.
Then $T_i\otimes T_k$ contains $T_m$ as a summand if and only if $k=m-i$, in which case $T_m$ appears with multiplicity 1.
\end{proposition}

\begin{proof} For $p=2$ the statement is clear from \eqref{eq1new}, so we may assume that $p>2$. It is clear that $T_i\ot T_{m-i}$ contains $T_m$ with multiplicity 1, so we only need to prove the ``only if'' statement.  
The case $i=0$ is clear as $T_0=\bf 1$. For $i\ge 1$ we proceed 
by induction in $i$. The base case $i=1$ is clear by inspection of \eqref{eq1} and \eqref{eq2}. Suppose $1<i\le p-1$ and the statement is known for $i-1$, and let us prove it for $i$. So by \eqref{eq2}, $T_{i-1}\otimes V\otimes T_k$ contains $T_m$. By the induction assumption, $V\otimes T_k$ must contain $T_{m-i+1}$. But $m-i+1\equiv -i-1\not \equiv -1\pmod{p}$. 
Thus by \eqref{eq1} we have either $k=m-i$ or $k=m-i+2$ and $V\otimes T_k$ contains $T_{m-i+1}$ with multiplicity 1. For the sake of contradiction assume that $k=m-i+2$. Then $T_{i-1}\otimes V\otimes T_k=(T_{i-2}\oplus T_i)\ot T_k$ contains $T_m$ with multiplicity 1. But $T_{i-2}\otimes T_k$
already contains $T_{i-2+k}=T_m$, so $T_i\otimes T_k$ cannot contain $T_m$ (as the multiplicity is 1). Thus this case is impossible and $k=m-i$, as desired. 
\end{proof}

\subsection{Objects $Y$ in symmetric tensor categories 
with $\wedge^2Y=\one$} 

Let $\C$ be a symmetric tensor category over an algebraically closed field $\k$ of characteristic $p\ge 0$. 
We denote by $c_{X,Y}: X\otimes Y\to Y\otimes X$ the braiding of $\C$. 
Recall that for every $Y\in \C$ we have an injection $\wedge^2Y\hookrightarrow Y\otimes Y$ as the image of $c-1$, where $c:=c_{Y,Y}$, and a surjection $Y\otimes Y\twoheadrightarrow \wedge^2Y$, the
quotient map by the kernel of $c-1$. The identification between these two copies of $\wedge^2 Y$ is implemented by the map $c-1: Y\otimes Y\to Y\otimes Y$. 
Note that for a filtered object $Y\in \C$ we have an injection 
$\wedge^2{\rm gr}Y\hookrightarrow {\rm gr}(\wedge^2Y)$ (an isomorphism if $p\ne 2$). Hence if $Y=[W,Z]$ is an extension of $W$ by $Z$ then the composition series of $\wedge^2Y$ contains the union of the composition series of $\wedge^2 W,Z\otimes W,\wedge^2 Z$ (if $p\ne 2$, they coincide). 

\begin{lemma}\label{lemm0} (i) If $\wedge^2Y=L$ is invertible then either $Y$ is simple or it is an extension $[\chi,\psi]$ of an even invertible object $\chi$ (i.e., such that $c_{\chi,\chi}=1$) by an even invertible object $\psi$, such that $\chi\otimes \psi\cong L$.

(ii) If in (i) $Y$ is not simple then $\wedge^3Y=0$. 
\end{lemma} 

\begin{proof} (i) Suppose $Y$ is not simple and is an extension of $\chi$ by $\psi$. 
Then $\wedge^2Y=L$ has composition series containing the union of composition series of $\wedge^2\chi,\chi\otimes \psi,\wedge^2\psi$. Thus $\chi\otimes \psi=L$ and $\wedge^2\chi=0$, $\wedge^2\psi=0$, which proves (i). 

(ii) We have a filtration on $Y$ with ${\rm gr}Y=\psi\oplus \chi$, thus
$\wedge^3({\rm gr}Y)=0$. If $p\ne 2$ then 
${\rm gr}\wedge^3Y$ is a quotient of $\wedge^3{\rm gr}Y$ 
(and they coincide if $p\ne 3$), so $\wedge^3Y=0$. 
On the other hand, if $p=2$ then 
note that ${\rm gr}(\wedge^2 Y)=\wedge^2({\rm gr}Y)=L$, hence 
${\rm gr}(\wedge^2 Y\otimes Y)=\wedge^2({\rm gr}Y)\otimes {\rm gr}Y$. 
But for any object $Z\in \C$, $\wedge^3Z$ is canonically a direct summand 
in $\wedge^2Z\otimes Z$ (the image of the idempotent $1+(123)+(132)$). Hence ${\rm gr}(\wedge^3 Y)=\wedge^3{\rm gr}Y=0$, i.e., 
$\wedge^3Y=0$. 
\end{proof} 

\begin{remark} $Y=[\chi,\psi]$ implies that $\wedge^2Y$ is invertible only in characteristic $p\ne 2$. For example, let $P$ be the projective cover of $\one$ in the category $\C_1$ of \cite{BE} (see also \cite{Ven}, Subsection 1.5). Then $P=[\one,\one]$ but $\wedge^2P=\one\oplus \one\ne \one$. 
\end{remark} 

\begin{definition} Let us say that an object $Y\in \C$ is a {\it fermion}
if $p\ne 2$, $Y\otimes Y\cong \one$, and $c_{Y,Y}=-1$. 
\end{definition}
 
In other words, a fermion is a simple object generating 
${\rm sVec}$.

\begin{proposition}\label{prop1} Let $Y\in \C$ and $\wedge^2 Y=\one$. Then the following conditions are equivalent: 

(i) $Y$ is invertible; 

(ii) $Y$ is a fermion; 

(iii) $\wedge^3 Y\ne 0$. 
\end{proposition} 

\begin{proof} (i)$\implies$(ii). 
We have $Y^{\otimes 2}=\wedge^2Y=\one$, so $c=c_{Y,Y}$ is a scalar such that 
$c^2=1$. Moreover, since $\wedge^2 Y\ne 0$, we have $c\ne 1$. 
Thus $c=-1$ and $p\ne 2$. 

(ii)$\implies$(iii) is obvious (we have $\wedge^3 Y=Y$). 

(iii)$\implies$(i). By Lemma \ref{lemm0}(ii), $Y$ is simple. Since 
$\wedge^2 Y\otimes Y=Y$, we must have $\wedge^3 Y=Y$, hence 
$\wedge^2Y\otimes Y=\wedge^3Y$. Thus the morphism $1+(23)$ vanishes 
on $\wedge^2Y\otimes Y={\rm Im}(1-(12))$. This means that on $Y^{\otimes 3}$ (hence on $Y^{\otimes m}$ for any $m\ge 3$) we have 
$(1+(23))(1-(12))=0$. Conjugating this equation by $S_m$, we get that on $Y^{\otimes m}$, $(1+(jk))(1-(ij))=0$, i.e., 
\begin{equation}\label{eqq}
1+(jk)=(1+(jk))(ij). 
\end{equation} 
for any distinct $1\le i,j,k\le m$. 

Assume that $p\ne 3$. Then, adding the equations \eqref{eqq} for 
$ijk=123, 231, 312$, we get that $3(1-(123))=0$, hence 
$(123)=1$ on $Y^{\otimes 3}$. 
This implies that $c_{Y\otimes Y,Y\otimes Y}=(13)(24)=(134)(234)$ acts trivially on $Y^{\otimes 4}$. Hence $\wedge^2(Y^{\otimes 2})=0$. So by Lemma 3.2 of \cite{BE}, 
$Y^{\otimes 2}$ is invertible. Hence $Y$ is invertible, as desired. 

It remains to consider the case $p=3$. 
From \eqref{eqq}, on $Y^{\otimes 4}$ we have 
$$
(1+(12))(1+(34))(13)(24)=(1+(12))(1+(34)).
$$
Since $(13)(24)$ commutes with $(1+(12))(1+(34))$, this implies that
$$
(1-(13)(24))(1+(12))(1+(34))=0.
$$
This means that $c_{S^2Y,S^2Y}=1$, so $\wedge^2(S^2Y)=0$. By Lemma 3.2 of \cite{BE}, this means that 
$S^2Y$ is zero or invertible. 

Let $d=\dim Y$. Then $\dim \wedge^2 Y=
\frac{d(d-1)}{2}=1$, so $d=2$ and $\dim S^2Y=\frac{d(d+1)}{2}=0$. 
Thus $S^2Y$ cannot be invertible, so 
$S^2Y=0$. Thus $Y\otimes Y=\wedge^2Y=\one$, i.e. $Y$ is invertible, as claimed. 
\end{proof} 

\subsection{Temperley-Lieb objects in monoidal categories} 
Let $\C$ be a monoidal category over an algebraically closed field $\k$. 

\begin{definition} {\it A (classical) Temperley-Lieb (TL) object} in $\C$ is an object $X\in \C$ admitting a left dual $X^*$ and equipped with an isomorphism $\phi: X\to X^*$ such that $\Tr((\phi^{-1})^*\circ \phi)=-2$. 
\end{definition} 

The maps $\phi,\phi^{-1}$ give rise to maps $\alpha: \one\to X\otimes X$, 
$\beta: X\otimes X\to \one$, which give $X$ a structure of the left dual (and hence also the right dual) of itself. Thus a TL object is canonically self-dual. We also have 
$\beta\circ \alpha=-2$. 

Now let $\C$ be symmetric.

\begin{definition} A TL object $X\in \C$ is said to be {\it symmetric} if 
 $c=c_{X,X}$ is given by the formula 
\begin{equation}\label{braideq}
c=1+\alpha\circ \beta. 
\end{equation} 
\end{definition} 

Note that if $X$ is a symmetric TL object then 
$$
c\circ \alpha=(1+\alpha\circ \beta)\circ \alpha=\alpha-2\alpha=-\alpha.
$$
Similarly, $\beta\circ c=-\beta$. Thus $\phi^*=-\phi$. So 
$$
-2=\Tr((\phi^{-1})^*\circ \phi)=\Tr(-1_X)=-\dim X,
$$
hence $\dim X=2$. 

\begin{example}\label{examm} 1. $X=0$ is a (symmetric) TL object iff $p=2$. 

2. A fermion is a symmetric TL object iff $p=3$. 

2. Let $V=T_1$ be the 2-dimensional tautological representation of $SL_2(\k)$, with standard basis $x,y$. We have an isomorphism $\phi: V\to V^*$ given by $\phi(x)=y^*$, $\phi(y)=-x^*$. Thus $\phi^*=-\phi$, so $(V,\phi)$ is a TL object (as $\dim V=2$). Moreover, $\alpha\circ \beta$ vanishes on $x\otimes x$ and $y\otimes y$ 
and maps $x\otimes y$ to $y\otimes x-x\otimes y$ and $y\otimes x$ to 
$x\otimes y-y\otimes x$, so we have $\alpha\circ \beta=c-1$. Thus $(V,\phi)$ is a symmetric TL object. 
\end{example} 

Recall that $\T_p=\T_p(\k)$ denotes the classical Temperley-Lieb category, i.e., the category of tilting modules for $SL_2(\k)$ (if $p=0$, it coincides with ${\rm Rep}(SL_2(\k)$). Example \ref{examm} shows that $V$ is a symmetric TL object both in $\T_p(\k)$ and in ${\rm Rep}(SL_2(\k))$.

\begin{proposition}\label{uni} 
(i) If $\C$ is a Karoubian monoidal category over $\k$ then 
the category of monoidal functors $F: \T_p\to \C$ is equivalent to the category of TL objects in $\C$, via $F\mapsto F(V)$. 

(ii) If $\C$ is a Karoubian symmetric monoidal category over $\k$ then the equivalence of (i) restricts to an equivalence between the category of symmetric monoidal functors $F: \T_p\to \C$ and the category of symmetric TL objects in $\C$. 
\end{proposition} 

\begin{proof} (i) was proved in \cite{O1}, Theorem 2.4. (ii) follows from (i) using Example \ref{examm}. Indeed, for a functor $F$ to be symmetric, it is necessary and sufficicient that it preserve the braiding on $V\otimes V$, which is equivalent to equation \eqref{braideq}. 
\end{proof} 

For $p>0$, recall that $\T_{n,p}=\T_{n,p}(\k)$ denotes the quotient of $\T_p=\T_p(\k)$ by the tensor ideal generated by the $n$-th Steinberg module ${\rm St}_n=T_{p^n-1}$ (we agree that $\T_{\infty,p}(\k)=\T_p(\k)$ for any $p\ge 0$).  Then the image $V_n$ of $V=T_1$ in $\T_{n,p}$ is a symmetric TL object. 

Let $Q_{n,p}$ be the $p^n-1$-th {\it Chebyshev polynomial of the second kind} defined by the formula
\begin{equation}\label{cheb} 
Q_{n,p}(2\cos x)=\frac{\sin(p^nx)}{\sin(x)}. 
\end{equation} 
Write $Q_{n,p}=Q_{n,p}^+-Q_{n,p}^-$, collecting the terms with positive coefficients into $Q_{n,p}^+$ and the terms with negative coefficients into $Q_{n,p}^-$. In the category $\T_p$ we have a (non-unique) isomorphism ${\rm St}_n\oplus Q_{n,p}^-(V)\cong Q_{n,p}^+(V)$ which restricts to a split injection $\iota_n: Q_{n,p}^-(V)\hookrightarrow Q_{n,p}^+(V)$. This injection defines an isomorphism $Q_{n,p}^-(V_n)\cong Q_{n,p}^+(V_n)$. 

Now let $\C$ be a symmetric tensor category over $\k$ of characteristic $p>0$ and $X$ be a symmetric TL object of $\C$. Then for each $n\ge 0$ Proposition \ref{uni}(ii) gives rise to a split injection $\iota_n: Q_{n,p}^-(X)\hookrightarrow Q_{n,p}^+(X)$. 

\begin{definition} We say that $X$ has {\it degree} $n\ge 1$ if $n$ is the smallest integer for which the split injection $\iota_n$ is an isomorphism. If there is no such $n$, we say that $X$ is of degree $\infty$ (in particular, the degree is always $\infty$ in characteristic zero). 
\end{definition} 

\begin{proposition}\label{uni11} If $\C$ is a Karoubian symmetric monoidal category and $1\le n\le \infty$, then the category of symmetric monoidal functors $F: \T_{n,p}\to \C$ is equivalent to the category of symmetric TL objects $X\in \C$ of degree $\le n$, via $X=F(V_n)$. Moreover, $F$ is faithful if and only if the degree of $X$ is exactly $n$. 
\end{proposition} 

\begin{proof} This follows from Propositions \ref{uni}(ii) and \ref{tensid}. 
\end{proof} 

\subsection{Symmetric Temperley-Lieb objects in symmetric tensor categories} 

From now on assume that $\C$ is an (abelian) symmetric tensor category.  
Let $X\in \C$ be a symmetric TL object. 

\begin{lemma}\label{lemm1} (i) The maps $\alpha$, 
$\beta$ factor through $\wedge^2 X$ and, if $X\ne 0$, 
induce isomorphisms $\one\cong \wedge^2X$, $\wedge^2X\cong \one$. 
Moreover, these isomorphisms are mutually inverse.  

(ii) If $p\ne 3$ then $\wedge^3X=0$. 
\end{lemma} 

\begin{proof} We may assume that $X\ne 0$. 

(i) Since $\alpha$ is a coevaluation map, it is injective. 
Similarly, $\beta$ is an evaluation map, hence surjective. 
Thus by \eqref{braideq} we have ${\rm Im}(\alpha)={\rm Im}(\alpha\circ \beta)={\rm Im}(c-1)=\wedge^2 X$. Hence 
$\alpha$ gives an isomorphism $\one\to \wedge^2 X$, as claimed. 
The second statement is proved similarly. The last statement follows immediately from \eqref{braideq}. 

(ii) Suppose $\wedge^3X\ne 0$. Then by Proposition \ref{prop1}, 
$X$ is a fermion and $\dim X=-1$. In particular, outside of characteristic $3$ we have
$\dim X\ne 2$, so $X$ is not a symmetric TL object. 
\end{proof} 

\begin{proposition}\label{prop2} $X\in \C$ is a symmetric TL object if and only if 
$\wedge^2X=\one$ and $\wedge^3X=0$, or $X$ is a fermion in characteristic $3$ (in which case $\wedge^2X=\one$ but $\wedge^3X=X\ne 0$), or $X=0$ in characteristic $2$.  
\end{proposition} 

Note that by Proposition \ref{prop1}, the condition $\wedge^3X=0$ in Proposition \ref{prop2} can be replaced with the condition that $X$ is not invertible. 

\begin{proof} Suppose $X\ne 0$ is a symmetric TL object. By Lemma \ref{lemm1}(i), 
we have $\wedge^2X=\one$. If $\wedge^3X\ne 0$ then by Proposition \ref{prop1}, 
$X$ is a fermion and by Lemma \ref{lemm1}(ii), $p=3$. 
 
For the converse
 we may suppose that $\wedge^2X=\one$ and $\wedge^3X=0$. If $X$ is not simple 
then by Lemma \ref{lemm0} $X=[\chi,\chi^*]$ with even invertible $\chi$. Thus 
$X$ is a symmetric TL object: we have $\dim X=2$ the natural map $\alpha: \one\to \wedge^2X$ gives rise to an isomorphism $\phi: X\to X^*$ such that $\phi^*=-\phi$. 

It remains to consider the case when $X$ is simple. 
The isomorphism $\wedge^2 X\cong \one$ defines a morphism $\phi: X\to X^*$ such that $\phi^*=-\phi$. Moreover, since $X$ is simple and $\phi\ne 0$, it is an isomorphism. 

It remains to show that $d:=\dim X$ equals $2$. 
For $p\ne 2,3$, we have $\frac{d(d-1)}{2}=1$, $\frac{d(d-1)(d-2)}{6}=0$, so $d=2$. For $p=3$ we have $\frac{d(d-1)}{2}=1$, so $d^2-d-2=0$, hence again $d=2$. Finally, for $p=2$ 
we have $\dim(\wedge^3X)=\dim(\wedge^2X)d$ (see for example \cite{EHO}), which implies 
that $d=0=2$. 
\end{proof} 

\begin{remark} The zero objects in characteristic $2$ and fermions in characteristic $3$ are exactly the symmetric TL objects of degree $1$ in characteristics $2,3$. 
\end{remark} 

\begin{corollary}\label{corro} Let $\C$ be a symmetric tensor category over $\k$. Then: 

(i) (\cite{BE}, Proposition 3.4) The category of symmetric monoidal functors $F: \T_p\to \C$ is equivalent to the category of objects $X\in \C$ with an isomorphism $\wedge^2X\cong \one$ such that 
$\wedge^3 X=0$ or $X$ is a fermion for $p=3$, or $X=0$ for $p=2$.  

(ii) (\cite{BE}, Corollary 3.5)\footnote{In the published version of \cite{BE}, in Proposition 3.4 and Corollary 3.5, the cases of odd invertible object $X$ in characteristic $3$ and $X=0$ in characteristic $2$ are missed. This inaccuracy is corrected in the latest arXiv version of \cite{BE}.}
 The category of symmetric monoidal functors $F: \T_{n,p}\to \C$ is equivalent to the category of objects $X\in \C$ as in (i) such that $Q_{n,p}^+(X)\cong Q_{n,p}^-(X)$, via $X=F(V_n)$. Moreover, $F$ is faithful if and only if $Q_{n-1,p}^+(X)\ncong Q_{n-1,p}^-(X)$. 
\end{corollary} 

\begin{proof} (i) This follows from Proposition \ref{uni}(ii) and Proposition \ref{prop2}. 

(ii) This follows from Proposition \ref{uni11} and Proposition \ref{prop2}. 
\end{proof} 

\section{The Verlinde categories $\Ver_{p^n}$}\label{s4}

\subsection{The definition of the Verlinde categories $\Ver_{p^n}$, $\Ver_{p^n}^+$} 
Let $p$ be a prime. Recall that ${\rm Ver}_p$ denotes the Verlinde category, which is the semisimplification of the category of modular representations $\Rep_\k(\Bbb Z/p)$ (\cite{GK,GM,O}). 

Let $n\ge 1$ be an integer. By Corollary \ref{separa}, the category $\T_{n,p}$ is separated, and by Proposition \ref{idspli} its ideal of splitting objects $\S_{n,p}=\mathcal{I}_{n-1}/\mathcal{I}_n$ is spanned by the objects $T_i$, $p^{n-1}-1\le i\le p^n-2$ and is indecomposable. Thus, using the main construction of Section 2, we can define finite symmetric tensor categories 
$$
\Ver_{p^n}=\Ver_{p^n}(\k):=\C(\T_{n,p}). 
$$ 
If $p^n>2$, we can also consider subcategories $\T_{n,p}^+\subset \T_{n,p}$ spanned by $T_i$ with even $i$, and the 
corresponding subcategories $\Ver_{p^n}^+:=\C(\T_{n,p}^+) \subset \Ver_{p^n}$. This makes $\Ver_{p^n}$ into faithfully $\Bbb Z/2$-graded tensor categories with trivial components $\Ver_{p^n}^+$ (\cite{EGNO}, Subsection 4.14). 

\begin{definition} We will call the categories $\Ver_{p^n}$, $\Ver_{p^n}^+$ the {\bf Verlinde categories}.
\end{definition} 

The motivation for this terminology will be clarified in Subsection \ref{lifttochar0}. 

\subsection{$\T_{n,p}$ is complete} 

\subsubsection{The completeness theorem} \label{compthe} 

\begin{theorem} \label{CartanFP0}
(i) The category $\T_{n,p}$ is complete, so it is a full rigid monoidal Karoubian subcategory of $\Ver_{p^n}=\C(\T_{n,p})$. 

(ii) $\Ver_{p^n}$ is the abelian envelope of $\T_{n,p}$. 
\end{theorem} 

\begin{proof} (i) For $p=2$ the result follows from \cite{BE}, so we may assume that $p\ge 3$. 

We need to show that for any $X,Y\in \T_{n,p}$, the map $\Hom(X,Y)\to \Hom(F(X),F(Y))$ 
is an isomorphism, or, equivalently, that the map $\Hom(\bold 1,Y\otimes X^*)\to \Hom(\one,F(Y)\otimes F(X)^*)=\Hom(\one,F(Y\otimes X^*))$ is an isomorphism. Thus, it suffices to show that 
for any $i\in [0,p^n-2]$ the map $\Hom(\bold 1,T_i)\to \Hom(\one, F(T_i))$ is an isomorphism. 
For $i\ge p^{n-1}-1$ this follows from Proposition \ref{fullness}, so it suffices to treat the case 
$i\le p^{n-1}-2$. 

Let $S=T_{p^{n-1}-1}$. To compute $\Hom(\one,F(T_i))$, we will use the presentation of $\one$ as the cokernel of the map 
$$
\tau: S^*\otimes S\otimes S^*\otimes S\to S^*\otimes S.
$$
Note that $S\cong S^*$. Hence we get that 
$$
\Hom(\one,F(T_i))={\rm Ker}(\Hom(S^{\otimes 2},T_i)\to \Hom(S^{\otimes 4},T_i)).
$$
Thus it suffices to check that the sequence 
$$
\Hom_{\T_{n,p}}(\bold 1,T_i)\to \Hom_{\T_{n,p}}(S^{\otimes 2},T_i)\to \Hom_{\T_{n,p}}(S^{\otimes 4},T_i)
$$ 
is exact. This sequence can be rewritten as 
\begin{equation}\label{seq1}
\Hom_{\T_{n,p}}(\bold 1,T_i)\to \Hom_{\T_{n,p}}(S^{\otimes 2},T_i)\to \Hom_{\T_{n,p}}(S^{\otimes 2},T_i\otimes S^{\otimes 2}).
\end{equation}
Let us now consider the same sequence in $\T_p$:
\begin{equation}\label{seq2}
\Hom_{\T_{p}}(\bold 1,T_i)\to \Hom_{\T_{p}}(S^{\otimes 2},T_i)\to \Hom_{\T_{p}}(S^{\otimes 2},T_i\otimes S^{\otimes 2}),
\end{equation}
which is clearly exact. But since $i\le p^{n-1}-2$ and $p\ge 3$, 
the sequence \eqref{seq2} actually belongs to $\T_{<p^n-1}$, hence is isomorphic 
to \eqref{seq1} by Proposition \ref{fullfaith}. Thus \eqref{seq1} is exact, as required.\footnote{This method of proof was suggested by K. Coulembier.} 

(ii) now follows from Theorem \ref{main1}. 
\end{proof}

Another proof of Theorem \ref{CartanFP0}(i) which works in any positive characteristic and does not use 
\cite{BE} is given in Subsection \ref{altproof} below. 

Theorem \ref{CartanFP0} allows one to easily compute the dimension of the space of invariants in $V^{\otimes r}$, where $V:=F(T_1)\in \Ver_{p^{n}}$. 
It is clear that if $r$ is odd then this space is zero, so it suffices to compute
the invariants in $V^{\otimes 2m}$. Let $d_{mn}=\dim \Hom(\one,V^{\otimes 2m})$. 

\begin{proposition}\label{inva} 
Let $f_n(z)=\sum_{m\ge 0}d_{mn}z^m$.
Then 
\begin{equation}\label{fnz}
f_n(z)=\frac{(t+t^{-1})(t^{p^n-1}-t^{-p^n+1})}{t^{p^n}-t^{-p^n}},
\end{equation}
where $z=(t+t^{-1})^{-2}$. 
\end{proposition} 

\begin{proof} By Theorem \ref{CartanFP0}, $d_{mn}=\dim \Hom(\bold 1,T_1^{\otimes 2m})$. 
The rest of the argument is the same as in \cite{BE}, proof of Proposition 3.9. 
\end{proof} 

{\bf Remark.} Proposition \ref{inva} generalizes \cite{BE}, Proposition 3.3. Note that there is a misprint in \cite{BE}, Proposition 3.3 and the proof of Proposition 3.9: $2^n$ should be replaced by $2^{n+1}$.  

\subsubsection{An alternative proof of Theorem \ref{CartanFP0}(i)}\label{altproof} Let us give another (somewhat more complicated) proof of Theorem \ref{CartanFP0}(i). The advantage of this proof is that it works for all positive characteristics and does not use \cite{BE}. 

We need to show that the evaluation map $S^{\otimes 2}\to \bold 1$ induces an isomorphism of functors
$$
\Hom_{\T_{n,p}}(\bold 1,?)\simeq {\rm Ker}(\Hom_{\T_{n,p}}(S^{\otimes 2},?)\to \Hom_{\T_{n,p}}(S^{\otimes 4},?)).
$$
Recall that by Proposition \ref{fullfaith} we have an equivalence of additive categories $\T_{n,p}\simeq \T_{<p^n-1}$, so we need to establish an
isomorphism 
$$
\Hom_{\T_{<p^n-1}}(\bold 1,?)\simeq {\rm Ker}(\Hom_{\T_{<p^n-1}}(A,?)\to \Hom_{\T_{<p^n-1}}(B,?))
$$
where $A,B\in \T_{<p^n-1}$ are objects corresponding to $S^{\otimes 2}$ and $S^{\otimes 4}$ under
the equivalence of Proposition \ref{fullfaith}. Let $C$ be the cokernel of the map $B\to A$ in the category
of $SL_2(\k)-$modules. Then the functor ${\rm Ker}(\Hom (A,?)\to \Hom (B,?))$ on the category
of $SL_2(\k)-$modules (and hence on its full subcategory $\T_{<p^n-1}$) is isomorphic to
$\Hom (C,?)$ and we need to show that the maps 
\begin{equation} \label{Cmaps}
\Hom (\bold 1,T_i)\to \Hom (C,T_i)
\end{equation}
induced by a natural map $C\to \bold 1$ are isomorphisms
for $i\in [0, p^n-2]$. Moreover, we already know that this is the case for $i\in [p^{n-1}-1,p^n-2]$. 
In particular, the map $C\to \bold 1$ is non-zero, hence surjective; thus the maps \eqref{Cmaps}
are injective. To show surjectivity we will use the following

\begin{lemma} \label{lemma embedd}
(i) Assume that $i\in [p^{k-1}-1,p^k-2]$. Then the module $T_i$ embeds into some module
$T_j$ with $j\in [p^k-1,p^{k+1}-2]$.

(ii) Assume that the module $T_i$ embeds into $T_{2p^k-2}$. Then $i=2p^l-2$ with $l\in [0,k]$.
\end{lemma}

\begin{proof}
(i) We proceed by induction in $k$. In the case $k=1$ we have an embedding $\bold 1=T_0\hookrightarrow T_{2p-2}$
(see Lemma \ref{lem2}); for any $i\in [0,p-2]$ the module $T_{2p-2}$ is a direct summand of 
$T_i\ot T_{2p-2-i}$ so we have a non-zero map $\bold 1\to T_i\ot T_{2p-2-i}$ and, thus, a non-zero
map $T_i=T_i^*\to T_{2p-2-i}$. This map is injective by Proposition \ref{standfacts} (i) and the base
case of the induction is done. 

Now let $k>1$. Write $i=a+pb$ where $p-1\le a\le 2p-2$ and $b\in \BZ_{\ge 0}$. Then $b$ satisfies
$p^{k-2}-1\le b\le p^{k-1}-2$ and by the induction assumption we have an embedding
$T_b \subset T_c$ where $c\in [p^{k-1}-1, p^k-2]$. Using Proposition \ref{standfacts} (iii) we have an embedding
$$T_i=T_a\ot T_b^{(1)}\hookrightarrow T_a\ot T_c^{(1)}=T_j$$
where $j=a+pc\in [p^k-1,p^{k+1}-2]$. This justifies the induction step.

(ii) By Lemma \ref{lem2} the socle of $T_{2p^k-2}$ is isomorphic to $\bold 1$. Thus the same is true
for $T_i$ and the result follows from Lemma \ref{lem2}.
\end{proof} 

By iterating Lemma \ref{lemma embedd} (i) we see that each $T_i$ with $i\in [0,p^{n-1}-2]$ embeds into $T_j$
with $j=j(i)\in [p^{n-1}-1,p^n-2]$; moreover by Lemma \ref{lemma embedd} (ii) $j(i)=2p^{n-1}-2$ if and only if $i=2p^l-2$ with $l\in [0,n-2]$. Thus we have an embedding 
$\Hom(C,T_i)\hookrightarrow \Hom(C,T_j)=\Hom(\bold 1,T_j)$; so, using Lemma \ref{lem2}, we conclude that
$\dim \Hom(C,T_i)\le \dim \Hom(\bold 1, T_i)$. Thus the injective maps \eqref{Cmaps} are surjective
and proof is complete.

\vskip .05in

{\bf Remark} (i) Note that in all cases $A=T_{p^{n-1}-1}^{\otimes 2}$ and $B$ is obtained from
$T_{p^{n-1}-1}^{\otimes 4}$ by omitting all the summands $T_j$ with $j\ge p^n-1$. Thus
for $p\ge 5$ we have $B=T_{p^{n-1}-1}^{\otimes 4}$, hence $C=\bold 1$. Moreover, 
for $p\ge 3$ 
$$
\Hom_{\T_{<p^{n-1}-1}}(B,T_{p^{n-1}-1}^{\otimes 2})=
\Hom_{\T_{n,p}}(T_{p^{n-1}-1}^{\otimes 4},T_{p^{n-1}-1}^{\otimes 2})=
$$
$$
\Hom_{\T_{n,p}}(T_{p^{n-1}-1}^{\otimes 3},T_{p^{n-1}-1}^{\otimes 3})=
\Hom_{SL_2(\bold k)}(T_{p^{n-1}-1}^{\otimes 3},T_{p^{n-1}-1}^{\otimes 3})=
\Hom_{SL_2(\bold k)}(T_{p^{n-1}-1}^{\otimes 4},T_{p^{n-1}-1}^{\otimes 2}),
$$
which again implies that $C=\bold 1$, yielding the result. This is, essentially, the first proof of Theorem \ref{CartanFP0}(i) given in Subsection \ref{compthe}. 

On the other hand,
for $p=2$ and $n=2$ we have $A=T_1^{\otimes 2}=T_2$, $T_1^{\otimes 4}=T_2\oplus T_2\oplus T_4$
and $B=T_2\oplus T_2$. The Loewy series of $T_2$ is $[L_0,L_2,L_0]$, so $C=T_2/L_0\ne \bold 1$ in this case. 

Thus the argument involving Lemma \ref{lemma embedd} is needed for $p=2$, but not really for $p\ge 3$.  

(ii) The proof of Lemma \ref{lemma embedd} gives a more precise result: for
$i\in [ap^{k-1}-1, (a+1)p^{k-1}-2]$ with $a\in [1,p-1]$, we get that $T_i$ embeds into $T_j$
with $j=2p^{k-1}(p-a)+i$.

\subsection{Basic properties of $\Ver_{p^n}$} 

\begin{theorem}\label{cyclot} (i) If $n=1$ then $\Ver_{p^n}=\Ver_p$, $\Ver_{p^n}^+=\Ver_p^+$.

(ii) We have a fully faithful monoidal functor $F: \T_{n,p}\to \Ver_{p^n}$ 
which defines an equivalence between $\S_{n,p}$ and the category 
of projectives in $\Ver_{p^n}$. In particular, $\Ver_{p^n}$ has 
$p^{n-1}(p-1)$ simple objects whose projective covers are 
$F(T_i)$, $p^{n-1}-1\le i\le p^n-2$, and all simple objects in $\Ver_{p^n}$ are self-dual. 

(iii) We have $\Ver_{2^n}=\C_{2n-2}$ and if $n\ge 2$ then $\Ver_{2^n}^+=\C_{2n-3}$, where $\C_i$ are the categories defined in \cite{BE}. 

(iv) For $p^n>2$ let $\mO_{n,p}$ be the ring of integers in the field $\QQ(2\cos(2\pi/p^{n}))=\QQ(q^2+q^{-2})$, where $q=q_{n,p}:=\exp(\pi i/p^{n})$; that is, $\mO_{n,p}=\ZZ[2\cos(2\pi/p^{n})]$ (\cite{W}, p.16, Proposition 2.16). Then the Frobenius-Perron dimension $\FPdim$ defines an isomorphism of the Grothendieck ring 
${\rm Gr}(\Ver_{p^n}^+)$ onto $\mO_{n,p}$.   

(v) There are no tensor functors $\Ver_{p^n}^+\to {\rm Ver}_{p^m}^+$ for $m<n$. 
\end{theorem} 

\begin{proof} Parts (i), (ii) are immediate from the previous results. Part (iii) follows from (ii) and \cite{BE}, Theorem 4.14. 

Let us prove (iv). Note that the split Grothendieck ring of $\T_{n,p}$ is a $\Bbb Z_+$-ring of finite rank (\cite{EGNO}, Section 3), so for $T\in \T_{n,p}$ we can define its Frobenius-Perron dimension $\FPdim(T)$. Recall that 
$Q_{n,p}$ denotes the Chebyshev polynomial of the second kind \eqref{cheb}. 
 By \cite{BE}, Corollary 3.5, 
$$Q_{n,p}(\FPdim(T_1))=0,$$ which implies that $\FPdim(T_1)=\FPdim(F(T_1))=2\cos(\pi/p^{n})$. 
Thus we have $\FPdim(T_2)=2\cos(2\pi/p^n)+1$ for $p>2$ and $\FPdim(T_2)=2\cos(2\pi/p^n)+2$ for $p=2$. Hence the image of $\FPdim$ contains $\mO_{n,p}$. On the other hand, since the rank of $\Ver_{p^n}^+$ is $p^{n-1}(p-1)/2$, which equals the degree of $\QQ(2\cos(2\pi/p^{n}))$ over $\Bbb Q$, we see that $\FPdim$ lands in $\QQ(2\cos(2\pi/p^{n}))$, hence in $\mO_{n,p}$, since the Frobenius-Perron dimension of any object is an algebraic integer (\cite{EGNO}, Proposition 3.3.4). Thus $\FPdim: {\rm Gr}(\Ver_{p^n}^+)\to \mO_{n,p}$ 
is an isomorphism. 

(v) follows from (iv).
\end{proof} 

\begin{remark} For $p=2$ Theorem \ref{cyclot} is proved in \cite{BE}, Corollary 2.2 (the ring $\mathcal{O}_{n,2}$ is denoted there by $\mO_{n-2}$). 
\end{remark} 

Recall that $Q_{n,p}=Q_{n,p}^+-Q_{n,p}^-$ is the splitting of the Chebyshev polynomial of the second kind $Q_{n,p}$ (see \eqref{cheb})
into the parts with positive and negative coefficients. 
Also let us say that an invertible object $X$ is {\it odd} 
if the braiding acts on $X\otimes X$ by $-1$.  

\begin{corollary}\label{uni1} (The universal property of $\Ver_{p^n}$)
Let $\D$ be a symmetric tensor category over $\k$, and $p^n>2$. Then the category of symmetric tensor functors $E\colon \Ver_{p^n}\to \D$ is equivalent to the category of objects $X\in \D$ equipped with an isomorphism 
$\wedge^2X\cong \one$ such that 

(1) $\wedge^3X=0$ or $X$ is an odd invertible object (fermion) in characteristic $3$; 
and 

(2) $Q_{n,p}^+(X)\cong Q_{n,p}^-(X)$; and 

(3) $Q_{n-1,p}^+(X)\ncong Q_{n-1,p}^-(X)$. 

This equivalence is given by $E\mapsto E(V)$, where $V:=F(T_1)$. In particular, the group of tensor automorphisms of the identity functor of $\Ver_{p^n}$ is $\Bbb Z/2$ for $p>2$ and trivial for $p=2$. 
\end{corollary}

\begin{proof} This follows from Theorem \ref{CartanFP0} and Corollary \ref{corro}(ii). The last statement follows since the group of automorphisms of $V$ preserving the identification $\wedge^2V\cong \one$ is $\Bbb Z/2$ for $p>2$ and trivial for $p=2$.
\end{proof} 

\subsection{The lift of $\Ver_{p^n}$ to characteristic zero}\label{lifttochar0}

Let $W(\k)$ be the ring of Witt vectors of $\k$, $K_0$ its fraction field, and $\overline{K}_0$ the algebraic closure of $K_0$. 
Fix $n$ and let $\zeta=\zeta_{n,p}\in \overline{K}_0$ be a primitive $p^n$-th root of unity. Consider the discrete valuation ring $R_{n,p}:=W(\k)[\zeta]$ with uniformizing element $\zeta-1$. Let $K$ be the fraction field of $R_{n,p}$ (a ramified extension of $K_0$ of degree $p^{n-1}(p-1)$). 

\begin{proposition} The category $\Ver_{p^n}(\k)$ admits a flat deformation to a braided category over $R_{n,p}$, whose generic fiber is the semisimple Verlinde category $\Ver_{p^n}(K)$ associated to the quantum group $SL_2$ at a $p^n$-th root of unity. 
\end{proposition} 

\begin{proof} We will use the theory of splitting objects for categories defined over a complete DVR developed in Subsection \ref{DVR}. 

Let $TL(\zeta)$ be the lopsided version of the Temperley-Lieb category with parameter $\delta=\zeta^{1/2}+\zeta^{-1/2}$, see \cite{MS}, Subsection 2.5; it is defined over $\Bbb Z[\zeta]$, hence over $R_{n,p}$, and we'll use the $R_{n,p}$-linear version. Then $TL(\zeta)$ is a flat braided monoidal Karoubian category over $R_{n,p}$ with indecomposables $\widetilde{T}_i$, $i\ge 0$. Also, $TL(\zeta)_\k=\T_p$ is the symmetric category of tilting modules for $SL_2(\k)$, while $TL(\zeta)_K$ is (the lopsided version of) the braided category of tilting modules for the quantum group $SL_2^{\zeta^{1/2}}$ over $K$. Moreover, $TL(\zeta)$ has a tensor ideal $\mathcal{I}(\zeta)$ generated by the Steinberg module ${\bf St}=\widetilde{T}_{p^n-1}$ over the quantum group, which specializes to $\mathcal{I}_n$ over $\k$, and localizes to $\mathcal{I}(\zeta)_K:=K\otimes_{R_{n,p}} \mathcal{I}(\zeta)$ over $K$ (note that other tensor ideals $\mathcal{I}_k$, $k\ne n$, do not lift to tensor ideals in $TL(\zeta)$). Thus we can define the flat braided monoidal category $\T(\zeta):=TL(\zeta)/\mathcal{I}(\zeta)$ over $R_{n,p}$ with indecomposables $\widetilde{T}_i$, $0\le i\le p^n-2$, whose reduction $\T(\zeta)_\k$ is the symmetric category $\T_{n,p}$ with indecomposables $T_i$, $0\le i\le p^n-2$, while the localization $\T_K$ is the split semisimple braided Verlinde category ${\rm Ver}_{p^n}(K)$ (the semisimplification of $TL(\zeta)_K$, see \cite{BK}, Subsection 3.3 and references therein), with simple objects $T_i^\zeta=W_i^\zeta=L_i^\zeta$, $0\le i\le p^n-2$, the Weyl modules over the quantum group (which are also tilting and simple modules). Thus the objects $\widetilde{T}_i$, $p^{n-1}-1\le i\le p^n-2$ are splitting objects of $\T(\zeta)$.
\end{proof} 

\begin{remark}\label{notmod} 
Note that for $p>2$, since $\zeta$ is a root of unity of odd order, the category $TL(\zeta)_K$ is {\bf not} modular (unlike the situation of \cite{BK}, 
Subsection 3.3); rather, its symmetric (M\"uger) center is $\sVec_K$ generated by $T_0^\zeta=\bold 1$ and $T_{p^n-2}^\zeta=\psi$. In fact, we have 
$\Ver_{p^n}(K)=\Ver_{p^n}^+(K)\boxtimes \sVec_K$ as braided categories, 
where $\Ver_{p^n}^+(K)$ is the (modular) category generated by the 
objects of $\Ver_{p^n}(K)$ with even highest weight.     
\end{remark} 

\subsection{The Frobenius-Perron dimension and the Grothendieck ring of $\Ver_{p^n}$}
\begin{proposition} \label{CartanFP}
(i) We have 
$$
\FPdim(\Ver_{p^n})=\frac{p^n}{2\sin^2(\frac{\pi}{p^n})}.
$$

(ii) The Cartan matrix of $\Ver_{p^n}$ is given by $C=DD^T$, where $D$ is its decomposition matrix (with respect to its lift to characteristic zero). 
\end{proposition} 

\begin{proof} 
(i) This follows from Proposition \ref{cde}(i) since $${\rm FPdim}(\Ver_N(K))=\sum_{k=1}^{N-1}\frac{\sin^2(\frac{\pi k}{N})}{\sin^2(\frac{\pi}{N})}=\frac{N}{2\sin^2(\frac{\pi}{N})}.$$ 

(ii) This follows from Theorem \ref{CartanFP0}(i) and Corollary \ref{semisi}, as the category $\Ver_{p^n}(K)$ is semisimple. 
\end{proof} 

\begin{corollary} \label{Supervec}
If $p>2$ then we have $\Ver_{p^n}=\Ver_{p^n}^+\boxtimes \sVec$. 
Thus the Grothendieck ring of $\Ver_{p^n}$ is given by ${\rm Gr}(\Ver_{p^n})=\mO_{n,p}\Bbb Z/2=\mO_{n,p}[g]/(g^2-1)$. 
\end{corollary} 

\begin{proof} The first statement holds because it holds for the category $\Ver_{p^n}(K)$ (cf. Remark \ref{notmod}). Namely, the nontrivial invertible object of $\Ver_{p^n}(K)$ has to descend to an invertible object over $\k$ (as it has $\FPdim=1$). The second statement then follows from Theorem \ref{cyclot}(iv). 
\end{proof} 

\subsection{The Cartan matrix and the decomposition matrix of $\Ver_{p^n}$}\label{carmat} 

Now we can compute the decomposition matrix $D$ (and hence the Cartan matrix $C$) of $\Ver_{p^n}$. Recall that the indecomposable projectives of $\Ver_{p^n}$ are $\bold P_i=F(T_i)$ for $p^{n-1}-1\le i\le p^n-2$. So we may label the simple objects by the same integers and denote them by $\bold L_i$; thus $\bold P_i$ is the projective cover of $\bold L_i$. Also we have the {\it extended decomposition matrix} $\widetilde D$, a square matrix of size $p^n-1$ with entries $d_{ij}$, $0\le i,j\le p^n-2$, where $d_{ij}:=[(\widetilde{T}_i)_K: W_j^\zeta]$ is the multiplicity of $W_j^\zeta$ in the reduction of $\widetilde{T}_i$ to $K$. Then $D$ is the submatrix of $\widetilde{D}$ consisting of the rows labeled by $p^{n-1}-1\le i\le p^n-2$. 

Note that the characters of $W_j^\zeta$ are $[j+1]_x=\frac{x^{j+1}-x^{-j-1}}{x-x^{-1}}$. Thus if $\chi_i(x)$ is the character of $T_i$ 
as an $SL_2(\k)$-module then the entries $d_{ij}$ of $\widetilde{D}$ (hence $D$) are determined from the formula 
\begin{equation}\label{Donkin}
\chi_i(x)=\sum_j d_{ij}\frac{x^{j+1}-x^{-j-1}}{x-x^{-1}}.
\end{equation}

\begin{proposition}\label{Cartan nondeg} (i) We have 
$$
\sum_j d_{ij}x^{j+1}=(\chi_i(x)(x-x^{-1}))_+,
$$
where the subscript $+$ denotes the part with positive powers of $x$. 

(ii) The decomposition matrix $D$ has maximal rank 
$p^{n-1}(p-1)$; 

(iii) The entries of the Cartan matrix $C$ are $c_{ij}=(\chi_i,\chi_j)$ (the inner product of $SL_2$-characters). 

(iv) The Cartan matrix $C$ of $\Ver_{p^n}$ is positive definite, hence nondegenerate. 
\end{proposition} 

\begin{proof} (i) is immediate from formula \eqref{Donkin}.
(ii) follows since $d_{ij}=0$ if $j>i$. (iii) follows from (i) 
and Proposition \ref{CartanFP}(ii). Finally, (iv) follows from (ii), (iii). 
\end{proof} 

\begin{corollary} \label{FPcomp}
The Frobenius-Perron dimensions $d_i=\FPdim(\bold L_i)$ of simple objects of $\Ver_{p^n}$ are determined uniquely by the system of equations 
$$
C\bold d=\bold p,
$$
where $\bold d=(d_i)$, and $\bold p=(p_i)$ is the vector of Frobenius-Perron dimensions 
$p_i$ of $\bold P_i$, $p^{n-1}-1\le i\le p^n-2$. That is, 
$$
\bold d=C^{-1}\bold p.
$$ 
\end{corollary} 

\begin{proof} This follows from Proposition \ref{Cartan nondeg}(iv). 
\end{proof} 

The vector $\bold p$ can be computed from the formula $\bold p=D\bold w$, where 
$w_j=[j+1]_{q}=\frac{\sin(\pi (j+1)/p^n)}{\sin(\pi/p^n)}$ are the Frobenius-Perron dimensions of the Weyl modules $W_j^\zeta$ in $\Ver_{p^n}(K)$. 
Thus we get 
$$
\bold d=C^{-1}D\bold w=(DD^T)^{-1}D\bold w
$$
So we obtain

\begin{proposition} \label{FPcompu}
We have $\bold d=(DD^T)^{-1}D\bold w$, where $w_j=[j+1]_{q}$ 
for $0\le j\le p^n-2$ and $d_{ij}$ are the coefficients of $\chi_i(x)(x-x^{-1})$
for $p^{n-1}-1\le i\le p^n-2$. 
\end{proposition}  

This determines the Grothendieck ring of $\Ver_{p^n}^+$ with $\Bbb Z_+$-basis since the Frobenius-Perron dimension defines an isomorphism of this ring with $\mathcal{O}_{n,p}$.

\begin{remark} We indicate here an alternative approach to the results of Subsections \ref{lifttochar0}, \ref{carmat} which avoids using the lifting to characteristic zero. Namely one can compute the entries of the Cartan matrix directly as
$$
c_{ij}=\dim \Hom_{\Ver_{p^n}}(\bold P_i,\bold P_j)=\dim \Hom_{\T_p}(T_i,T_j)
$$ 
using Proposition \ref{fullfaith} and Theorem \ref{CartanFP0}. However, it is well known that $\dim \Hom_{\T_p}(T_i,T_j)=(\chi_i,\chi_j)$, see for example \cite[Section 3A]{AST}, so we get the formula $C=DD^T$ as in Proposition \ref{CartanFP} (ii) where $D$ is {\em defined} by equation \eqref{Donkin}. This implies that Propositions
\ref{Cartan nondeg} and \ref{FPcompu} hold and an easy computation gives the formula from Proposition \ref{CartanFP} (i).
\end{remark}

The matrix $\widetilde{D}$ (hence $D$) can now 
be easily computed via Donkin's explicit algorithm, see 
Proposition \ref{standfacts}(iii) and \cite[E.9]{Ja}. However, 
there is an even more explicit formula for $\widetilde{D}$, given in \cite{TW}. 
Namely, given a number $a\in [p^{n-1},p^n-1]$ written in base $p$ as 
$\overline{a_1...a_n}$ (so $a_1\ne 0$), we say that $b$ is a {\it descendant} of $a$ if $b=a_1p^{n-1}\pm a_2p^{n-2}\pm...\pm a_n$ for some choice of signs (clearly, in this case $b\in [1,p^n-1]$). Thus, every number $a$ has $2^\ell$ distinct descendants, where $\ell$ is the number of nonzero digits among $a_2,...,a_n$. 

\begin{proposition}\label{TWquiver}\cite{TW}
One has $d_{ij}=1$ if $j+1$ is a descendant of $i+1$, otherwise 
$d_{ij}=0$. 
\end{proposition} 

This yields a very explicit description of the Cartan matrix of $\Ver_{p^n}$. 

\begin{corollary}\label{carmat1} The entry $c_{ij}$ of the Cartan matrix 
is the number of common descendants of $i+1$ and $j+1$. 
\end{corollary} 

\begin{proof} This follows from Proposition \ref{CartanFP}(ii) and Proposition \ref{TWquiver}. 
\end{proof} 

\subsection{An explicit formula for the Cartan matrix of $\Ver_{p^n}$} 

Let us now derive an explicit formula for the Cartan matrix $C=C_{n,p}$ of the category $\Ver_{p^n}$. We first consider the case $p=2$. In this case let $C_{n,2}^+$ 
be the Cartan matrix of $\Ver_{2^n}^+$. Then it is shown in \cite{BE}, 
Proposition 4.2 that 
$$
C_{n+1,2}^+=\left(\begin{matrix} 2C_{n,2}^+ & C_{n,2}^+\\ C_{n,2}^+ & 2C_{n-1,2}\end{matrix}\right)
$$
and 
$$
C_{n,2}=C_{n,2}^+\oplus C_{n-1,2},
$$
for an appropriate labeling of rows and columns, where $C_{1,2}=1,C_{2,2}^+=2$. Hence
$$
\left(\begin{matrix} C_{n+1,2}^+\\ C_{n,2}\end{matrix}\right)=
\begin{pmatrix}
  \begin{matrix}
  2 & 1 \\
  1 & 0
  \end{matrix}
  & \rvline & \begin{matrix}
  0 & 0 \\
  0 & 2
  \end{matrix} \\
\hline
 \begin{matrix}
  1 & 0 \\
  0 & 0
  \end{matrix}  & \rvline &
  \begin{matrix}
  0 & 0 \\
  0 & 1
  \end{matrix}
\end{pmatrix}
\otimes \left(\begin{matrix} C_{n,2}^+\\ C_{n-1,2}\end{matrix}\right).
$$
Thus we get 

\begin{proposition}\label{chara2}
One has
$$
\left(\begin{matrix} C_{n+1,2}^+\\ C_{n,2}\end{matrix}\right)=
\begin{pmatrix}
  \begin{matrix}
  2 & 1 \\
  1 & 0
  \end{matrix}
  & \rvline & \begin{matrix}
  0 & 0 \\
  0 & 2
  \end{matrix} \\
\hline
 \begin{matrix}
  1 & 0 \\
  0 & 0
  \end{matrix}  & \rvline &
  \begin{matrix}
  0 & 0 \\
  0 & 1
  \end{matrix}
\end{pmatrix}^{\otimes n-1}
\otimes \left(\begin{matrix} 2\\ 1\end{matrix}\right).
$$
\end{proposition} 

This can also be derived from Corollary \ref{carmat1}. 

Now let us pass to the case $p>2$, and compute $C_{n,p}$ in a somewhat different way. For two numbers $a,b$ with exactly $m+1$ digits in the base $p$ expansion let $X_{ab}(m)$ be the number of common descendants of $a,b$. By Corollary \ref{carmat1}, $C_{n+1,p}=X(n)$ (where we label simple objects of $\Ver_{p^n}$ by highest weight of the projective cover). Let $Y_{ab}(m)$ be the number of descendants $\widetilde{a}$ of $a$ such that $\widetilde{b}:=\widetilde{a}-1$ is a descendant of $b$. We will often write 
$X_{ab}(m)$, $Y_{ab}(m)$ as $X_{ab}$, $Y_{ab}$ respectively, using that 
$m$ can be recovered as the number of digits of $a$ and $b$.
It is clear that for $m=0$ we have $X_{ab}=\delta_{ab}$ and $Y_{ab}=\delta_{a,b+1}$. 

Let $a=\overline{a_m...a_1a_0}$, $b=\overline{b_m...b_1b_0}$ where $a_m,b_m\ne 0$, and let $a'=\overline{a_m...a_1},b'=\overline{b_m...b_1}$.  

\begin{proposition}\label{Tro1} (i) If $a_0=b_0=0$ then $X_{ab}=X_{a'b'}$. 

(ii) If $a_0=b_0\ne 0$ then $X_{ab}=2X_{a'b'}$. 

(iii) If $a_0+b_0=p$ then $X_{ab}=Y_{a'b'}+Y_{b'a'}$.

(iv) Otherwise $X_{ab}=0$. 
\end{proposition} 

\begin{proof} (i), (ii), (iv) are easy, let us prove (iii). 
Having a common descendant means that $\widetilde{a}'p\pm a_0=\widetilde{b}'p\pm b_0$ 
for some choice of signs. The possible choices are $+,-$, which gives 
$\widetilde{a}'-\widetilde{b}'=-1$ and $-,+$, which gives $\widetilde{a}'-\widetilde{b}'=1$, as desired. 
\end{proof} 

\begin{proposition}\label{Tro2} (i) If $a_0-b_0=\pm 1$ then $Y_{ab}=X_{a'b'}$. 

(ii) If $a_0-b_0=p-1$ then $Y_{ab}=Y_{a'b'}$. 

(iii) If $a_0+b_0=p+1$ then $Y_{ab}=Y_{b'a'}$. 

(iv) Otherwise $Y_{ab}=0$. 
\end{proposition} 

\begin{proof} Similar to Proposition \ref{Tro1}.
\end{proof} 

Let $Z=Y+Y^T$. 

\begin{corollary}\label{Tro3} (i) If $a_0=b_0=0$ then $X_{ab}=X_{a'b'}$. 

(ii) If $a_0=b_0\ne 0$ then $X_{ab}=2X_{a'b'}$. 

(iii) If $a_0+b_0=p$ then $X_{ab}=Z_{a'b'}$.

(iv) Otherwise $X_{ab}=0$. 
\end{corollary} 

\begin{corollary} 
(i) If $a_0-b_0=\pm 1$ then $Z_{ab}=2X_{a'b'}$. 

(ii) If $a_0+b_0=p-1$ then $Z_{ab}=Z_{a'b'}$. 

(iii) If $a_0+b_0=p+1$ then $Z_{ab}=Z_{a'b'}$. 

(iv) Otherwise $Z_{ab}=0$. 
\end{corollary} 

Define the following $p\times p$ matrices $A,B,S,D$ with rows and columns labeled by $0,...,p-1$.
Let $A$ be defined by $A_{ij}=\delta_{i,j-1}+\delta_{i,j+1}$, 
$B$ be defined by $B_{ij}=\delta_{i+j,p-1}+\delta_{i+j,p+1}$, $S$ be the matrix with $S_{ij}=\delta_{i+j,p}$, and $D:={\rm diag}(1,2,2,...,2)$. 

\begin{corollary}\label{Tro5} We have $X(0)={\rm Id}_{p-1},Z(0)=\overline{A}$, 
where $\overline{A}$ is obtained from $A$ by deleting the 0-th row and column,
and for $m\ge 1$
$$
X(m)=D\otimes X(m-1)+S\otimes Z(m-1),\ Z(m)=2A\otimes X(m-1)+B\otimes Z(m-1).
$$
\end{corollary} 

Thus we get 
$$
\left(\begin{matrix}X(m)\\ Z(m)\end{matrix}\right)=\left(\begin{matrix}D & S\\ 2A & B\end{matrix}\right)\otimes\left(\begin{matrix}X(m-1)\\ Z(m-1)\end{matrix}\right).
$$
So we get 

\begin{proposition}\label{chaodd} One has 
$$
\left(\begin{matrix}C_{n,p}\\ Z(n-1)\end{matrix}\right)=\left(\begin{matrix}D & S\\ 2A & B\end{matrix}\right)^{\otimes n-1}\otimes \left(\begin{matrix}{\rm Id}_{p-1}\\ \overline{A}\end{matrix}\right)
$$
 \end{proposition} 
 
\begin{example} 
For $p=3$ we have 
$$
\left(\begin{matrix} C_{n,3}\\ Z(n-1)\end{matrix}\right)=
\begin{pmatrix}
  \begin{matrix}
  1 & 0 & 0 \\
  0 & 2 & 0 \\
  0 & 0 & 2
  \end{matrix}
  & \rvline & 
 \begin{matrix}
  0 & 0 & 0 \\
  0 & 0 & 1 \\
  0 & 1 & 0
  \end{matrix}   \\
\hline
  \begin{matrix}
  0 & 2 & 0 \\
  2 & 0 & 2 \\
  0 & 2 & 0
  \end{matrix}  & \rvline &
   \begin{matrix}
  0 & 0 & 1 \\
  0 & 1 & 0 \\
  1 & 0 & 1
  \end{matrix}
\end{pmatrix}^{\otimes n-1}
\otimes \left(\begin{matrix}  \begin{matrix}
  1 & 0 \\
  0 & 1
  \end{matrix}\\ \hline \begin{matrix}
  0 & 1 \\
  1 & 0
  \end{matrix}\end{matrix}\right).
$$
Thus, for $n=2$ we get 
$$
C_{2,3}=\begin{pmatrix}
 \begin{matrix}
  1 & 0 & 0 \\
  0 & 2 & 0 \\
  0 & 0 & 2
  \end{matrix}
  & \rvline &
 \begin{matrix}
  0 & 0 & 0 \\
  0 & 0 & 1 \\
  0 & 1 & 0
  \end{matrix}   \\
\hline
 \begin{matrix}
  0 & 0 & 0 \\
  0 & 0 & 1 \\
  0 & 1 & 0
  \end{matrix}
  & \rvline &
 \begin{matrix}
  1 & 0 & 0 \\
  0 & 2 & 0 \\
  0 & 0 & 2
  \end{matrix}
  \end{pmatrix}
 $$
 which after permuting rows and columns falls into a direct sum of 
 the Cartan matrices for the four blocks of $\Ver_{3^2}$, two copies of $1$ 
 and two copies of 
 $\left(\begin{matrix}
  2 & 1 \\
  1 & 2   
 \end{matrix}\right)$.
\end{example} 

\begin{corollary}\label{powersof2} All entries of the Cartan matrix $C_{n,p}$ are either $0$ 
or $2^m$ with $0\le m\le n-1$. 
\end{corollary} 

\begin{proof} It follows from Propositions \ref{chara2},\ \ref{chaodd} that 
while computing entries of $C_{n,p}$, we never have to add nonzero numbers and only have to multiply them (for example, for $p>2$ this is so because the matrices 
$D$ and $S$ don't have nonzero entries in the same position, and $A,B$ don't either). 
This implies the statement. 
\end{proof} 

\begin{example}\label{cart11} The Cartan matrix entry $c_{\one,\one}$ of $C=C_{n,p}$ equals $2^{n-1}$. 
\end{example} 

\begin{remark} 1. For $p=2$ Corollary \ref{powersof2} is proved in \cite{BE}, Proposition 4.2. 

2. Here is another proof of Corollary \ref{powersof2}.
Namely, by Proposition \ref{standfacts}(iii), it suffices to show that for $p-1\le a,c\le 2p-2$ the space
$$\Hom(T_a\ot T_b^{(1)},T_c\ot T_d^{(1)})=\Hom(T_a\ot
T_c, (T_b\ot T_d)^{(1)})
$$
has dimension $0$ or $2^m$, $0\le m\le n-1$. But among the direct summands of $T_a\ot T_c$ only $T_{2p-2}$ and $T_{3p-2}$ have 
quotients on which the Frobenius kernel acts trivially (note that $T_{3p-2}$ appears only if $a$ and $c$
have different parities and the corresponding quotient is $T_1^{(1)}$). Thus we have the following
recursive formulas:
\begin{equation}\label{hom rec}
\dim \Hom(T_a\ot T_b^{(1)},T_c\ot T_d^{(1)})=
\left\{ 
\begin{array}{ccc} 
2\dim \Hom(T_b,T_d)&\mbox{if}&a=c\ne p-1,\\
\dim \Hom(T_b,T_d)&\mbox{if}&a=c=p-1,\\
\dim \Hom(T_b, V\ot T_d)&\mbox{if}&a+c=3p-2.\\
0\text{  otherwise}.
\end{array}
\right.
\end{equation}
We can combine formula \eqref{hom rec} with \eqref{eq1} and \eqref{eq2}. One observes that in all cases the number 
$\dim \Hom(T_b, V\ot T_d)$ equals to $\dim \Hom(T_b, T_{d\pm 1})$ or $\dim \Hom(T_{b\pm 1}, T_d)$
as the summands in \eqref{eq1} are typically in different linkage classes. This implies that all
nonzero entries of the Cartan matrix are powers of 2, and the power is $\le n-1$.
\end{remark}

\subsection{The objects $\Bbb T_i=F(T_i)$}
Let us now discuss the properties of the objects $\Bbb T_i:=F(T_i)$. Recall that for $p^{n-1}-1\le i\le p^n-2$
the objects $\Bbb T_i=\bold P_i$ are precisely the indecomposable projectives of the category $\Ver_{p^n}$.

Next consider the case $i<p$. Note that for such $i$ we have 
$\Bbb T_i=S^i\Bbb T_1$. 

\begin{lemma}\label{l11} For $i<p$ the objects $\Bbb T_i$ are simple. The projective cover of 
$\Bbb T_i$ is $\Bbb T_{2p^{n-1}-2-i}$.
\end{lemma} 

\begin{proof} We claim that if, for $j\in [p^{n-1}-1,p^n-2]$, we have $\Hom(T_j,T_i)\ne 0$ then it is \linebreak 1-dimensional and $j=2p^{n-1}-2-i$. Indeed, 
$$
\Hom(T_j,T_i)=\Hom({\bf 1},T_i\ot T_j^*)=\Hom({\bf 1},T_i\ot T_j).
$$
Thus by Lemma \ref{lem2} $\Hom(T_j,T_i)\ne 0$ if and only if $T_i\ot T_j$ contains $T_{2p^{n-1}-2}$ as a
direct summand. Now the claim follows from Proposition \ref{Tm appears}. 

It follows that for any $\Bbb T_i$ there exists a unique indecomposable projective $P$ such that 
$\Hom(P, \Bbb T_i)\ne 0$ and this $\Hom$--space is one dimensional. This implies the statement. 
\end{proof} 

\begin{lemma}\label{l12} The object $\Bbb T_p$ has length $3$ and composition series 
$[\Bbb T_{p-2},U,\Bbb T_{p-2}]$ for some simple $U$. 
\end{lemma} 

\begin{proof} We have $\Bbb T_p=\Bbb T_1\otimes \Bbb T_{p-1}$. Thus 
$\FPdim(\Bbb T_p)=\frac{(q^2-q^{-2})(q^{p}-q^{-p})}{(q-q^{-1})^2}$. 
But we have nonzero morphisms $\Bbb T_p\to \Bbb T_{p-2}$ and $\Bbb T_{p-2}\to \Bbb T_p$, 
so by Lemma \ref{l11}, $\Bbb T_{p}$ has composition series as claimed with some object $U$. It remains to show that $U$ is simple. To this end, note that 
$$
\FPdim(U)=\FPdim(\Bbb T_p)-2\FPdim(\Bbb T_{p-2})=
$$
$$
=\frac{(q^2-q^{-2})(q^{p}-q^{-p})-2(q-q^{-1})(q^{p-1}-q^{-p+1})}{(q-q^{-1})^2}=q^p+q^{-p}.
$$
Thus $\FPdim(U)<2$, so $U$ is simple. 
\end{proof} 

\subsection{Some properties of the Frobenius functor}
Consider the Frobenius functor 
$$
\Bbb F: \Ver_{p^n}\to \Ver_{p^n}^{(1)}\boxtimes \Ver_p
$$ 
defined in \cite{EO} (note that $\Ver_{p^n}^{(1)}\cong \Ver_{p^n}$ since the categories $\T_{n,p}$ and $\Ver_{p^n}$ are defined over the prime field). 

\begin{proposition}\label{notfroex} For $n\ge 2$ the functor $\Bbb F$ is not exact, i.e., the category $\Ver_{p^n}$ is not Frobenius exact in the sense of \cite{EO}. Equivalently, 
$\Ver_{p^n}$ is not locally semisimple in the sense of \cite{C2} (see \cite{EO}, Remark 7.2). 
\end{proposition} 

\begin{proof} Since there are no tensor functors $\Ver_{p^n}\to \Ver_p$ (Theorem \ref{cyclot}(v)), by \cite[Section 8]{EO} the category $\Ver_{p^n}$ is not Frobenius exact. 
\end{proof} 

In this section we discuss some properties of $\Bbb F$ restricted to $F(\T_{n,p})\subset \Ver_{p^n}$. We assume that $n>1$; see \cite[Example 3.8]{O} for the case $n=1$.

\begin{corollary}\label{c11} We have $\Bbb F(\Bbb T_1)=U\boxtimes \one \in \Ver_{p^n}\boxtimes \Ver_p$, where $U$ is the object from Lemma \ref{l12}. 
\end{corollary} 

\begin{proof} We have a decomposition $\Bbb T_1^{\otimes p}=\coplus_{i<p-2}\Bbb T_i\otimes \pi_i\oplus M$, where $\pi_i$ are free $\Bbb Z/p$-modules and 
$$
M=(p-2)\Bbb T_{p-2}\oplus \Bbb T_p=
(p-2)\Bbb T_{p-2}\oplus [\Bbb T_{p-2},U,\Bbb T_{p-2}],
$$
where the last equality comes from Lemma \ref{l12}.
This implies the statement. 
\end{proof} 

Corollary \ref{c11} implies that for $n>1$ the functor $\Bbb F$ restricted to $F(\T_{n,p})$ lands in $\Ver_{p^n}\boxtimes \one
\subset \Ver_{p^n}\boxtimes \Ver_p$. In what follows we identify $\Ver_{p^n}\boxtimes \one$ with
$\Ver_{p^n}$, so $\Bbb F: F(\T_{n,p})\to \Ver_{p^n}$.

The following result should be compared with \eqref{Donkin alg}.
\begin{proposition} Let $n>1$. 
For $p-1\le a\le 2p-2$ and $b\in \BZ_{\ge 0}$ we have
\begin{equation} \label{Frobenii}
\Bbb T_a\ot \Bbb F (\Bbb T_b)=\Bbb T_{a+pb}.
\end{equation}
\end{proposition}

\begin{proof} The morphisms $\Bbb T_{p-2}\to \Bbb T_p$ and $\Bbb T_p\to \Bbb T_{p-2}$ from the
proof of Lemma \ref{l12} split after tensoring with $\St =\Bbb T_{p-1}$ by Proposition \ref{splitmor},
so we get $\St \ot U=\Bbb T_{2p-1}$ by an easy computation (where we denote $F(\St)$ just by $\St$ for brevity).

Let $\T^1_{n,p}=\mathcal{I}_1/\mathcal{I}_n\subset \T_{n,p}$ be the thick ideal generated by $\St$. 
We prove now that \linebreak $\St \ot \Bbb F (\Bbb T_i)\in F(\T^1_{n,p})$ by induction in $i$. The base case $i=0$ is clear
as $T_0=\one$. Assume that $\St \ot \Bbb F (\Bbb T_i)\in F(\T^1_{n,p})$ holds for some $i$. Then $\Bbb T_{i+1}$ is a direct summand of $\Bbb T_i \ot \Bbb T_1$, so $\St \ot \Bbb F (\Bbb T_{i+1})$ is a direct
summand of $\St \ot \Bbb F (\Bbb T_i)\ot \Bbb F (\Bbb T_1)=\St \ot U\ot \Bbb F (\Bbb T_i)=
\Bbb T_{2p-1}\ot \Bbb F (\Bbb T_i)$. Since $\Bbb T_{2p-1}$ is a direct summand of $\Bbb T_{p}\ot \St$, 
we see that $\St \ot \Bbb F (\Bbb T_{i+1})$ is a direct summand of $\Bbb T_{p}\ot \St \ot \Bbb F (\Bbb T_i)$ and thus $\St \ot \Bbb F (\Bbb T_{i+1})\in F(\T^1_{n,p})$ since the category $\T^1_{n,p}$ is closed under tensor products and direct summands.

For $a\ge p-1$ the object $\Bbb T_a$ is a direct summand of $\Bbb T_{a-(p-1)}\ot \St$, so 
$\Bbb T_a\ot \Bbb F (\Bbb T_b)$ is a direct summand of $\Bbb T_{a-(p-1)}\ot \St \ot \Bbb F (\Bbb T_b)$,
so $\Bbb T_a\ot \Bbb F (\Bbb T_b)\in F(\T^1_{n,p})$ by the preceding paragraph.

We now prove the proposition by induction in $b$ with the base case $b=0$ being clear.
Thus assume that \eqref{Frobenii} holds for all $b\le m$. Let $S=S(m)$ be the multi-subset of $\BZ_{\ge 0}$ such that 
$$T_1\ot T_m=T_{m+1}\oplus \coplus_{s\in S}T_s.$$
It is clear that $s<m$ for any $s\in S$, so \eqref{Frobenii} holds for $b=s\in S$. We have
$$\St \ot \Bbb T_a\ot \Bbb F (\Bbb T_{m+1}\oplus \coplus_{s\in S}\Bbb T_s)=\St \ot \Bbb T_a\ot \Bbb F (\Bbb T_1\ot \Bbb T_m)=\St \ot U\ot \Bbb T_a\ot \Bbb F (\Bbb T_m)=\Bbb T_{2p-1}\ot \Bbb T_{a+pm}=$$
$$F(T_{2p-1}\ot T_{a+pm})=F(T_{p-1}\ot T_1^{(1)}\ot T_a\ot T_m^{(1)})=\St \ot F(T_a\ot (T_1\ot T_m)^{(1)})=$$
$$\St \ot F\biggl(T_a\ot (T_{m+1}\oplus \coplus_{s\in S}T_s)^{(1)}\biggr)=\St \ot F(T_{a+p(m+1)}\oplus \coplus_{s\in S}T_{a+ps})=
$$
$$
=
\St \ot (\Bbb T_{a+p(m+1)}\oplus \coplus_{s\in S}\Bbb T_{a+ps}).$$
Thus using the induction assumption for $b=s\in S$ we get
$$\St \ot \Bbb T_a\ot \Bbb F (\Bbb T_{m+1})=\St \ot \Bbb T_{a+p(m+1)},$$
that is \eqref{Frobenii} holds for $b=m+1$ after tensoring with $\St$. 

Luckily we can cancel $\St$: in the split Grothendieck ring $K(\T_{n,p})$ we have $K(\T^1_{n,p})=[\St]K(\T_{n,p})$ and we get the
cancellation by applying the following result to the ring $K(\T_{n,p})$ and $a=[\St]$:
in a commutative semisimple ring the equality $a^2b_1=a^2b_2$
implies $ab_1=ab_2$.  
\end{proof}

\subsection{The inclusion $\Ver_{p^{n-1}}\hookrightarrow \Ver_{p^{n}}$}

In this section we use \eqref{Frobenii} to prove the following theorem. 

\begin{theorem} \label{inclusion}
For $n\ge 2$ we have a natural inclusion $H: \Ver_{p^{n-1}}\hookrightarrow \Ver_{p^{n}}$ 
as a tensor subcategory. 
\end{theorem} 

\begin{proof}
Recall that we have a tensor functor $\Bbb F: F(\T_{n,p})\to \Ver_{p^n}$.

\begin{corollary} $\Bbb F (F(\St_{n-1}))=0$ and $\Bbb F(F(\St_{n-2}))\ne 0$.
\end{corollary}

\begin{proof} By \eqref{Frobenii}
we have 
$$
\Bbb F (\Bbb T_{p^{n-1}-1})\ot \Bbb T_{p-1}=\Bbb T_{p-1+p(p^{n-1}-1)}=\St_n=0
$$
and
$$
\Bbb F (\Bbb T_{p^{n-2}-1})\ot \Bbb T_{p-1}=\Bbb T_{p-1+p(p^{n-2}-1)}=\St_{n-1}\ne 0. \qedhere
$$
\end{proof}

Thus the functor $\Bbb F\circ F$ induces a faithful monoidal functor $\widetilde H: \T_{n-1,p}\to \Ver_{p^n}$, see Proposition \ref{tensid} or \cite{BE}, Lemma 3.7. But  $\T_{n-1,p}$ has an abelian envelope $\Ver_{p^{n-1}}$ by Theorem
\ref{CartanFP0}. Thus the functor $\widetilde H$ extends to an (exact) tensor functor $H: \Ver_{p^{n-1}}\to \Ver_{p^n}$. This functor is automatically faithful, see \cite[Remark 4.3.10]{EGNO}. We claim
that it is also full, i.e., $H$ is an embedding, see \cite[6.3]{EGNO}.

\begin{corollary} The functor $\widetilde H$ is full. In particular the functor $H$
is full on projective objects of $\Ver_{p^{n-1}}$.
\end{corollary}

\begin{proof} It is sufficient to show that the injective map $\Hom(\one,\Bbb T_b)\to \Hom(\Bbb F (\one), \Bbb F (\Bbb T_b))$ is surjective for any $b$. Thus by Lemma \ref{lem2} we need to show that $\Hom(\one, \Bbb F (\Bbb T_b))$
is one dimensional for $b=2p^l-2, 0\le l\le n-1$ and zero otherwise. But since $\one$ is a quotient of
$\Bbb T_{2p-2}^*$, we have
$$\dim \Hom(\one, \Bbb F (\Bbb T_b))\le \dim \Hom(\Bbb T_{2p-2}^*, \Bbb F (\Bbb T_b))=
$$
$$
=\dim \Hom(\one, \Bbb T_{2p-2}\ot \Bbb F (\Bbb T_b))=\dim \Hom(\one, \Bbb T_{2p-2+pb}),$$
and the result follows from Lemma \ref{lem2}.
\end{proof}

It follows that the functor $H$ is full which proves the theorem.
\end{proof}

\begin{remark} Here is an alternative approach to the existence of the functor $H$. 
By Proposition \ref{notfroex} the category $\Ver_{p^n}$ for $n\ge 2$ is not Frobenius exact, so by \cite{EO}, Proposition 7.6, $\Bbb F(Q)=0$ for every projective $Q\in \Ver_{p^n}$. 
This means that the monoidal functor $\Bbb F\circ F: \T_{n,p}\to \Ver_{p^n}\boxtimes \Ver_p$ factors through $\T_{n-k,p}$ for some $k>0$. But 
$\T_{n-k,p}$ has an abelian envelope $\Ver_{p^{n-k}}$. Thus the functor $\Bbb F\circ F$ factors through a tensor functor $\widetilde H: \Ver_{p^{n-k}}\to \Ver_{p^n}\boxtimes \Ver_p$. By Corollary \ref{c11}, this functor lands in $\Ver_{p^n}$ (as $\Bbb T_1$ generates $\Ver_{p^{n-k}}$). Recall that the Frobenius-Perron
dimensions of the objects of $\Ver_{p^{n-k}}$ are contained in the ring $\mO_{n-k,p}$, see Theorem 
\ref{cyclot} (iv). However by Corollary \ref{c11} the image of the functor $\widetilde H$ contains an object $U\boxtimes \one$ such that its Frobenius-Perron dimension generates the ring $\mO_{n-1,p}$,
see the proof of Lemma \ref{l12}. We conclude that $k=1$.
\end{remark}

In what follows we consider $\Ver_{p^{n-1}}$ as a subcategory of $\Ver_{p^n}$ via the functor $H$.

\begin{corollary} \label{covers}
Assume $n\ge 2$.
Let $L$ be a simple object of $\Ver_{p^{n-1}}$ with projective cover $\Bbb T_r$ where
$r\in [p^{n-2}-1, p^{n-1}-2]$. Then projective cover of $L$ considered as an object of $\Ver_{p^n}$
is $\Bbb T_{2p-2+pr}$.
\end{corollary}

\begin{proof} By \eqref{Frobenii}
we have $\Bbb T_{2p-2+pr}=\Bbb F (\Bbb T_{r})\ot \Bbb T_{2p-2}$, so $H(\Bbb T_r)=\Bbb F (\Bbb T_{r})$ is a quotient of $\Bbb T_{2p-2+pr}$ as
$\Bbb T_{2p-2}$ surjects onto $\one$. 
\end{proof}

\begin{proposition}\label{not maximal}
The functor $X\mapsto H(X)\ot {\Bbb T}_{p-1}$ is an (additive) equivalence of $\Ver_{p^{n-1}}$ with a direct summand of the category $\Ver_{p^n}$.
\end{proposition}

\begin{proof} It is clear that this functor is exact and faithful. We claim that it sends projective objects to projective objects and that it is full on projective objects. The first assertion is immediate
from \eqref{Frobenii} where we substitute $a=p-1$. The second assertion is an immediate consequence of \eqref{hom rec}. The proposition is proved.
\end{proof} 

\subsection{The tensor product theorem for $\Ver_{p^n}$}\label{tenprot}
Here is a generalization of Lemma \ref{l11}.

\begin{proposition} \label{simples}
Assume $n\ge 2$.
Let $L$ be a simple object of $\Ver_{p^{n-1}}$ with projective
cover $\Bbb T_{r}$, $r\in [p^{n-2}-1, p^{n-1}-2]$. Then for any $i\in [0,p-1]$ the object $\Bbb T_i\ot L$ 
of $\Ver_{p^n}$ is simple with projective cover $\Bbb T_{2p-2-i+pr}$. Any simple object of $\Ver_{p^n}$ is of this form.
\end{proposition}

\begin{proof} We claim that there is a unique $j\in [p^{n-1}-1,p^{n}-2]$ such that
$\Hom(\Bbb T_j, \Bbb T_i\ot L)\ne 0$ and this $\Hom$--space is one-dimensional.
But 
$$
\Hom(\Bbb T_j, \Bbb T_i\ot L)=\Hom(\Bbb T_j\ot \Bbb T_i^*, L)=\Hom(\Bbb T_j\ot \Bbb T_i, L)
$$
and by Corollary \ref{covers} this space is nonzero if and only if $\Bbb T_j\ot \Bbb T_i$ contains 
$\Bbb T_{2p-2+pr}$ as a direct summand. The first statement now follows from Proposition \ref{Tm appears}. The second statement is a consequence of the count of the simple modules. 
\end{proof}

By iterating Theorem \ref{inclusion} we get a sequence of embeddings $\Ver_p\subset \Ver_{p^2}\subset \ldots \subset \Ver_{p^n}$.
For any $r\ge 2$ and $i\in [0,p-1]$ let us denote by $\Bbb T_i^{[r]}$ the simple object of the category
$\Ver_{p^r}$ as in Lemma \ref{l11} possibly considered as an object of $\Ver_{p^s}$ with $s\ge r$. Similarly, for $i\in [0,p-2]$ let $\Bbb T_i^{[1]}$ denote the simple object of $\Ver_p$
possibly considered as an object of $\Ver_{p^s}$ with $s\ge 1$. Then Theorem \ref{inclusion} 
implies 
\begin{proposition}\label{frobsim}
For $0\le i\le p-1$ and $m\ge 2$, we have $\Bbb F(\Bbb T_i^{[m]})=\Bbb T_i^{[m-1]}$ 
except if $m=2$ and $i=p-1$, in which case $\Bbb F(\Bbb T_i^{[m]})=0$. 
\end{proposition} 

For a sequence of base $p$ digits
$a_1,a_2,\ldots, a_n \in [0,p-1]$ let 
$$
\overline{a_1a_2\ldots a_n}=a_1p^{n-1}+\ldots +a_{n-1}p+a_n
$$ 
be the
corresponding integer in base $p$. For a base $p$ digit $a$ let $a^*:=p-1-a$. By iterating Proposition \ref{simples} we get the following result (cf. Steinberg's tensor product theorem \cite[3.17]{Ja} and \cite[Theorem 2.1 (ix)]{BE}):

\begin{theorem}\label{tpt}
Let $i_1\in [0,p-2]$, $i_2,\ldots,i_n\in [0,p-1]$, and $i=\overline{i_1...i_n}\in [0,p^{n-1}(p-1)-1]$ the corresponding integer written in base $p$. Then the object
$L_i:=\Bbb T^{[1]}_{i_1}\ot \Bbb T^{[2]}_{i_2}\ot \ldots \ot \Bbb T^{[n]}_{i_n}\in \Ver_{p^n}$ is simple, and any simple object of $\Ver_{p^n}$ can be uniquely written in this form.\footnote{The object $L_i$ should not be mixed up with the simple module over $SL_2(\k)$ with highest weight $i$ which we also denoted by $L_i$. We hope this will not lead to confusion since the simple $SL_2(\k)$-modules will not be used from this point on.}   The projective cover of $L_i$ is $P_i:=\bold P_s=\Bbb T_s$ with 
$$
s=s(i):=p^{n-1}-1+\overline{i_1i_2^*\ldots i_n^*};
$$
 in other words, $L_i=\bold L_s$. Moreover, the simple objects of $\Ver_{p^n}^+$ are the
$L_i$ for even $i$. 
\end{theorem}

Thus, Theorem \ref{tpt} gives a new way of labeling the simple objects 
of $\Ver_{p^n}$, and the transition between the two labelings is given by the formula
$L_i=\bold L_{s(i)}$. For example, for $p=2$ we have $i_1=0$, so $s(i)=2^n-2-i$, so $L_i=\bold L_{2^n-2-i}$ for $0\le i\le 2^{n-1}-1$. Also, $L_{p^{n-1}(p-2)}=\bold L_{p^n-2}$ is 
the generator of $\sVec\subset \Ver_{p^n}$. 

\begin{remark} For $p=2$, the simple objects $L_1^{[m]}$ were denoted in \cite{BE} by $X_{m-1}$, $2\le m\le n$. Thus for a subset $S\subset [1,n-1]$, the object $X_S\in \Ver_{2^n}=\C_{2n-2}$ from \cite{BE} is $L_i$ where $i=\sum_{k\in S}2^{n-1-k}$.  
\end{remark} 

\begin{corollary}\label{dims} 
(i) The categorical dimensions of simple objects of $\Ver_{p^n}$ are given by the formula
 $$
 \dim(L_i)=\prod_{k=1}^n (i_k+1)\in \k.
 $$
 
 (ii) The Frobenius-Perron dimensions of simple objects of $\Ver_{p^n}$ are given by the formula
 $$
 \FPdim(L_i)=\prod_{k=1}^n [i_k+1]_{q^{p^{n-k}}}.
$$
Thus $\Ver_{p^n}^+$ is a categorification of the ring $\mathcal{O}_{n,p}$ with basis 
$b_i=\prod_{k=1}^n [i_k+1]_{q^{p^{n-k}}}$ for even $i\in [0,p^{n-1}(p-1)-1]$. 
\end{corollary} 

\begin{proof} 
This follows since $\dim \Bbb(T_j)=j+1$ and ${\rm FPdim}(\Bbb T_j)=[j+1]_q$ in $\Ver_{p^n}$. 
\end{proof} 

\begin{corollary}\label{frobact} The Frobenius functor acts on the simple objects of $\Ver_{p^n}$ as follows. Let $0\le r\le p-2$ and $0\le b\le p^{n-1}-1$. Then if $b\ge p^{n-1}-p^{n-2}$, we have 
$\Bbb F(L_{rp^{n-1}+b})=0$. Otherwise $\Bbb F(L_{rp^{n-1}+b})=L_b^{[n-1]}\boxtimes L_r$ if $r$ is even and $\Bbb F(L_{rp^{n-1}+b})=(L_{p-2}^{[1]}\otimes L_b^{[n-1]})\boxtimes L_{p-2-r}$ if $r$ is odd. 
\end{corollary} 

\begin{proof} This follows from Proposition \ref{frobsim}, Theorem \ref{tpt}, \cite[Example 3.8]{O} and the fact that $\Bbb F$ is a monoidal functor, see \cite{EO}. 
\end{proof} 

\begin{remark} The following remark is a generalization to odd characteristic of 
\cite{BE}, Remark 3.14. The category $\Ver_{p^n}$ is filtered by Serre subcategories $\Ver_{p^n}^r$, $r\ge 0$, whose simple objects are those occurring in $V^{\otimes j}$, $j\le r$ (where we recall that $V=\Bbb T_1$). These categories are not closed under tensor product, but we have partial tensor products $\Ver_{p^n}^j\times \Ver_{p^n}^{r-j}\to \Ver_{p^n}^r$ with associativity isomorphisms satisfying the pentagon relation. Moreover, we claim that the Frobenius functor $\Bbb F$  maps $\Ver_{p^n}^r$ to $\Ver_{p^{n-1}}^r$ when $n$ is large enough compared to $r$, preserving this partial tensor product. Indeed, for simple objects this follows from Corollary \ref{frobact}, and for arbitrary objects it follows since for a filtered object $Y$ any composition factor of $\Bbb F(Y)$ is also one of $\Bbb F({\rm gr}(Y))$ (see \cite{EO}, Proposition 3.6). 
\end{remark}

\begin{example} \label{tptex}
Let $p\le a\le 2p-2$. Then by Theorem \ref{tpt} the projective cover of $\Bbb T^{[n-1]}_1\ot \Bbb T^{[n]}_{a-p}$ is $\Bbb T_{2p^{n-1}-2-a}$. We have a nonzero morphism $\Bbb T_{2p^{n-1}-2-a}\ot \Bbb T_a\to
\Bbb T_{2p^{n-1}-2}\to \one$, hence a nonzero morphism $\Bbb T_{2p^{n-1}-2-a}\to \Bbb T_a$.
Hence $\Bbb T^{[n-1]}_1\ot \Bbb T^{[n]}_{a-p}$ is a simple constituent of $\Bbb T_a$. Similarly,
using the nonzero morphism $\Bbb T_a\ot \Bbb T_{2p-2-a}\to \Bbb T_{2p-2}\to \one$, we see that the
simple object $\Bbb T_{2p-2-a}$ (see Lemma \ref{l11}) is both a subobject and a quotient object of $\Bbb T_a$.
By comparing the Frobenius-Perron dimensions we deduce that the object $\Bbb T_a$ has 
composition series $$\Bbb T_a=[\Bbb T_{2p-2-a}, \Bbb T^{[n-1]}_1\ot \Bbb T_{a-p}, \Bbb T_{2p-2-a}],$$ 
generalizing
Lemma \ref{l12}.
\end{example} 

\begin{corollary}\label{autoeq}
The only braided tensor functor $E: \Ver_{p^n}\to \Ver_{p^n}$ is isomorphic to the identity. In particular the group
of braided autoequivalences of the category $\Ver_{p^n}$ is trivial.
\end{corollary} 

\begin{proof} We may assume that $p^n>2$. By Corollary \ref{uni1}, braided (= symmetric) tensor functors $E: \Ver_{p^n}\to \Ver_{p^n}$  correspond to objects $X\in \Ver_{p^n}$ with $\wedge^2X\cong \one$, $\wedge^3X=0$ or $X$ odd invertible in characteristic $3$, and $Q_{n,p}^+(X)\cong Q_{n,p}^-(X)$ but 
$Q_{n-1,p}^+(X)\ne Q_{n-1,p}^-(X)$. Such an object $X$ must have Frobenius-Perron dimension $q+q^{-1}$, in particular, it is simple, since this number is less than $2$. Thus it follows from Corollary \ref{dims}(ii) that $X\cong L_1$. Thus there is a unique braided tensor functor $E$ up to an isomorphism, which must therefore be isomorphic to the identity. 
\end{proof} 

\subsection{Tensor product of simple objects} Let us now compute the tensor product of simple objects of $\Ver_{p^n}$ (in its Grothendieck ring). We start with the following result about the tensor product $\Bbb T_m\otimes \Bbb T_r$
for $0\le m,r\le p-1$. 

\begin{proposition}\label{tenprotilt1}  
The tensor product $\Bbb T_m\otimes \Bbb T_r$ for $1\le m,r\le p-1$ is given by the formula 
$$
\Bbb T_m\otimes \Bbb T_r=\coplus_{|m-r|\le k\le 2(p-2)-m-r}\Bbb T_k\oplus \coplus_{p-1\le k\le m+r}\Bbb T_k,
$$
if $m+r\ge p$, and 
$$
\Bbb T_m\otimes \Bbb T_r=\coplus_{|m-r|\le k\le m+r}\Bbb T_k
$$
if $m+r<p$, where the summation is over integers $k$ such that $m+r-k$ is even. 
\end{proposition} 

\begin{proof} This follows immediately from Lemma \ref{smalltilt}.  
\end{proof} 

\begin{corollary}\label{multrule} The multiplication of simple objects in the Grothendieck ring ${\rm Gr}(\Ver_{p^n})$ 
is given by the following recursive rule. Let $n\ge 2$, $a=a'p+m$, $b=b'p+r$, where $0\le m,r\le p-1$ and $0\le a',b'\le p^{n-2}(p-1)-1$. 
Then for $m+r\ge p$ we have 
$$
L_aL_b=
$$
$$
\left(\sum_{|m-r|\le k\le 2(p-2)-m-r}L_k+\sum_{2(p-1)-m-r\le k\le p-1}(2-\delta_{k,p-1})L_k\right)(L_{a'}L_{b'})^{[n-1]}+
$$
$$
+\sum_{p\le k\le m+r}L_{k-p}(VL_{a'}L_{b'})^{[n-1]}. 
$$
 and for $m+r<p$ we have
$$
L_aL_b=
\left(\sum_{|m-r|\le k\le m+r}L_k\right) (L_{a'}L_{b'})^{[n-1]},
$$
where $V:=\Bbb T_1$ and the summation is over integers $k$ such that $m+r-k$ is even. 
\end{corollary} 

\begin{proof} This follows from Proposition \ref{tenprotilt1} and Example \ref{tptex}.
\end{proof} 

This allows one to quickly compute products in ${\rm Gr}(\Ver_{p^n})$, since 
the products $L_{a'}L_{b'}$ and $VL_{a'}L_{b'}$ are computed in the
smaller category $\Ver_{p^{n-1}}$. 

Here is a somewhat more elegant way to reformulate Corollary \ref{multrule}.  
Let $A$ be a unital commutative $\Bbb Z_+$-ring (\cite{EGNO}, Subsection 3.1).
By a {\it unital $\Bbb Z_+$-ring over $A$} we will mean a ring $B\supset A$ 
with a free $A$-basis $b_i,i\in I$ (where $b_0=1$ for some element $0\in I$), 
such that $A$ is central in $B$ and $b_ib_j=\sum_k N_{ij}^k b_k$, where 
$N_{ij}^k$ are {\it nonnegative} elements of $A$. In this case $B$ is a unital $\Bbb Z_+$-ring 
in the sense of \cite{EGNO}, Subsection 3.1, with basis $a_ib_j$, where $a_i$ is the basis of $A$. 

\begin{corollary}\label{multrule1} For $n>1$ the ring ${\rm Gr}(\Ver_{p^n})$ 
is a unital $\Bbb Z_+$-ring over ${\rm Gr}(\Ver_{p^{n-1}})$ with basis 
$L_m$, $0\le m\le p-1$, and multiplication given by 
$$
L_mL_r=
\sum_{|m-r|\le k\le 2(p-2)-m-r}L_k+\sum_{2(p-1)-m-r\le k\le p-1}(2-\delta_{k,p-1})L_k+\sum_{p\le k\le m+r}VL_{k-p}
$$
for $m+r\ge p$ and
$$
L_mL_r=
\sum_{|m-r|\le k\le m+r}L_k,
$$
for $m+r<p$, where $V=\Bbb T_1^{[n-1]}\in {\rm Gr}(\Ver_{p^{n-1}})$ and the summation is over integers $k$ such that $m+r-k$ is even. 
\end{corollary} 

\begin{example} Since $\Ver_3=\sVec$, the ring ${\rm Gr}(\Ver_3)$ is $\Bbb Z[V]/(V^2-1)$ with basis $1,V$. Thus by Corollary \ref{multrule1} the ring ${\rm Gr}(\Ver_9)$ has basis $1,L_1,L_2$ over ${\rm Gr}(\Ver_3)$ with multiplication given by 
$$
L_1^2=1+L_2,\ L_1L_2=2L_1+V,\ L_2^2=2+L_2+VL_1. 
$$
So the basis of ${\rm Gr}(\Ver_9^+)$ is $1,M_1=VL_1,M_2=L_2$ with multiplication given by 
$$
M_1^2=1+M_2,\ M_1M_2=1+2M_1,\ M_2^2=1+M_1+M_2. 
$$
This agrees with the data in Subsection \ref{psquared}. 
\end{example} 

\subsection{The blocks of $\Ver_{p^n}$}  

\begin{proposition} \label{blocks}
The objects $\Bbb T_i$ and $\Bbb T_j$ are in the same block if and only if
the following holds:

(a) $i$ and $j$ have the same parity;

(b) base $p$ expansions of $i+1$ and $j+1$ have the same number of zeros at the end;

(c) the last nonzero base $p$ digits of $i+1$ and $j+1$ are either equal or sum up to $p$.
\end{proposition}

\begin{proof} This follows from Proposition \ref{TWquiver}.
\end{proof}

It now follows from Proposition \ref{blocks} that  $\Ver_{p^n}$ has the following block structure. 

\begin{proposition}\label{blockstr} The category $\Ver_{p^n}$ has 
$n(p-1)$ blocks. More specifically, it has $p-1$ blocks of size $1$ formed by projective (simple) objects $\Bbb T_i$ with $i+1$ divisible by $p^{n-1}$, $p-1$ blocks
of size $p-1$ formed by projective objects $\Bbb T_i$ with $i+1$ divisible by $p^{n-2}$ but not $p^{n-1}$, $p-1$ blocks
of size $p^2-p$ formed by projective objects $\Bbb T_i$ with $i+1$ divisible by $p^{n-3}$ but not $p^{n-2}$, etc. In particular it has $p-1$ blocks of maximal size $p^{n-1}-p^{n-2}$ formed by projective objects $\Bbb T_i$ with $i+1$ not divisible by $p$; all these blocks are equivalent to each other by the translation principle, see \cite{Ja}. 
\end{proposition} 

Observe that by Proposition \ref{not maximal} the functor $T\mapsto \St \ot T^{(1)}$ from
$\T_{\le p^{n-1}-1}$ to $\T_{\le p^n-1}$ is fully faithful; its image is generated by the objects $T_i$ such
that $i+1$ is divisible by $p$, i.e., by the objects lying in the blocks not of maximal size. Thus we get
the following statement:

\begin{proposition} \label{samecar}
The blocks of the same size in the categories $\Ver_{p^n}$ (where $n$ is allowed to vary) are all equivalent to each other; in particular they have the same Cartan matrices.
\end{proposition}

\begin{remark} Let $p=2$ and let $\Ver^-_{2^n}\subset \Ver_{2^n}$ be the additive subcategory 
generated by $\Bbb T_i$ with odd $i$. Then  $\Ver^+_{2^n}$ acts on $\Ver^-_{2^n}$ by left multiplication. At the same time the category $\Ver_{2^{n-1}}$ acts on $\Ver^-_{2^n}$ by right
multiplication: $X,Y\mapsto X\ot H(Y)$; clearly these two actions commute. It follows from
the discussion above that $\Ver^-_{2^n}$ considered as a module category over $\Ver_{2^{n-1}}$ is 
equivalent to the regular module (i.e., to $\Ver_{2^{n-1}}$ considered as a module over itself). This
implies that there exists a tensor (but not braided) functor $\Ver^+_{2^n}\to \Ver_{2^{n-1}}$; thus
the category $\Ver^+_{2^n}$ can be described as the representation category of some triangular Hopf algebra in $\Ver_{2^{n-1}}$. This can be considered as an a posteriori explanation of the construction in \cite{BE}.
\end{remark}

\subsection{The principal block} 
One of the blocks of the category $\Ver_{p^n}$ of maximal size is the {\em principal block
} containing the unit object. For $p>2$ and $n>1$ the projective objects in this block are $\Bbb T_i$ with even $i$ such that $p^n-1>i\ge p^{n-1}-1$ and 
the last base $p$ digit of $i$ is $0$ or $p-2$. That is, $i\in R_n$, where
\begin{equation}\label{principal block}
R_n=\{ p^{n-1}+p-2, p^{n-1}+p, p^{n-1}+3p-2, p^{n-1}+3p,
\ldots , p^n-p-2, p^n-p\} .
\end{equation}

We are going to describe this block in terms of the category $\Ver_{p^{n-1}}$. Let $\Lambda$ be the
image of the exterior algebra $\wedge^\bullet (T_1)\in \T_p$ in $\T_{p,n-1}\subset \Ver_{p^{n-1}}$
(we have $\Lambda =\wedge^\bullet ({\Bbb T}_1)$ except when $p=3$ and $n=2$). 

\begin{proposition}\label{prinbl} The principal block of the category $\Ver_{p^n}$ is equivalent to the category
$(\Ver_{p^{n-1}})_{\Lambda}$ of right $\Lambda-$modules in $\Ver_{p^{n-1}}$ as an abelian category.
\end{proposition}

\begin{proof} The case $p=2$ is known from \cite{BE}, so we will assume $p>2$.
We are going to show that the endomorphism algebras of the projective generators (with every indecomposable projective appearing with multiplicity 1) of both categories are isomorphic.

Let us start with the category $(\Ver_{p^{n-1}})_{\Lambda}$. Its projective objects are free modules
$P\ot \Lambda$ where $P$ is a projective object of $\Ver_{p^{n-1}}$. We have
$$
\Hom_{\Lambda}(P_1\ot \Lambda ,P_2\ot \Lambda)=\Hom(P_1, P_2\ot \Lambda)
$$
where the composition is induced by the multiplication in the algebra $\Lambda$. Let
$\Lambda_0\subset \Lambda$ be the isotypic component corresponding to $\one$; it is clear
that $\Lambda_0=\k[x]/(x^2)$ and $\Lambda =\Lambda_0\oplus \Bbb T_1$. The multiplication of $x$
and $\Bbb T_1$ is zero and the multiplication $\Bbb T_1\ot \Bbb T_1\to \k x\subset \Lambda_0$ is the unique up to scaling nonzero map $\Bbb T_1\ot \Bbb T_1\to \one$. Thus
$$
\Hom(\Bbb T_i, \Bbb T_j\ot \Lambda)=\left\{ \begin{array}{cc} \Hom(\Bbb T_i, \Bbb T_j)\ot \k[x]/(x^2)&\mbox{if}\; \; i\equiv j\pmod{2}\\ \Hom(\Bbb T_i, \Bbb T_j\ot \Bbb T_1)&\mbox{if}\; \; i\not \equiv j\pmod{2}
\end{array}\right.
$$
and the compositions are described by the multiplication rules above.

We now consider the principal block of the category $\Ver_{p^n}$. By Proposition \ref{fullfaith} we need 
to compute $\Hom(T_r,T_s)$ with $r,s\in R_n$, see \eqref{principal block}. Note that each $r\in R_n$
can be uniquely written as $r=a+pi$ where $a=p$ or $a=2p-2$ and $p^{n-2}-1\le i\le p^{n-1}-2$. 
Thus by \eqref{Donkin alg} we have $T_r=T_a\ot T_i^{(1)}$ and $T_s=T_{a'}\ot T_j^{(1)}$ where
$a$ and $a'$ are $p$ or $2p-2$. It is clear that $a=a'$ if and only if $i\equiv j\pmod{2}$. We have
$$
\Hom(T_r,T_s)=\Hom(T_a\ot T_i^{(1)},T_{a'}\ot T_j^{(1)})=\Hom(T_i^{(1)},T_j^{(1)}\ot T_a^*\ot T_{a'}).
$$
We can compute the last $\Hom$ is stages: first over the Frobenius kernel of $SL_2(\k)_1$ of $SL_2(\k)$ and then
over the quotient $SL_2(\k)^{(1)}$ of  $SL_2(\k)$ by the Frobenius kernel. Thus we can replace the tensor factor $T_a^*\ot T_{a'}$
by the invariants under the Frobenius kernel. Using the Jordan-H\"older series of $T_p$ and
$T_{2p-2}$ (see e.g. \cite[E.1]{Ja}) it is easy to compute that these invariants are
$\k[x]/(x^2)$ if $a=a'$ and $T_1^{(1)}$ if $a\ne a'$. Thus
$$
\Hom(T_r, T_s)=\left\{ \begin{array}{cc} \Hom(T_i, T_j)\ot \k[x]/(x^2)&\mbox{if}\; \; i\equiv j\pmod{2}\\ \Hom(T_i, T_j\ot T_1)&\mbox{if}\; \; i\not \equiv j\pmod{2}
\end{array}\right.
$$
In view of Proposition \ref{fullfaith} this matches perfectly with the above description for 
$(\Ver_{p^{n-1}})_{\Lambda}$. Moreover, it is easy to check that the compositions are described
in the same way. The result follows.
\end{proof}

\begin{remark} 1. Note that the proof of Proposition \ref{prinbl} actually constructs a specific equivalence of categories, uniquely defined up to an isomorphism of tensor functors. This equivalence
is determined by a choice of a generic copy of the algebra $\Lambda$ inside 
${\rm End}(T_p\oplus T_{2p-2})$ where ${\rm End}$ denotes the endomorphism algebra over the Frobenius kernel. Here ``generic'' means that $T_1\subset \Lambda$ does not act by zero on either
direct summand (thus there are exactly $2$ non-generic choices). Also all generic
choices are related by the action of ${\rm Aut}(T_p\oplus T_{2p-2})$, which shows that the constructed equivalence is uniquely defined up to an isomorphism. 

2. Observe that the algebra $\Lambda$ has a natural structure of a supercocommutative Hopf superalgebra. Thus the equivalent categories $(\Ver_{p^{n-1}})_{\Lambda}$ and the principal block of $\Ver_{p^n}$ are finite symmetric tensor categories in a natural way, equipped with a symmetric tensor functor $(\Ver_{p^n})_\Lambda\to \Ver_{p^{n}}$. For $p>2$ this tensor product on the principal block is different from the tensor product in $\Ver_{p^n}$, however (and the principal block is not closed under the latter); on the other hand, for $p=2$ the principal block coincides with $\Ver_{2^n}^+$ and the tensor products are the same, but the braiding is different, see \cite{BE}.
\end{remark}

\subsection{$\Ver_{p^{n}}$ is a Serre subcategory in $\Ver_{p^{n+k}}$}

\begin{proposition}\label{serre} Let $p>2$. Then 

(i) Any tensor subcategory $\C\subset \Ver_{p^{n+1}}$ containing $\Ver_{p^n}$ as a full subcategory coincides with $\Ver_{p^n}$ or with $\Ver_{p^{n+1}}$. 

(ii) The category $\Ver_{p^{n}}$ is a Serre subcategory in $\Ver_{p^{n+k}}$. 
\end{proposition}  

\begin{proof} 
(i) Let $X=\Bbb T_{p^{n-1}-1}\in \Ver_{p^{n}}$. 
This is a simple projective object. On the other hand, in $\Ver_{p^{n+1}}$ it is not projective and 
has projective cover $P_X$. 

\begin{lemma}\label{mult2} 
We have $[P_X:X]=2$. 
\end{lemma} 

\begin{proof}
By Corollary \ref{covers} we have $P_X=\Bbb T_{p^n+p-2}$. Now $[P_X:X]=\dim \Hom(P_X,P_X)=2$ by Proposition \ref{TWquiver}.
\end{proof} 

Now suppose $\Ver_{p^n}\subset \C\subset \Ver_{p^{n+1}}$. Let $Q_X$ be the projective cover of $X$ in $\C$. 
Then we have a diagram of surjective morphisms $P_X\to Q_X\to X$. It is clear that $Q_X$ is self-dual, so by Lemma \ref{mult2} either $Q_X=P_X$ or $Q_X=X$. In the first case every indecomposable projective in $\C$, being a direct summand in $Y\otimes Q_X$ for some simple $Y\in \C$, is projective in $\Ver_{p^{n+1}}$, which implies that 
$\C=\Ver_{p^{n+1}}$ (using the multiplication rule for simple objects from Corollary \ref{multrule1}). On the other hand, if $Q_X=X$ then 
similarly all projectives of $\Ver_{p^n}$ are projective in $\C$, so we get $\C=\Ver_{p^n}$. 

(ii) It suffices to prove the statement for $k=1$, in which case it follows from (i). 
\end{proof} 

\begin{corollary}\label{subcat} The only tensor subcategories of $\Ver_{p^n}$ are 
$\Ver_{p^m}$ and $\Ver_{p^m}^+$ for $m\le n$. 
\end{corollary}  

\begin{proof} For $p=2$ this is proved in \cite{BE}, Corollary 2.6, so let us consider the case $p>2$. 
It suffices to show that the only tensor subcategories of $\Ver_{p^n}^+$ are $\Ver_{p^m}^+$ for $m\le n$. Let $\C\subset \Ver_{p^n}^+$ be a tensor subcategory. 
Let $m$ be the largest integer such that $\C$ contains all simple objects 
of $\Ver_m^+$. Then by Corollary \ref{multrule}, $\C$ cannot contain 
any other simple objects. Thus $\C$ is contained in the Serre closure of $\Ver_{p^m}^+$ 
inside $\Ver_{p^n}^+$. So it follows from Proposition \ref{serre} that $\C=\Ver_{p^m}^+$. 
\end{proof} 

\subsection{${\rm Ext}^1$ between simple objects}

Let us describe a recursive procedure for computing ${\rm Ext}^1$ between simple objects of $\Ver_{p^n}$. We will focus on the case $p>2$ since the case $p=2$ is addressed in \cite{BE}, Subsection 4.5. 
Recall that by Theorem \ref{tpt} the simple objects of $\Ver_{p^n}$ look like $X\otimes \Bbb T_i$ where $X\in \Ver_{p^{n-1}}$ and $0\le i\le p-1$. Also by block considerations for $X,Y\in \Ver_{p^{n-1}}$ 
we have 
$$
{\rm Ext}^m_{\Ver_{p^n}}(X\otimes \Bbb T_i,Y\otimes \Bbb T_j)=0,\ m\ge 0
$$ 
unless $i=j$ or $i=p-2-j$. 
Consider first the case $i=j$. 

\begin{lemma}\label{donk} For each $m$ we have a canonical isomorphism ${\rm Ext}^m_{{\rm Ver}_{p^n}}(X\otimes \Bbb T_{p-1},Y\otimes \Bbb T_{p-1})\cong {\rm Ext}^m_{{\rm Ver}_{p^{n-1}}}(X,Y)$. 
\end{lemma} 

\begin{proof}
This is an immediate consequence of Proposition \ref{not maximal}.
\end{proof} 

\begin{lemma} (i)  If $0\le i\le p-2$ then we have a canonical isomorphism 
$${\rm Ext}^m_{{\rm Ver}_{p^n}}(X\otimes \Bbb T_i,Y\otimes \Bbb T_i)\cong {\rm Ext}^m_{{\rm Ver}_{p^n}}(X,Y).$$

(ii) If $p>2$ and $0\le i\le p-2$ then we have a canonical isomorphism $${\rm Ext}^1_{{\rm Ver}_{p^n}}(X\otimes \Bbb T_i,Y\otimes \Bbb T_i)\cong {\rm Ext}^1_{{\rm Ver}_{p^{n-1}}}(X,Y).$$
\end{lemma} 

\begin{proof} (i) We have 
$$
{\rm Ext}^m_{{\rm Ver}_{p^n}}(X\otimes \Bbb T_i,Y\otimes \Bbb T_i)={\rm Ext}^m_{{\rm Ver}_{p^n}}(X, Y\otimes \Bbb T_i\otimes \Bbb T_i). 
$$
Now, $\Bbb T_i\otimes \Bbb T_i$ can be written as a direct sum of indecomposables, and by block considerations the only 
direct summand which contributes to ${\rm Ext}^\bullet$ is $\be$. This implies the statement. 

(ii) This follows from (i) and Proposition \ref{serre}(ii). 
\end{proof} 

Now consider the case $j=p-2-i$. In this case we can again write 
$$
{\rm Ext}^m_{{\rm Ver}_{p^n}}(X\otimes \Bbb T_i,Y\otimes \Bbb T_j)={\rm Ext}^m_{{\rm Ver}_{p^n}}(X,Y\otimes \Bbb T_i\otimes \Bbb T_j), 
$$
and by block considerations the only indecomposable direct summand in $\Bbb T_i\otimes \Bbb T_j$ that contributes to ${\rm Ext}^\bullet$ is $\Bbb T_{p-2}$. Thus, it suffices to compute 
$$
{\rm Ext}^1_{{\rm Ver}_{p^n}}(X,Y\otimes \Bbb T_{p-2}).
$$
To this end, note that $\Bbb T_1\otimes \Bbb T_{p-1}=[\Bbb T_{p-2},\Bbb T_1^{[n-1]},\Bbb T_{p-2}]$ (Lemma \ref{l12}). 
Thus, 
$$
{\rm Ext}^m_{{\rm Ver}_{p^n}}(X,Y\otimes [\Bbb T_{p-2},\Bbb T_1^{[n-1]},\Bbb T_{p-2}])=
$$
$$
={\rm Ext}^m_{{\rm Ver}_{p^n}}(X,Y\otimes \Bbb T_1\otimes \Bbb T_{p-1})={\rm Ext}^m_{{\rm Ver}_{p^n}}(X\otimes \Bbb T_{p-1},Y\otimes \Bbb T_1)=0
$$
by block considerations. From the long exact sequence it then follows that 
$$
{\rm Ext}^1_{{\rm Ver}_{p^n}}(X,Y\otimes \Bbb T_{p-2})\cong \Hom_{\Ver_{p^n}}(X,Y\otimes [\Bbb T_{p-2},\Bbb T_1^{[n-1]}])\cong 
{\rm Hom}_{{\rm Ver}_{p^{n-1}}}(X,Y\otimes V). 
$$
Finally, this Hom can be computed using the following proposition. 

\begin{proposition} Let $X,Y\in \Ver_{p^{n-1}}$ and $0\le i,j\le p-2$. 
Then ${\rm Hom}(X\otimes L_i, Y\otimes L_j\otimes V)=0$ 
unless $|i-j|=1$. In the latter case, 
${\rm Hom}(X\otimes L_i, Y\otimes L_j\otimes V)=\Hom(X,Y)$. 
\end{proposition} 

\begin{proof} We have 
$$
\Hom(X\otimes L_i, Y\otimes L_j\otimes V)=\Hom(L_i\otimes L_j\otimes V,Y\otimes X^*). 
$$ 
Now, the product $L_i\otimes L_j\otimes V$ decomposes as a direct sum of $\Bbb T_j$ 
for $0\le j\le 2p-1$. But by block considerations for every object $Z\in \Ver_{p^{n-1}}\subset \Ver_{p^n}$, we have $\Hom(\Bbb T_j,Z)=0$ unless $j=0$ or $j=2p-2$, and 
$$
\Hom(\Bbb T_{2p-2},Z)=\Hom(\one,Z).
$$
This implies the statement. 
\end{proof} 

Thus, we obtain the following proposition. 

\begin{proposition}\label{ext1formula} Let $n\ge 2$, $p>2$, $a=a'p+m$, $b=b'p+r$, where $0\le m,r\le p-1$ and $0\le a',b'\le p^{n-2}(p-1)-1$. Then the following hold for 
${\rm Ext}^1$ between simple objects of $\Ver_{p^n}$. 

(i)  ${\rm Ext}^1_{\Ver_{p^n}}(L_a,L_b)=0$ unless $r=m$ or $r=p-2-m$. 

(ii) If $r=m$ then ${\rm Ext}^1_{\Ver_{p^n}}(L_a,L_b)={\rm Ext}^1_{\Ver_{p^{n-1}}}(L_{a'},L_{b'})$. 

(iii) If $r=p-2-m$ then ${\rm Ext}^1_{\Ver_{p^n}}(L_a,L_b)=0$ unless the last digit of $a'$ is $<p-1$ and $b'=a'+1$ or the last digit of $a'$ is $>0$ and $b'=a'-1$. In the latter two cases, ${\rm Ext}^1_{\Ver_{p^n}}(L_a,L_b)=\k$. 

In other words, ${\rm Ext}^1_{\Ver_{p^n}}(L_a,L_b)\ne 0$ if and only if $a$ and $b$ differ only in two 
consecutive digits, of which the first ones differ by $1$ and
the second ones add up to $p-2$, and in this case  ${\rm Ext}^1_{\Ver_{p^n}}(L_a,L_b)=\k$. 
\end{proposition} 

\begin{remark} For $p=2$ the fact that for any $a,b$, $\Ext^1_{\Ver_{p^n}}(L_a,L_b)$ is either $0$ or $\k$ is proved in \cite{BE}, Proposition 4.12. 
\end{remark} 

\begin{example} 1. If $p>2$ then ${\rm Ext}^1_{\Ver_{p^n}}(L_a,L_a)=0$ for all $a$. 

2. Let $n=2$, $p>2$. Then ${\rm Ext}^1_{\Ver_{p^n}}(L_a,L_b)=0$ unless 
$a=kp+c-1$ and $b=(k+2)p-c-1$ or $b=kp+c-1$ and $a=(k+2)p-c-1$ or 
$a=kp+c-1$ and $b=kp-c-1$ or $b=kp+c-1$ and $a=kp-c-1$ for some $1\le c\le p-1$. 
In these cases, ${\rm Ext}^1_{\Ver_{p^n}}(L_a,L_b)=\k$. Note that this agrees with the data in Subsection \ref{psquared}. 
\end{example} 

\subsection{The Grothendieck ring of the stable category}
Given a finite tensor category $\C$, one may consider the Grothendieck ring
${\rm GrStab}(\C)$ of the stable category ${\rm Stab}(\C)$, which by definition is the 
quotient of ${\rm Gr}(\C)$ by the ideal of classes of projective objects. It is clear that as an abelian group, 
${\rm GrStab}(\C)$ is the cokernel of the Cartan matrix $C$ of $\C$. Thus if $C$ is invertible then this 
ring is finite of order $|{\rm det}(C)|$. 

In this section we compute the ring ${\rm GrStab}(\C)$ for $\C=\Ver_{p^n}$ and $\C=\Ver_{p^n}^+$ for $n\ge 2$.\footnote{We thank R. Rouquier for asking this question.} 
Note that for $p>2$ we have ${\rm GrStab}(\Ver_{p^n})={\rm GrStab}(\Ver_{p^n}^+)[u]/(u^2-1)$, so it suffices to compute 
${\rm GrStab}(\Ver_{p^n}^+)$. 

Recall that $q=e^{\pi i/p^n}$ and $[a]_q=\frac{q^a-q^{-a}}{q-q^{-1}}$. 

\begin{proposition}\label{stabrings} (i) For $p>2$ the ring ${\rm GrStab}(\Ver_{p^n}^+)$ is isomorphic to 
$\mathcal{O}_{n,p}/([p^{n-1}]_q)$. 

(ii) The ring ${\rm GrStab}(\Ver_{2^n})$ is isomorphic to 
$\mathcal{O}_{n+1,2}/([2^{n-1}]_q)$.  

(iii) The ring ${\rm GrStab}(\Ver_{2^{n}}^+)$ is isomorphic to 
$\mathcal{O}_{n,2}/((q+q^{-1})[2^{n-1}]_q)$.  
\end{proposition}  

\begin{proof} (i) Recall that ${\rm Gr}(\Ver_{p^n}^+)\cong \mathcal{O}_{n,p}$. 
Also, since 
$$
\Bbb T_1^{\otimes m}\otimes \Bbb T_{p^{n-1}-1}=\Bbb T_{p^{n-1}+m-1}\oplus \coplus_{i=0}^{m-1}b_m\Bbb T_{p^{n-1}-1+i},
$$
the ideal of projectives is generated by $T_{p^{n-1}-1}$. Therefore, the corresponding ideal in $\mathcal{O}_{n,p}$ is generated
by the Frobenius-Perron dimension of $T_{p^{n-1}-1}$, which equals $[p^{n-1}]_q$. This implies the statement. 
 
(ii) The proof is the same as (i), using that ${\rm Gr}(\Ver_{2^n})\cong \mathcal{O}_{n+1,2}$.

(iii) The proof is the same as (ii), using that ${\rm Gr}(\Ver_{2^n}^+)\cong \mathcal{O}_{n,2}$ and replacing $T_{2^{n-1}-1}$ by $T_{2^{n-1}}=T_1\otimes T_{2^{n-1}-1}$ (as 
$T_{2^{n-1}-1}\notin \Ver_{2^{n}}^+$). 
\end{proof} 

Now we would like to compute these rings more explicitly, which is now a problem in elementary number theory. 
First of all, note that $[p^{n-1}]_q=\frac{q^{p^{n-1}}-q^{-p^{n-1}}}{q-q^{-1}}$ divides $p$, since 
already $q^{p^{n-1}}-q^{-p^{n-1}}=e^{\pi i/p}-e^{-\pi i/p}$ divides $p$. Indeed, for  $p>2$ the norm 
of $e^{\pi i/p}-e^{-\pi i/p}$ equals $-p$, while 
$e^{\pi i/2}-e^{-\pi i/2}=2i$ obviously divides $2$. 

Moreover, for $p=2$ the number $(q+q^{-1})[2^{n-1}]_q$ divides $2$. 
Indeed, $\frac{2}{(q+q^{-1})[2^{n-1}]_q}=-i\frac{q^2-1}{q^2+1}$. But the number $x=\frac{q^2-1}{q^2+1}$ is a unit, since 
$q^2=\frac{1+x}{1-x}$, thus $(1+x)^{2^{n-1}}=-(1-x)^{2^{n-1}}$, i.e., 
$P(x):=\frac{1}{2}((1+x)^{2^{n-1}}+(1-x)^{2^{n-1}})=0$, while $P$ is a monic polynomial with integer coefficients and constant term $1$. 

Thus, the above rings are $\Bbb F_p$-algebras. So in Proposition \ref{stabrings}(i) the ring $R$ in question is 
the quotient of $\Bbb F_p[q^2+q^{-2}]$ by the relations 
$$
(q-q^{-1})^{p^{n-1}-1}=0,\ q^{-(p-1)p^{n-1}}(1+q^{2p^{n-1}}+...+q^{2(p-1)p^{n-1}})=0.
$$
The second relation can be rewritten in characteristic $p$ as $(q-q^{-1})^{(p-1)p^{n-1}}=0$, hence it follows from the first relation. Thus 
$$
R=\Bbb F_p[z]/z^{\frac{p^{n-1}-1}{2}}
$$ 
with $z:=(q-q^{-1})^2$. 

In Proposition \ref{stabrings}(ii) we similarly get 
$$
R=\Bbb F_p[w]/(w^{2^{n-1}-1},\ w^{2^{n-1}})=\Bbb F_p[w]/w^{2^{n-1}-1},
$$ 
where $w=q+q^{-1}$. 

Finally, in Proposition \ref{stabrings}(iii) we get 
$$
R=\Bbb F_p[z]/(z^{2^{n-1}},z^{2^{n-2}})=\Bbb F_p[z]/z^{2^{n-2}} 
$$
with $z:=(q-q^{-1})^2=q^2+q^{-2}$. 

Thus we obtain 

\begin{proposition}\label{stabrings1} (i) For $p>2$ the ring ${\rm GrStab}(\Ver_{p^n}^+)$ is isomorphic to $\Bbb F_p[z]/z^{\frac{p^{n-1}-1}{2}}$ and the ring 
${\rm GrStab}(\Ver_{p^n})$ is isomorphic to $\Bbb F_p[z,g]/(z^{\frac{p^{n-1}-1}{2}},g^2-1)$. 

(ii) The ring ${\rm GrStab}(\Ver_{2^n})$ is isomorphic to 
$\Bbb F_p[z]/z^{2^{n-1}-1}$.  

(iii) The ring ${\rm GrStab}(\Ver_{2^{n}}^+)$ is isomorphic to 
$\Bbb F_p[z]/z^{2^{n-2}}$.  
\end{proposition}  

\begin{proposition}\label{detcar} The determinant of the Cartan matrix of 
any block of $\Ver_{p^n}$ of size $p^{m-1}(p-1)$ equals $p^{p^{m-1}}$. 
\end{proposition} 

\begin{proof} By Proposition \ref{samecar}, all blocks of the same size have the same Cartan matrices. So the determinant in question depends only on $m$; let us denote it by $D_m$. Then Proposition \ref{stabrings1} and Proposition \ref{blockstr} imply that 
$$
(\prod_{m=1}^{n-1}D_m)^{p-1}=p^{p^{n-1}-1}.
$$
Thus $D_m^{p-1}=p^{p^{m-1}(p-1)}$. Also since $C_{n,p}$ is positive definite (Proposition \ref{Cartan nondeg}(iv)), $D_m$ is a positive integer. 
Thus $D_m=p^{p^{m-1}}$, as claimed. 
\end{proof}

\subsection{Incompressibility} Let $\D$ be a (not necessarily braided) tensor category. 

\begin{theorem} \label{incompressible}
(i) Assume $p>2$. Let 
$E: \Ver_{p^n}^+\to \D$ be a (not necessarily braided) tensor functor. Then $E$ is an embedding (i.e., it is fully faithful).

(ii) Let $\D$ be braided and $E: \Ver_{p^n}\to \D$ be a braided tensor functor. Then $E$ is an embedding. 
\end{theorem} 

\begin{remark} \label{incompress rem}
(i) In the language of \cite[Definition 4.3]{BE} Theorem \ref{incompressible} says that the category $\Ver_{p^n}^+$ is {\em incompressible} as a tensor category, and $\Ver_{p^n}$ is incompressible as a symmetric (or braided) category.

(ii) Theorem \ref{incompressible}(i) is not true if $\Ver_{p^n}^+$ is replaced by $\Ver_{p^n}$: since there exists a tensor functor $\sVec\to {\rm \Vec}$, we have a non-injective tensor functor $\Ver_{p^n}=\Ver_{p^n}^+\boxtimes \sVec\to \Ver_{p^n}^+$, see Corollary \ref{Supervec}.

(iii) See \cite[Theorem 4.4]{BE} for the case $p=2$. Note that in this case $\Ver_{2^n}$ is incompressible as a tensor category while $\Ver_{2^n}^+$ is not. However, both are incompressible as symmetric (or braided) categories. 
\end{remark} 

\begin{proof} (ii) easily follows from (i), so it suffices to prove (i). 
By replacing $\D$ with the image of $E$ we can assume that $\D$ is finite
and $E$ is surjective. Moreover by Corollary \ref{Supervec} we can extend $E$ to a tensor functor $R: \Ver_{p^n}\to \D \boxtimes \sVec$.  

\begin{lemma} \label{FP tricks}
(i) Assume $i\in [0,p-1]$. Then $R(\Bbb T_i)$ is simple and for $i>0$ it is not isomorphic
to an object from $R(\Ver_{p^{n-1}})$.

(ii) Let $a\in [p-1,2p-2]$ and let $L$ be a simple object from $R(\Ver_{p^{n-1}})\subset \D$. 
Assume $\Hom(R(\Bbb T_a),L)\ne 0$. Then $a=2p-2$, $L=\one$ and  ${\rm Hom}(R(\Bbb T_a),L)$
is one-dimensional.
\end{lemma}

\begin{proof} (i) We employ induction in $i$. The base case $i=0$ is clear. Assume $R(\Bbb T_{i+1})$ 
is not simple. Thus it has a nontrivial subobject $N$; clearly either $\FPdim(N)\le \frac12\FPdim(\Bbb T_{i+1})$ or $\FPdim(R(\Bbb T_{i+1})/N)\le \frac12\FPdim(\Bbb T_{i+1})$; since $R(\Bbb T_{i+1})$
is self-dual, in either case we will have a quotient $M$ with $\FPdim(M)\le \frac12\FPdim(\Bbb T_{i+1})$.
By \eqref{eq2} we have a nonzero morphism $R(\Bbb T_{1})\ot R(\Bbb T_{i})\to M$, so we have a 
nonzero morphism $R(\Bbb T_{i})\to R(\Bbb T_{1})\ot M$. This must be an embedding as 
$R(\Bbb T_{i})$ is simple by the induction assumption. We have
$$\FPdim(R(\Bbb T_{1})\ot M/R(\Bbb T_{i}))=
$$
$$
=[2]_q\FPdim(M)-[i+1]_q\le \frac{[2]_q[i+2]_q}2-[i+1]_q=\frac{[i+3]_q-[i+1]_q}2<1,$$
so it must vanish and $\FPdim(M)=\frac{[i+1]_q}{[2]_q}$. 
But we claim that the number $\alpha:=\frac{[i+1]_q}{[2]_q}=\frac{q^{m+2}-q^{-m-2}}{q^2-q^{-2}}$ cannot be the Frobenius-Perron dimension of any object. 
Indeed, let $g$ be the Galois group element such that $g(q^4)=q^2$ (it exists since the order of $q^2$ is a power of $p$). Then  
$g(q^{2m+4})=(-1)^mq^{m+2}$, so we have 
$$
|g(\alpha)|=\left|\frac{(-1)^mq^{m+2}-1}{q^2-1}\right|=|\alpha|\cdot \left|\frac{q^2+1}{(-1)^mq^{m+2}+1}\right|>|\alpha|,
$$
as desired. Thus $R(\Bbb T_{i+1})$ must be simple and (i) is proved.

(ii) We already know that the statement holds for $a=p-1$ by (i). For $a\in [p,2p-2]$ the object
$\Bbb T_a$ has 3 simple components: $\Bbb T_{2p-2-a}$ with multiplicity 2 and $\Bbb T_1^{[n-1]}\ot
\Bbb T_{a-p}$, see Example \ref{tptex}. If $L\ne \one$ or $a\ne 2p-2$ then $\Hom(R(\Bbb T_{2p-2-a}),L)=0$, so we must
have $\Hom(R(\Bbb T_1^{[n-1]}\ot \Bbb T_{a-p}),L)\ne 0$ which implies $\Hom(R(\Bbb T_{a-p}),R(\Bbb T_1^{[n-1]})\ot L)\ne 0$ which is impossible for $a\ne p$ by (i) as $R(\Bbb T_1^{[n-1]})\ot L\in R(\Ver_{p^{n-1}})$.

It remains to show that $\Hom(R(\Bbb T_p),R(\Bbb T_1^{[n-1]}))=0$ and $\dim\Hom(R(\Bbb T_{2p-2}),\one)=1$. For the former we observe that $\Bbb T_p=\Bbb T_1\ot \Bbb T_{p-1}$ by \eqref{eq1},
so $\Hom(R(\Bbb T_p),R(\Bbb T_1^{[n-1]}))\ne 0$ would imply that $R(\Bbb T_{p-1})\subset R(\Bbb T_1\ot \Bbb T_1^{[n-1]})$. This is impossible for $p\ge 5$ since 
$$\FPdim(\Bbb T_{p-1})>\FPdim(\Bbb T_1\ot \Bbb T_1^{[n-1]});$$
 it is also impossible for $p=3$ since $0<\FPdim(\Bbb T_1\ot \Bbb T_1^{[n-1]})-
\FPdim(\Bbb T_{p-1})<1$ in this case. Finally, $\Hom(R(\Bbb T_{2p-2}),\one)$ is a direct summand
of $\Hom(R(\Bbb T_{p-1})\ot R(\Bbb T_{p-1})^*,\one)=\Hom(R(\Bbb T_{p-1}), R(\Bbb T_{p-1})),$ so it
is one-dimensional by (i).
\end{proof}

Lemma \ref{FP tricks}(i) implies that the restriction of $R$
to $\Ver_p\subset \Ver_{p^n}$ is an embedding.
Thus Theorem \ref{incompressible} is implied by the following result.

\begin{proposition}\label{incompress2}
 Let $R: \Ver_{p^n}\to \D$ be a surjective tensor functor such that
its restriction to $\Ver_p\subset \Ver_{p^n}$ is an embedding. Then $R$ is an equivalence.
\end{proposition}

\begin{proof} We proceed by induction in $n$. The base case $n=1$ is immediate. By applying
the inductive assumption to the image of the restriction of $R$ to $\Ver_{p^{n-1}}$ we can
assume that $R$ restricted to $\Ver_{p^{n-1}}$ is an embedding.

It is known that the functor $R$ sends projective objects to projective objects, see \cite[Theorem 6.1.16]{EGNO}.
To prove the proposition it remains to show that $R$ is fully faithful on projectives. Note
that it is faithful automatically by \cite[Remark 4.3.10]{EGNO}, so let us 
show that it is full on projectives. Since $\Hom(Q_1,Q_2)=\Hom(Q_1\ot Q_2^*,\one)$, it
is sufficient to show that for any projective $P\in \Ver_{p^n}$ the map $\Hom(P,\one)\to \Hom(R(P),\one)$
is surjective. We can assume that $P$ is indecomposable. Thus it is enough to show that for
$p^{n-1}-1\le r\le p^n-2$ we have $\Hom(R(\Bbb T_r),\one)=0$ except for the case $r=2p^{n-1}-2$
when $\Hom(R(\Bbb T_r),\one)$ should be one-dimensional, see Lemma \ref{lem2}. Let us write
$r=a+pb$ where $p-1\le 2p-2$ and $p^{n-2}-1\le b\le p^{n-1}-2$. By \eqref{Frobenii} we have
$\Hom(R(\Bbb T_r),\one)=\Hom(R(\Bbb T_a),R(\Bbb F(\Bbb T_b))^*)$ where $R(\Bbb F(\Bbb T_b))\in
R(\Ver_{p^{n-1}})$ by definition of the functor $H$, see Theorem \ref{inclusion}. Thus by Lemma 
\ref{FP tricks} $\Hom(R(\Bbb T_a),R(\Bbb F(\Bbb T_b))^*)\ne 0$ implies that $a=2p-2$ and any
nonzero morphism factorizes through $\one$. Using the inductive assumption we see that $b=2p^{n-2}-2$
and the nonzero Hom space is one-dimensional. Thus the functor $R$ is full on projectives, which
proves Proposition \ref{incompress2}. 
\end{proof}
Theorem \ref{incompressible} is proved. 
\end{proof}

\subsection{Classification of symmetric tensor categories generated by an object with invertible exterior square} 

The goal of this subsection is to classify symmetric tensor categories $\C$ over an algebraically closed field $\k$ of any characteristic $p\ge 0$ generated by an object $X$ with invertible $\wedge^2X$. 

\begin{proposition}\label{prop3} Let $X$ be a symmetric TL object in $\C$ of degree $n$, and assume that $\C$ is tensor generated by $X$. Then $\C\cong \Ver_{p^n}$ if $n<\infty$ (with $X=\Bbb T_1$ and $p^n>2$) and $\C\cong {\rm Rep}(\Gamma)$ for a closed subgroup scheme $\Gamma\subset SL_2(\k)$ if $n=\infty$ (with $X$ being the tautological representation). This gives a classification of such pairs $(\C,X)$. 
\end{proposition} 

\begin{proof} Since $\Ver_{p^n}$ is the abelian envelope of $\T_{n,p}$ (Theorem \ref{CartanFP0}) 
and ${\rm Rep}(SL_2(\k))$ is the abelian envelope of 
$\T_{\infty,p}$ (\cite{CEH}, Theorem 3.3.1), 
we have a surjective symmetric tensor functor
$F: \Ver_{p^n}\to \C$ if $n<\infty$ and ${\rm Rep}(SL_2(\k))\to \C$ 
if $n=\infty$. In the first case, since $\Ver_{p^n}$ is incompressible (Theorem \ref{incompressible}), 
$F$ is an equivalence and we are done. In the case $n=\infty$, 
$\C$ is Tannakian (\cite{Bez}, Proposition 1), and thus has to be of the form ${\rm Rep}(\Gamma)$ where $\Gamma$ is a closed subgroup scheme of $SL_2(\k)$, as claimed. 
\end{proof} 

\begin{corollary} \label{class1}
Let $\C$ be a symmetric tensor category over $\k$ tensor generated by an object $X$ with $\wedge^2X=\one$. Then $(\C,X)$ is exactly one of the following: 

(i) $\C=\Ver_{p^n}$, $p^n>2$, $X=\Bbb T_1$; 

(ii) $\C={\rm sVec}$, $X$ the fermion ($p\ne 2,3$); 

(iii) $\C={\rm Rep}(\Gamma)$ where $\Gamma$ is a closed subgroup scheme 
of $SL_2(\k)$, $X$ the tautological representation. 

This gives a classification of such pairs $(\C,X)$.
\end{corollary} 

\begin{proof} This follows from Propositions \ref{prop1}, \ref{prop2} and \ref{prop3}.  
\end{proof} 

Let us now generalize this to the case when $\wedge^2 Y=L$ is an invertible object. 

\begin{lemma}\label{lemm3} $L$ is even, i.e., $c_{L,L}=1$. 
\end{lemma} 

\begin{proof} Assume $L$ is odd, so $p\ne 2$. 
Let $\dim Y=d$. We have $\dim \wedge^2 Y=d(d-1)/2$, so $d(d-1)/2=-1$. This equation has no solutions mod $3$, which means that $p\ne 3$. 

By Lemma \ref{lemm0}, $Y$ is simple or an extension $[\chi,\psi]$ between even invertible objects such that $\chi\otimes \psi=L$. In the latter case $L$ is even as well, so this case is impossible. 

It remains to consider the case when $Y$ is simple. In this case, if 
$\wedge^3Y=0$ then we have $\dim \wedge^3Y=d(d-1)(d-2)/6=0$, so 
$d=0,1$ or $2$. But then $d(d-1)/2$ equals $0$ or $1$, so never to $-1$, i.e., this case is ruled out. Thus $\wedge^3Y\ne 0$. Since $\wedge^3Y$ is a subobject of the simple object $\wedge^2Y\otimes Y=L\otimes Y$, 
we have $\wedge^3Y=L\otimes Y$. 
Thus $d(d-1)(d-2)/6=-d$, i.e., $d(d+1)(d+2)=0$. So $d=0,-1$ or $-2$, i.e., $d(d-1)/2=0,1,3$. This never equals $-1$, which is a contradiction. 
\end{proof} 

\begin{example} \label{cyclic twist}
For an abelian group $A$ let $\Vec_A$ be the corresponding pointed symmetric tensor category with the trivial associator and braiding (i.e., the category of $A$-graded vector spaces). Assume that $A$ is cyclic with a generator $g$. Let $(\C,X)$ be as in Corollary \ref{class1}(i) or (ii). Then
the object $Y:=X\boxtimes g\in \C \boxtimes \Vec_A$ satisfies $\wedge^2 Y=\one \boxtimes g^2$, so
$\wedge^2 Y$ is invertible.

Let $\C'=\C'_A$ be the subcategory of $\C \boxtimes \Vec_A$ generated by $Y$. Thus $\C'=\C \boxtimes \Vec_A$ if and only if $A$ is finite of odd order. Furthermore, the pairs $(\C'_A,Y)$ and $(\C'_B,Y)$ are
equivalent if and only if they come from the same $(\C,X)$ and either $A\simeq B$ or one of these groups has odd order and index $2$ in the other. 
\end{example}

\begin{theorem} \label{class2}
Let $\C$ be a symmetric tensor category tensor generated by an object $Y$ with invertible $\wedge^2Y$. Then $(\C,Y)$ is either one of $(\C'_A,Y)$ constructed in Example \ref{cyclic twist} or $\C={\rm Rep}(\Gamma)$ where $\Gamma$ is a closed subgroup scheme 
of $GL_2(\k)$, $Y$ the tautological representation. This gives a classification of such pairs $(\C,Y)$. 
\end{theorem}

\begin{proof}  We will show first that $\C$ embeds into a category $\widetilde \C$ which contains an even
object which is a ``square root'' of $\wedge^2Y$. Namely let $\D=\C \boxtimes \Vec_{\Bbb Z}$ where
$\Bbb Z$ is the infinite cyclic group with generator $s=1$. The category $\D$ contains a subcategory
generated by the invertible object $\chi=\wedge^2Y\boxtimes s^{-2}$. In view of Lemma \ref{lemm3}
this subcategory is equivalent to $\Vec_\BZ=\Rep(\Bbb G_m)$. Now let $\widetilde \C$ be the de-equivariantization of
$\D$ with respect to this subcategory (i.e., with respect to $\Bbb G_m$); in other words, 
$\widetilde\C$ is the category of finitely generated modules over the algebra $\mathcal A:=\oplus_{i\in \Bbb Z}\chi^{\otimes i}$ in ${\rm Ind}(\C)$. We have an obvious embedding functor $\C \hookrightarrow \widetilde \C$;
also it is clear that $\wedge^2Y\simeq s^2$ in the category $\widetilde \C$. The categories
$\D$ and $\widetilde \C$ are clearly generated by $Y$ and $s$.

Consider the object $Z:=Y\ot s^{-1}\in \widetilde \C$. Then, since the object $s$ is even, 
we have $\wedge^2Z=\one$. Thus the subcategory $\langle Z\rangle $ tensor generated by $Z$ 
is one of the categories described in Corollary \ref{class1}. Also the category $\widetilde \C$ is generated
by $\langle Z\rangle $ and by $s$, that is, we have a surjective symmetric tensor functor $\langle Z\rangle \boxtimes \Vec_A\to \widetilde \C$ where $A=\langle s\rangle$ is a cyclic group. If the category $\langle Z\rangle $ is as in Corollary \ref{class1} (iii) then the categories $\widetilde \C$ and $\C$ are Tannakian and we are done by \cite{Bez}, Proposition 1. On the other hand, 
if the category $\langle Z\rangle $ is as in Corollary \ref{class1} (i) or (ii) then we have an equivalence $\widetilde \C \simeq \langle Z\rangle \boxtimes \Vec_A$ since the category
$\langle Z\rangle$ has no even invertible objects. Thus we are in the situation of Example \ref{cyclic twist} and again we are done.
\end{proof}


\section{Examples}\label{s5}

\subsection{The category $\Ver_{p^2}$}\label{psquared}

The category $\Ver_p$, as already mentioned, is semisimple. The simple objects
are $L_0,\dots,L_{p-2}$, where $L_0=\one$ is the tensor identity.\footnote{Note that this labeling differs by a shift from the one in some other papers, for example \cite{EO}, where the simple objects of $\Ver_p$ are denoted $L_1=\one,L_2,...,L_{p-1}$.} The tensor
products are determined by the formula 
$ L_1 \otimes L_m = \begin{cases} L_1 & m=0 \\
L_{m-1}\oplus L_{m+1} & 0 < m < p-2. \\
L_{p-3} & m=p-2 \end{cases} $

The Frobenius-Perron dimension of $L_m$ is $\frac{\sin(\pi (m+1) /p)}{\sin(\pi/p)}$.

The category $\Ver_{p^2}$ is not semisimple, but has finite representation type.
There are $p-1$ semisimple blocks, corresponding to the projective simple
modules $L_{p-1},L_{2p-1},\dots,L_{p^2-1}$,
and $p-1$ non-semisimple blocks, each of whose basic algebra is a Brauer
tree algebra with inertial index $p-1$, so that there is no exceptional vertex. 
The tree is a straight line with $p$ vertices. Thus we obtain 
\begin{proposition}\label{blocksp2} The non-semisimple blocks of $\Ver_{p^2}$ are 
equivalent to the non-semisimple block of the group algebra of the 
symmetric group of degree $p$. In particular, such a block has $p^2-p$ indecomposable objects.
\end{proposition} 
These results can be obtained from Proposition \ref{standfacts} and the results of Subsection \ref{Donkin1}.

Here is the structure of the tilting modules for $p=3$ and $p=5$, with $n=2$. 
The ones we need to use
are $T_m$ with $p^{n-1}-1 \le m \le p^n-2$. These 
are the ones below the double line. In these diagrams, we denote by $[m]$
the simple module with highest weight $m$. This information can
be read off from the paper of Doty and Henke~\cite{DoH}. 

{\tiny
\[ p=3:\qquad
\begin{array}{|c|c|c|}
\hline && \\
& T_0=[0] & T_1=[1]
\\ && \\ \hline\hline && \\
T_2=[2] &
T_3=\begin{array}{c}[1]\\{}[3]\\{}[1]\end{array} &
T_4=\begin{array}{c}[0]\\{}[4]\\{}[0]\end{array} 
\\ && \\ \hline && \\
T_5=[5] &
T_6=\begin{array}{c}[4]\\{}[0]\ [6]\\{}[4]\end{array} &
T_7=\begin{array}{c}[3]\\{}[1]\ [7]\\{}[3]\end{array}
\\ && \\ \hline
\end{array} \]
}

{\tiny
\[ p=5:\qquad
\begin{array}{|c|c|c|c|c|} 
\hline &&&& \\ 
& T_0=[0] & T_1 = [1] & T_2 = [2] & T_3 = [3] 
\\ &&&& \\ \hline\hline &&&& \\
T_4 = [4] &
T_5=\begin{array}{c}[3]\\{}[5]\\{}[3]\end{array} &
T_6=\begin{array}{c}[2]\\{}[6]\\{}[2]\end{array} &
T_7=\begin{array}{c}[1]\\{}[7]\\{}[1]\end{array} &
T_8=\begin{array}{c}[0]\\{}[8]\\{}[0]\end{array} 
\\ &&&& \\ \hline &&&& \\
T_9=[9] &
T_{10}=\begin{array}{c}[8]\\{}[0]\ [10]\\{}[8]\end{array} &
T_{11}=\begin{array}{c}[7]\\{}[1]\ [11]\\{}[7]\end{array} &
T_{12}=\begin{array}{c}[6]\\{}[2]\ [12]\\{}[6]\end{array} &
T_{13}=\begin{array}{c}[5]\\{}[3]\ [13]\\{}[5]\end{array} 
\\ &&&& \\ \hline &&&& \\
T_{14}=[14] &
T_{15}=\begin{array}{c}[13]\\{}[5]\ [15]\\{}[13]\end{array} &
T_{16}=\begin{array}{c}[12]\\{}[6]\ [16]\\{}[12]\end{array} &
T_{17}=\begin{array}{c}[11]\\{}[7]\ [17]\\{}[11]\end{array} &
T_{18}=\begin{array}{c}[10]\\{}[8]\ [18]\\{}[10]\end{array} 
\\ &&&& \\ \hline &&&& \\
T_{19}=[19] &
T_{20}=\begin{array}{c}[18]\\{}[10]\ [20]\\{}[18]\end{array} &
T_{21}=\begin{array}{c}[17]\\{}[11]\ [21]\\{}[17]\end{array} &
T_{22}=\begin{array}{c}[16]\\{}[12]\ [22]\\{}[16]\end{array} &
T_{23}=\begin{array}{c}[15]\\{}[13]\ [23]\\{}[15]\end{array} 
\\ &&&& \\ \hline
\end{array} \]
}
\bigskip

In general if $p$ is odd, for $p-1\le ap+b-1 \le p^2-2$, 
with $1\le a\le p-1$ and $0\le b \le p-1$, the structure is
{\tiny
\[ T_{ap+b-1} = \begin{array}{c} [ap-b-1] \\{} [(a-2)p+b-1]\qquad [ap+b-1]\qquad \\{} [ap-b-1]\end{array} \]
}
unless $b=0$, in which case $T_{ap+b-1}=[ap+b-1]$. Also, 
if $a=1$ then $(a-2)p+b-1<0$, and we 
omit this composition factor so that we are left with a uniserial module.

The simple modules for $A=\End_{SL_2(\k)}\left(\coplus_{m=p-1}^{p^2-2}T_m\right)$ correspond to the
indecomposable tilting modules $T_m$ with $p-1\le m \le p^2-2$. We label each simple module
with the isomorphism type of the top composition factor of the corresponding projective object $\Bbb T_m:=F(T_m)$. 
So the simple corresponding to $\Bbb T_{ap+b-1}$ is labelled $L_{ap-b-1}$. This has the effect that
the simples are $L_i$ with $0\le i \le p^2-p-1$ (this agrees with the labeling in Subsection \ref{tenprot}). 
The structure of their projective covers is as
follows. If $i=ap-1$ then $L_i$ is a projective simple in a block of its own. 
The block containing $L_{p-b-1}$ with $1\le b\le p-1$ also contains the modules
$L_{(2a+1)p\pm b -1}$, which are the ones whose highest weights are
in the same orbit of the affine Weyl group under the dot action. The Brauer tree is as follows:
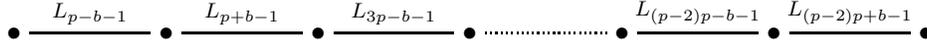
\[ \xymatrix@C=16mm{\bullet \ar@{-}[r]^{L_{p-b-1}}& 
\bullet \ar@{-}[r]^{L_{p+b-1}}& \bullet\ar@{-}[r]^{L_{3p-b-1}} &
\bullet \ar@{.}[r] & \bullet\ar@{-}[r]^{L_{(p-2)p-b-1}} &
\bullet \ar@{-}[r]^{L_{(p-2)p+b-1}} & \bullet} \]
Thus the structure of the projective indecomposables is:
{\tiny
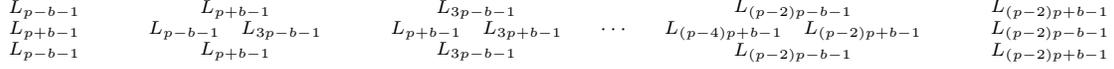
\[ \begin{array}{c}L_{p-b-1}\\L_{p+b-1}\\L_{p-b-1} \end{array}\qquad
\begin{array}{c}L_{p+b-1}\\L_{p-b-1}\quad L_{3p-b-1}\\ L_{p+b-1} \end{array}\qquad
\begin{array}{c}L_{3p-b-1}\\L_{p+b-1}\quad L_{3p+b-1}\\L_{3p-b-1}\end{array}\quad
\cdots
\quad
\begin{array}{c}L_{(p-2)p-b-1}\\L_{(p-4)p+b-1}\quad L_{(p-2)p+b-1}\\L_{(p-2)p-b-1}\end{array}
\qquad
\begin{array}{c}L_{(p-2)p+b-1}\\L_{(p-2)p-b-1}\\L_{(p-2)p+b-1}\end{array}
\]
}
Projective resolutions in this category are formed in the usual way by walking around the Brauer tree,
see Green~\cite{G}. In particular, we have 

\begin{proposition} \label{extp2}
If $p$ is odd then the algebra ${\rm Ext}^\bullet_{\Ver_{p^2}}(\one,\one)$ is isomorphic 
to $\k[x,\xi]$ where $x$ is an even generator of degree $2p-2$ and $\xi$ is an odd generator of degree $2p-3$ (thus $\xi^2=0$). 
In particular, the Hilbert series of this algebra is $\frac{1+t^{2p-3}}{1-t^{2p-2}}$. 
\end{proposition} 

For $p=3$, the tensor products in $\Ver_{p^2}$ are given by the
following table, in which
$P_n$ denotes the projective cover of the simple module $L_n$.\par
{\tiny
\[ \renewcommand{\arraystretch}{1.2}
\begin{array}{|c|c|c|c|c|c|} \hline
L_0&L_1&L_2&L_3&L_4&L_5\\ \hline
L_1& L_0\oplus L_2 & P_3 & L_4 & L_3 \oplus L_5 & P_4 \\ \hline
L_2& P_3&L_2\oplus P_0&L_5&P_4& L_5\oplus P_3 \\ \hline
L_3& L_4&L_5&L_0&L_1&L_2 \\ \hline
L_4& L_3\oplus L_5&P_4&L_1&L_0\oplus L_2&P_3 \\ \hline
L_5 & P_4 &L_5\oplus P_3&L_2&P_3&L_2\oplus P_0 \\ \hline
\end{array} \]
}
Note that $L_3$ satisfies $\Lambda^2(L_3)=L_0$ and $S^2(L_3)=0$, so
it generates $\sVec\subset \Ver_9$, and tensoring with $L_3$ acts as a parity change (recall that in general the generator of  $\sVec\subset \Ver_{p^n}$ is $L_{p^{n-1}(p-2)}$).

For $p=5$, we only list the tensor products of simples in
$\Ver_{p^2}^+$, together with their tensor products with the parity
change module $L_{15}$, which satisfies $\Lambda^2(L_{15})=L_0$, 
$S^2(L_{15})=0$.\par
{\tiny
\[ \setlength{\arraycolsep}{0.5mm}
\renewcommand{\arraystretch}{1.2}
\begin{array}{|c|c|c|c|c|c|c|c|c|c|} \hline
L_0 &L_2&L_4&L_6&L_8&L_{10}&L_{12}&L_{14}&L_{16}&L_{18} 
\\ \hline
L_2&
\begin{array}{c}L_0\oplus L_2\\{}\oplus L_4\end{array}&
L_4\oplus P_2&L_6 \oplus L_8&L_6\oplus P_8&L_{12}&
\begin{array}{c}L_{10}\oplus L_{12}\\{}\oplus L_{14}\end{array}&
L_{14} \oplus P_{12}&L_{16}\oplus L_{18}&L_{16}\oplus P_{18} 
\\ \hline
L_4&L_4\oplus P_2&
\begin{array}{c}L_4\oplus P_0\\{}\oplus P_2\end{array}&
P_8&P_6\oplus P_8&L_{14}&L_{14}\oplus P_{12}&
\begin{array}{c}L_{14}\oplus P_{10}\\{}\oplus P_{12}\end{array}&
P_{18}&P_{16}\oplus P_{18} 
\\ \hline
L_6&L_6\oplus L_8&P_8&
\begin{array}{c}L_0\oplus L_2\oplus{}\\L_{10}\oplus L_{12}\end{array}&
\begin{array}{c}L_2\oplus L_4\oplus{}\\L_{12}\oplus  L_{14}\end{array}&
L_6\oplus L_{16}&
\begin{array}{c}L_6\oplus L_8\oplus{}\\L_{16}\oplus L_{18}\end{array}&
P_8\oplus P_{18}&L_{10}\oplus L_{12}&L_{12}\oplus L_{14}
\\ \hline
L_8&L_6\oplus P_8&P_6\oplus P_8&
\begin{array}{c}L_2\oplus L_4\oplus{}\\L_{12}\oplus L_{14}\end{array}&
\begin{array}{c}L_0 \oplus L_4 \oplus{}\\L_{10}\oplus L_{14}\oplus{}\\
P_2\oplus P_{12}\end{array}&
L_8\oplus L_{18}&
\begin{array}{c}L_6\oplus L_{10}\oplus{}\\P_8\oplus P_{18}\end{array}&
\begin{array}{c}P_6\oplus P_8 \oplus{}\\P_{16}\oplus P_{18}\end{array}&
L_{12} \oplus L_{14}&
\begin{array}{c}L_{10}\oplus L_{14}\\{}\oplus P_{12}\end{array}
\\ \hline
L_{10}&L_{12}&L_{14}&L_6\oplus L_{16}&L_8\oplus L_{18}&
L_0\oplus L_{10}&L_2\oplus L_{12}&L_4\oplus L_{14}&L_6&L_8 
\\ \hline
L_{12}&\begin{array}{c}L_{10}\oplus L_{12}\\{}\oplus L_{14}\end{array}&
L_{14}\oplus P_{12}&
\begin{array}{c}L_6\oplus L_8 \oplus{}\\L_{16}\oplus L_{18}\end{array}&
\begin{array}{c}L_6\oplus L_{10}\oplus{}\\P_8\oplus P_{18}\end{array}&
L_2\oplus L_{12}&
\begin{array}{c}L_0\oplus L_2 \oplus{}\\L_4\oplus L_{10}\oplus{}\\
L_{12}\oplus L_{14}\end{array}&
\begin{array}{c}L_4\oplus L_{14}\oplus{}\\P_2\oplus P_{12}\end{array}&
L_6\oplus L_8&L_6\oplus P_8
\\ \hline
L_{14}&L_{14}\oplus P_{12}&
\begin{array}{c}L_{14}\oplus P_{10}\\{}\oplus P_{12}\end{array}&
P_8\oplus P_{18}&
\begin{array}{c}P_6\oplus P_8 \oplus{}\\P_{16}\oplus P_{18}\end{array}&
L_4\oplus L_{14}&
\begin{array}{c}L_4\oplus L_{14}\oplus{}\\P_2\oplus P_{12}\end{array}&
\begin{array}{c}L_4\oplus L_{14}\oplus{}\\P_0\oplus P_2\oplus{}\\
P_{10}\oplus P_{12}\end{array}&P_8&P_6\oplus P_8
\\ \hline
L_{16}&L_{16}\oplus L_{18}&P_{18}&
L_{10}\oplus L_{12}&L_{12} \oplus L_{14}&
L_6&L_6\oplus L_8&P_8&L_0\oplus L_2&L_2\oplus L_4 
\\ \hline
L_{18}&L_{16}\oplus P_{18}&P_{16}\oplus P_{18}
&L_{12}\oplus L_{14}&
\begin{array}{c}L_{10}\oplus L_{14}\\{}\oplus P_{12}\end{array}&
L_8&L_6\oplus P_8&P_6\oplus P_8&L_2\oplus L_4&
\begin{array}{c}L_0\oplus L_4\\{}\oplus P_2\end{array} 
\\ \hline\hline
L_{15}&L_{17}&L_{19}&L_{11}&L_{13}&L_5&L_7&L_9&L_1&L_3 
\\ \hline
\end{array} \]
}

\begin{remark} \label{simples for p squared}
The above tables suggest that the tensor product of two or more simple objects in the category
$\Ver_{p^2}$ is a direct sum of simple objects and projective objects. This is indeed the case. Namely, by Theorem
\ref{tpt} any simple object of $\Ver_{p^2}$ factors as $X_1\ot X_2$ where $X_1\in \T_{2,p}\subset \Ver_{p^2}$ and $X_2\in \Ver_p\subset \Ver_{p^2}$. The result then follows since both subcategories
$\T_{2,p}$ and $\Ver_p$ are closed under the tensor product, and the indecomposable objects
of $\Ver_p$ are simple while the indecomposable objects of $\T_{2,p}$ are either simple or projective.
A similar argument implies that a tensor product of simple objects in the category
$\Ver_{p^n}$ is a direct sum of a negligible object (i.e., direct sum of indecomposable objects of dimension $0$) and simple objects. 

Recall that the {\em semisimplification} (see \cite{EO1}) of the category $\T_{2,p}$ is $\Ver_p$. Thus
the simple objects of $\Ver_{p^2}$ generate a subcategory $\Ver_p \boxtimes \Ver_p$ in the semisimplification $\overline{\Ver_{p^2}}$ of the category $\Ver_{p^2}$. Similarly the simple objects
of $\Ver_{p^n}$ generate a subcategory $\Ver_p^{\boxtimes n}$ in the semisimplification $\overline{\Ver_{p^n}}$.
\end{remark} 

Let us give a complete description of the category $\overline{\Ver_{p^2}}$ for $p>2$ (note that we should not expect a sensible explicit description of the category $\overline{\Ver_{p^n}}$ for $n>2$ since the category $\Ver_{p^n}$ has wild representation type in this case). Let $\Vec_{\BZ_{2p-2}}^{\rm odd}$ be the pointed tensor category generated by an odd invertible object $T$ of order $2p-2$. Let $\delta \in \Ver_p$ be the nontrivial invertible object. The object $\delta \boxtimes \delta \boxtimes T^{\ot p-1}\in \Ver_p\boxtimes \Ver_p\boxtimes \Vec_{\BZ_{2p-2}}^{\rm odd}$ is even invertible of order 2. 
Hence the de-equivariantization $(\Ver_p\boxtimes \Ver_p\boxtimes \Vec_{\BZ_{2p-2}}^{\rm odd})_{\langle \delta \boxtimes \delta \boxtimes T^{\ot p-1}\rangle}$ is well defined.

\begin{proposition} We have $\overline{\Ver_{p^2}}\cong (\Ver_p\boxtimes \Ver_p\boxtimes \Vec_{\BZ_{2p-2}}^{\rm odd})_{\langle \delta \boxtimes \delta \boxtimes T^{\ot p-1}\rangle}$.
\end{proposition}

\begin{proof} We already know that $\overline{\Ver_{p^2}}\supset \Ver_p\boxtimes \Ver_p$, see
Remark \ref{simples for p squared}. We also have an invertible object $T\in \overline{\Ver_{p^2}}$
which is the Heller shift of the unit object. 
It is clear that $T$ is odd and Proposition \ref{extp2} implies that the order of $T$ is divisible
by $2p-2$. Thus $T^{\ot p-1}$ is the lowest positive tensor power of $T$ which might be in
$\Ver_p\boxtimes \Ver_p$. It follows that the subcategory of $\overline{\Ver_{p^2}}$ generated
by the simple objects of $\Ver_{p^2}$ and $T$ contains at least $(p-1)^3$ simple objects. 

On the other hand, it follows from Proposition \ref{blocksp2} that the category $\overline{\Ver_{p^2}}$ has
precisely $(p-1)\cdot (p^2-p)+p-1=(p-1)^3+p^2-p$ indecomposable objects. Recall that we have $p^2-p$ indecomposable projective objects,
so we get at most $(p-1)^3$ non-negligible objects. Comparing this with the previous paragraph
we conclude that the category $\overline{\Ver_{p^2}}$ is generated
by the simple objects of $\Ver_{p^2}$ and $T$, and $T^{\ot p-1}$ is indeed contained in $\Ver_p\boxtimes \Ver_p$.
Since $T^{\ot p-1}$ is even and non-trivial, we see that $T^{\ot p-1}\simeq \delta \boxtimes \delta$ and
the order of $T$ is precisely $2p-2$. Thus we have a surjective symmetric tensor functor
$\Ver_p\boxtimes \Ver_p\boxtimes \Vec_{\BZ_{2p-2}}^{\rm odd}\to \overline{\Ver_{p^2}}$ such
that its right adjoint $I$ satisfies $I(\one)=\one \oplus \delta \boxtimes \delta \boxtimes T^{\ot p-1}$.
This implies the result.
\end{proof}

Note that this proof shows that all negligible objects of $\Ver_{p^2}$ are
projective.

\subsection{The category $\Ver^+_{p^3}$ in the case $p=3$} 

We now describe the category $\Ver^+_{p^3}$ in the case $p=3$.
The relevant even indexed tilting modules
have the following Loewy series.\par
{\tiny
\[ \begin{array}{ccccc} 
&&&& \\
T_8=[8] & 
T_{10}=\begin{array}{c} [6]\\{}[4]\\{}[10]\\{}[4]\\{}[6]\end{array} &
T_{12}=\begin{array}{c} [4] \\{}[10]\ [0]\ [6]\\{}[4]\ [12]\ [4]\\{}[6]\ [0]\ [10]\\{}[4]\end{array}  &
T_{14}=\begin{array}{c} [2]\\{}[14]\\{}[2]\end{array} &
T_{16}=\begin{array}{c}[0]\\{}[4]\ [12]\\{}[0]\ [10]\ [16]\ [0]\\{}[12]\ [4]\\{}[0]\end{array} \\
&&&&\\
& T_{18}=\begin{array}{c}[16]\\{}[12]\\{}[0]\ [18]\\{}[12]\\{}[16]\end{array} &
T_{20}=\begin{array}{c}[14]\\{}[2]\ [20]\\{}[14]\end{array} &
T_{22}=\begin{array}{c}[12]\\{}[0]\ [18]\ [10]\ [16] \\{}[4]\ [12]\ [22]\ [12] \\{}[0]\ [18]\ [10]\ [16] \\{}[12]\end{array} &
T_{24}=\begin{array}{c}[10]\\{}[4]\ [12]\ [22]\\{}[0]\ [6]\ [10]\ [18]\ [24]\ [10] \\{}[4]\ [12]\ [22]\\{}[10] \end{array} 
\end{array} \]
}
Again we label the simples for $A$ by the top composition factors of the tilting modules. 
So here is a table of the correspondence:\par
{\tiny
\[ \begin{array}{|c|c|c|c|c|c|c|c|c|} \hline
L_0 & L_2 & L_4 & L_6 & L_8 & L_{10} & L_{12} & L_{14} & L_{16} \\ \hline
T_{16} & T_{14} & T_{12} & T_{10} & T_8 & T_{24} & T_{22} & T_{20} & T_{18} \\ \hline
\end{array} \]
\par}
These fall into three blocks. One is a semisimple block containing the projective simple $L_8$,
one is a block of finite representation type containing $L_2$ and $L_{14}$, and the remaining
six simples, $L_0$, $L_4$, $L_6$, $L_{10}$, $L_{12}$, $L_{16}$ 
lie in a single block of wild representation type. The Cartan matrix is\par
{\tiny
\[ \begin{array}{|c|cccccc|cc|c|} \hline
&L_0 & L_4 & L_6 & L_{10} & L_{12}&L_{16}&L_2&L_{14}&L_8 \\ \hline
L_0   &4&2&0&1&2&1&\multicolumn{3}{c}{} \\
L_4   &2&4&2&2&1&0&\multicolumn{3}{c}{} \\
L_6   &0&2&2&1&0&0&\multicolumn{3}{c}{} \\
L_{10}&1&2&1&4&2&0&\multicolumn{3}{c}{} \\
L_{12}&2&1&0&2&4&2&\multicolumn{3}{c}{} \\
L_{16}&1&0&0&0&2&2&\multicolumn{3}{c}{} \\ \cline{1-9}
L_2&&&&&&&2&1\\
L_{14}&&&&&&&1&2\\ \cline{1-1}\cline{8-10}
L_8&\multicolumn{8}{c|}{}&1 \\ \cline{1-1}\cline{10-10}
\end{array} \]
}
This can be easily obtained from the results of Section \ref{s4}. 

The structure of the projective $A$-modules (obtained using a computer calculation) is as follows
(where the diagrams for the non-uniserial modules are in the sense of
Alperin~\cite{Alperin:1980b} or Benson and Carlson~\cite{Benson/Carlson:1987a}):\par
{\tiny
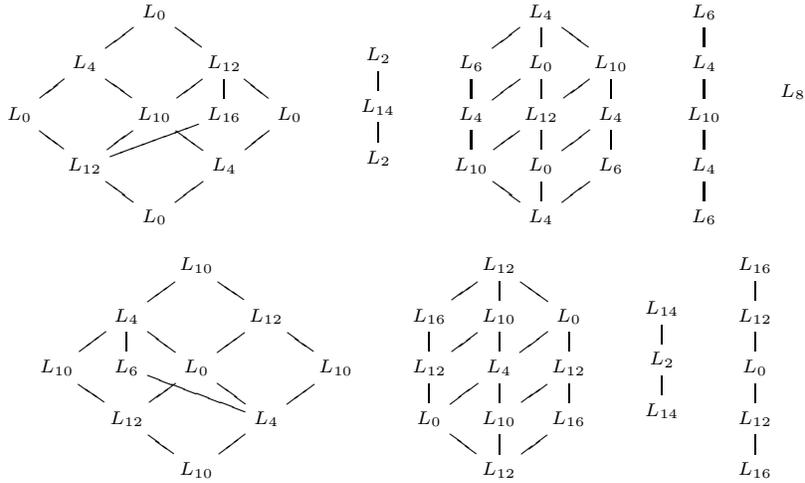
\[ \xymatrix@=2.5mm{ && L_0 \ar@{-}[dl]\ar@{-}[dr] \\ 
&L_4\ar@{-}[dl]\ar@{-}[dr]&& L_{12}\ar@{-}[dl]\ar@{-}[d]\ar@{-}[dr] \\ 
L_0\ar@{-}[dr]&& L_{10}\ar@{-}[dl]\ar@{-}[dr]& L_{16}\ar@{-}[dll]& L_0\ar@{-}[dl] \\
&L_{12}\ar@{-}[dr]&& L_{4}\ar@{-}[dl]\\ &&L_0} 
\qquad
\xymatrix@=2.5mm{\\L_2\ar@{-}[d] \\ L_{14} \ar@{-}[d] \\ L_2}
\qquad
\xymatrix@=2.5mm{&L_4\ar@{-}[dl]\ar@{-}[d]\ar@{-}[dr] \\ 
L_{6}\ar@{-}[d] &L_0\ar@{-}[d]\ar@{-}[dl]&L_{10}\ar@{-}[d]\ar@{-}[dl] \\
L_4\ar@{-}[d] & L_{12}\ar@{-}[d]\ar@{-}[dl] & L_4\ar@{-}[d]\ar@{-}[dl] \\ 
L_{10}\ar@{-}[dr] & L_0\ar@{-}[d] & L_{6}\ar@{-}[dl] \\ 
& L_4}
\qquad
\xymatrix@=2.5mm{L_6\ar@{-}[d] \\ L_4\ar@{-}[d] \\ 
L_{10}\ar@{-}[d] \\ L_4\ar@{-}[d] \\ L_6} 
\qquad 
\xymatrix@=2.5mm{\\ \\ L_8}
\]
\[ \xymatrix@=2.5mm{&&L_{10}\ar@{-}[dl]\ar@{-}[dr] \\ 
&L_4\ar@{-}[dl]\ar@{-}[d]\ar@{-}[dr] && L_{12}\ar@{-}[dl]\ar@{-}[dr] \\ 
L_{10}\ar@{-}[dr] & L_6\ar@{-}[drr] & L_0\ar@{-}[dl]\ar@{-}[dr] && L_{10}\ar@{-}[dl] \\ 
&L_{12}\ar@{-}[dr] && L_{4}\ar@{-}[dl] \\ &&L_{10}} 
\qquad
\xymatrix@=2.5mm{&L_{12}\ar@{-}[dl]\ar@{-}[d]\ar@{-}[dr]\\
L_{16}\ar@{-}[d] & L_{10}\ar@{-}[d]\ar@{-}[dl] & L_{0}\ar@{-}[d]\ar@{-}[dl] \\
L_{12}\ar@{-}[d] & L_{4}\ar@{-}[d]\ar@{-}[dl] & L_{12}\ar@{-}[d]\ar@{-}[dl] \\
L_0\ar@{-}[dr] & L_{10}\ar@{-}[d] & L_{16}\ar@{-}[dl] \\
&L_{12}}
\qquad
\xymatrix@=2.5mm{\\ L_{14}\ar@{-}[d] \\ L_2 \ar@{-}[d] \\ L_{14}}
\qquad
\xymatrix@=2.5mm{L_{16} \ar@{-}[d] \\ L_{12} \ar@{-}[d] \\ 
L_0 \ar@{-}[d] \\ L_{12} \ar@{-}[d] \\ L_{16}}
\]
\par}
For the wild block, the basic algebra is given by the following quiver and relations:
{\tiny
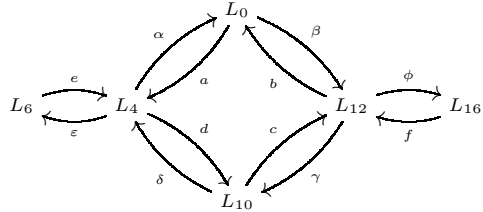
\[ \xymatrix{&&L_{0}\ar@/^/[dl]^a\ar@/^/[dr]^\beta \\
L_{6}\ar@/^/[r]^e&L_{4}\ar@/^/[l]^\varepsilon\ar@/^/[ur]^{\alpha}\ar@/^/[dr]^d&&
L_{12}\ar@/^/[ul]^b\ar@/^/[dl]^\gamma\ar@/^/[r]^\phi&L_{16}\ar@/^/[l]^f\\
&&L_{10}\ar@/^/[ul]^\delta\ar@/^/[ur]^c} \]
}
Relations: 
$\alpha e=0$, 
$\gamma f=0$,
$\varepsilon e=0$, 
$\varepsilon a=0$,
$\phi c=0$,
$\phi f=0$,
$d\delta d=0$,
$\delta d \delta=0$,
$b\beta b=0$,
$\beta b \beta=0$,
$\alpha\delta = bc$,
$\beta\alpha=cd$,
$\gamma\beta=da$,
$\delta\gamma = ab$,
$e\varepsilon=a\alpha$,
$f\phi=(1 + \beta b)c\gamma$.

Note that it follows from this presentation that 
$\beta bf\phi=\beta bc\gamma + \beta b\beta bc\gamma= \beta bc\gamma$,
so $f\phi=c\gamma+\beta b f\phi$.  So the last relation can be replaced by
$c\gamma=(1-\beta b)f\phi$. Moreover,
\[ \beta bc\gamma = \beta \alpha\delta\gamma = 
cdab=c\gamma\beta b \]
so $c\gamma$ and hence also $f\phi$ commute with $\beta b$.

Table~\ref{table:tensors} describes the tensor products of the simples in $\Ver^+_{3^3}$. 

{\tiny
\begin{table}[t]
\[ \renewcommand{\arraystretch}{1.2}
\setlength{\arraycolsep}{0.5mm}
\begin{array}{|c|c|c|c|c|c|c|c|c|} \hline
L_0 & L_2 & L_4 & L_6 & L_8 & L_{10} & L_{12} & L_{14} & L_{16} 
\\ \hline
L_2 & L_2 \oplus \begin{array}{c} L_0\\L_4\\L_0\end{array} & 
\begin{array}{c} L_4 \\ L_0\ L_6\\ L_4 \end{array} & 
L_8 & P_6 \oplus L_8 & 
\begin{array}{c} L_{10} \\ L_{12} \\ L_{10}\end{array} &
L_{14} &L_{14} \oplus 
\begin{array}{c} L_{12} \\ L_{10}\ L_{16} \\ L_{12} \end{array}  & P_{16} 
\\ \hline
L_4 & \begin{array}{c}L_4 \\ L_0\ L_6 \\ L_4 \end{array} &
\begin{array}{c} L_0 \oplus L_2 \\ {} \oplus L_6 \oplus L_8\end{array} & 
\begin{array}{c} L_4 \\ L_{10} \\ L_4\end{array} & P_4 &
L_{12} \oplus L_{14} & L_{10} \oplus L_{16} &
P_{16} \oplus\begin{array}{c} L_{10} \\ L_{12} \\ L_{10} \end{array} &
P_{14} \oplus \begin{array}{c}L_{12}\\L_0\\L_{12}\end{array}
\\ \hline
L_6 & L_8 & \begin{array}{c} L_4 \\ L_{10} \\ L_4 \end{array} &
L_6 \oplus \begin{array}{c} L_0 \\ L_{12} \\ L_0 \end{array} &
P_2 \oplus L_8 & L_{16} &
\begin{array}{c}L_{12}\\L_0\\L_{12}\end{array} & P_{14} &
L_{16} \oplus \begin{array}{c}L_{10}\\L_4\\L_{10}\end{array}
\\ \hline
L_8 & P_6 \oplus L_8 & P_4 & P_2 \oplus L_8 & 
\begin{array}{c} P_0 \oplus P_2 \\ {} \oplus P_6 \oplus L_8\end{array} & 
P_{16} & P_{14} & P_{12} \oplus P_{14} & P_{10} \oplus P_{16}
\\ \hline
L_{10} & \begin{array}{c} L_{10} \\ L_{12} \\ L_{10} \end{array} &
L_{12} \oplus L_{14} & L_{16} & P_{16} & L_0\oplus L_2 & L_4 &
\begin{array}{c} L_4 \\ L_0\ L_6\\ L_4 \end{array} &
L_6 \oplus L_8
\\ \hline
L_{12} & L_{14} & L_{10} \oplus L_{16}  &
\begin{array}{c}L_{12}\\L_0\\L_{12}\end{array} &
P_{14} & L_4 & L_0 \oplus L_6 & L_2 \oplus L_8 &
\begin{array}{c} L_4 \\ L_{10} \\ L_4 \end{array}
\\ \hline
L_{14} & L_{14} \oplus 
\begin{array}{c} L_{12} \\ L_{10}\ L_{16} \\ L_{12} \end{array} &
P_{16} \oplus\begin{array}{c} L_{10} \\ L_{12} \\ L_{10} \end{array} &
P_{14} & P_{12} \oplus P_{14} &
\begin{array}{c} L_4 \\ L_0\ L_6\\ L_4 \end{array} &
L_2 \oplus L_8 &
L_2 \oplus L_8 \oplus P_6 \oplus 
\begin{array}{c} L_0\\L_4\\L_0\end{array} &
P_4
\\ \hline
L_{16} & P_{16} &
P_{14} \oplus \begin{array}{c}L_{12}\\L_0\\L_{12}\end{array} &
L_{16} \oplus \begin{array}{c}L_{10}\\L_4\\L_{10}\end{array} & 
P_{10} \oplus P_{16} & L_6 \oplus L_8 & 
\begin{array}{c} L_4 \\ L_{10} \\ L_4 \end{array} &
P_4 & L_{16} \oplus P_{10} \oplus P_{16} \oplus 
\begin{array}{c} L_{10} \\ L_4 \\ L_{10} \end{array}
\\ \hline
\end{array} \]
\caption{Tensor products in $\Ver^+_{3^3}$}
\label{table:tensors}
\end{table}
}

\end{document}